\documentclass[12pt,reqno]{amsart}

\usepackage[T1]{fontenc}
\usepackage[utf8]{inputenc}
\usepackage[english]{babel}

\usepackage{mathrsfs, csquotes, amsfonts,amssymb,
amsmath,mathtools,amssymb, enumerate, xcolor}

\usepackage{color}
\definecolor{bleu_sombre}{rgb}{0,0,0.6}
\definecolor{Bl}{rgb}{0,0,0.6}
\definecolor{rouge_sombre}{rgb}{0.8,0,0}
\definecolor{vert_sombre}{rgb}{0,0.6,0}
\usepackage[plainpages=false,colorlinks,linkcolor=bleu_sombre,
citecolor=rouge_sombre,urlcolor=vert_sombre,breaklinks]{hyperref}

\usepackage{pgf,tikz,pgfplots}
\definecolor{webblue}{rgb}{0.22,0.45,0.70}
\definecolor{webred}{rgb}{0.5, 0.09, 0.09}
\definecolor{zzttqq}{rgb}{0.6,0.2,0.}
\usetikzlibrary{arrows}

\renewcommand{\leq}{\leqslant}	
\renewcommand{\geq}{\geqslant}
\newcommand{\C}{\mathbb{C}}
\newcommand{\R}{\mathbb{R}}
\newcommand{\N}{\mathbb{N}}
\renewcommand{\Re}{\mathrm{Re}\,}
\newcommand{\dd}{\mathrm{d}}
\newcommand{\D}{\partial}

\def\bbb{{\mathcal B}}

\def\ooo{{\mathscr O}}

\def\R{\mathbb R} 
\def\C{\mathbb C}\def\N{\mathbb N}\def\Z{\mathbb Z}

\def\dd{\mathrm{d}}    
\def\norm#1{\left\Vert#1\right\Vert}

\def\set#1{\left\{#1\right\}}
\def\seq#1{\left<#1\right>}
\def\sep#1{\left(#1\right)}
\def\ie{{\it i.e. }}
\def\eps{\varepsilon}

\theoremstyle{plain}
\newtheorem{theorem}{{Theorem}}[section]
\newtheorem*{theorem*}{{Theorem}}
\newtheorem{proposition}[theorem]{Proposition}

\newtheorem*{proposition*}{Proposition}

\newtheorem*{corollary*}{Corollary}
\newtheorem{lemma}[theorem]{Lemma}

\newtheorem{assumption}[theorem]{Assumption}
\newtheorem*{lemma*}{Lemma}
\theoremstyle{definition}
\newtheorem{definition}[theorem]{Definition}
\newtheorem*{definition*}{Definition}
\theoremstyle{remark}
\newtheorem{remark}[theorem]{Remark}

\makeatletter

\@addtoreset{equation}{section}
\makeatother

\numberwithin{equation}{section}

\title[Semiclassical spectral gaps]{Semiclassical spectral gaps of the 3D Neumann Laplacian with constant magnetic field}

\author[F. H\'erau]{Fr\'ed\'eric H\'erau}
\address[F. H\'erau]{LMJL - UMR6629, Nantes Université, CNRS, 2 rue de la Houssini\`ere, BP 92208, F-44322 Nantes cedex 3, France}
\email{herau@univ-nantes.fr}

\author[N. Raymond]{Nicolas Raymond}
\address[N. Raymond]{Univ Angers, CNRS, LAREMA, SFR MATHSTIC, F-49000 Angers, France}
\email{nicolas.raymond@univ-angers.fr}

\begin{document}
\begin{abstract}
This article deals with the spectral analysis of the semiclassical Neumann magnetic Laplacian on a smooth bounded domain in dimension three. When the magnetic field is constant and in the semiclassical limit, we establish a four-term asymptotic expansion of the low-lying eigenvalues, involving a geometric quantity along the apparent contour of $\Omega$ in the direction of the field. In particular, we prove that they are simple.
\end{abstract}
\maketitle

\section{Introduction}
In this article, $\Omega\subset\R^3$ is a smooth, bounded, connected, and open set.
We consider
\[\mathscr{L}_h=(-ih\nabla-\mathbf{A})^2\]
with domain
\[\mathrm{Dom}(\mathscr{L}_h)=\{\psi\in H^1(\Omega) : (-ih\nabla-\mathbf{A})^2\psi\in L^2(\Omega)\,, \mathbf{n}\cdot(-ih\nabla-\mathbf{A})\psi=0\,\,\mbox{ on }\partial\Omega\}\,.\]
Here, $\mathbf{n}$ denotes the outward pointing normal to the boundary, and $\mathbf{A} : \overline{\Omega}\to\R^3$ is a smooth vector potential generating a constant magnetic field
\[\nabla\times\mathbf{A}=\mathbf{B}=\mathbf{e}_3\,.\]
The operator $\mathscr{L}_h$ is self-adjoint with compact resolvent. We denote by $(\lambda_n(\mathscr{L}_h))_{n\geq 1}$ the non-decreasing sequence of its eigenvalues and by $\mathscr{Q}_h$ its associated closed form on $H^1(\Omega)$.

The aim of this article is to describe the eigenvalues in the semiclassical limit $h\to 0$. The problem of estimating the eigenvalues of magnetic Schrödinger operators has a long story told in the books \cite{FH10} and \cite{Ray}. In the next section, we recall a fundamental result directly related to this article.

\subsection{On the Helffer-Morame's results}
In \cite[Theorem 4.4]{HM02}, Helffer and Morame have established that
\begin{equation}\label{eq.HM01}
\lambda_1(h)=\Theta_0 h+o(h)\,,
\end{equation}
with $\Theta_0\in(0,1)$ defined as
\begin{equation} \label{eq:bigthetazero}
\Theta_0=\min_{\xi\in\mathbb{R}}\mu_1^{\mathrm{dG}}(\xi)\,,
\end{equation}
where $\mu_1^{\mathrm{dG}}(\xi)$ is the smallest eigenvalue of the de Gennes operator with parameter $\xi$.

This operator is defined for all $\xi\in\mathbb{R}$ as the Neumann realization of the differential operator acting on $L^2(\mathbb{R}_+)$ as
\[\mathfrak{h}_\xi=-\partial_t^2+(\xi-t)^2\,.\]
It is known that $\mu_1^{\mathrm{dG}}$ is real analytic and that is has a unique minimum, which is non-degenerate, attained at $\xi_0>0$ and not attained at infinity (see \cite[Section 3.2]{FH10}, \cite[Section 2.4]{Ray} or the original reference \cite{DH93}). We let
\[\alpha_0 =\frac{\left(\mu_1^{\mathrm{dG}}\right)''(\xi_0)}{2}>0\,.\]
In \eqref{eq.HM01}, we see that the main term $\Theta_0 h$ does not involve the shape of $\Omega$. In fact, Helffer and Morame also investigated the effect of the curvature of the boundary on the spectral asymptotics under the following assumption.

\begin{assumption}\label{hyp.genericcancellation}
	The subset $\Gamma:=\{x\in\partial\Omega : \mathbf{B}(x)\cdot\mathbf{n}(x)=0\}$ is a smooth closed submanifold of dimension one of $\partial\Omega$. Moreover, the function $\partial\Omega\ni x\mapsto \mathbf{B}(x)\cdot\mathbf{n}(x)=n_3(x)$ vanishes linearly on $\Gamma$.
\end{assumption}
It will be convenient to parametrize $\Gamma$ by arc-length thanks to $\gamma :[0,2L)\ni s\to \gamma(s)\in\Gamma$, where $L$ is the half-length of $\Gamma$, and to define adapated coordinates near $\Gamma$ thanks to the geodesic distance (inside $\partial\Omega$) to $\Gamma$ denoted by $r$ ($r$ is chosen so that $\mathbf{n}(\gamma(r,s))\cdot\mathbf{e}_3>0$ when $r<0$ for orientation reasons, see Figure \ref{fig.geo}). In terms of the variable $r$, we have $(\partial_r n_3)_{x}<0$ for all $x\in\Gamma$.

\begin{remark}
When $\Omega$ is strictly convex, Assumption \ref{hyp.genericcancellation} is satisfied. 	
\end{remark}

\begin{figure}[ht!]
\includegraphics[width=11cm]{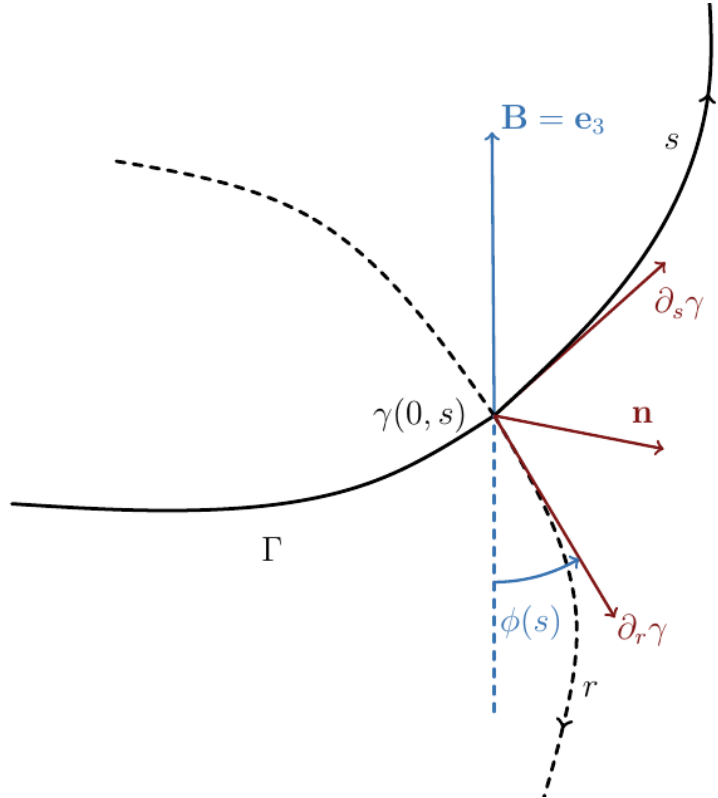}
	\caption{Local geometry near $\Gamma$; the dashed curve is a geodesic inside $\partial\Omega$}
\label{fig.geo}
\end{figure}
It will also be important to describe the relative variation of the magnetic field in the tangent plane to the boundary attached the points of $\Gamma$. Note that, due to our choice of adapted coordinates, we will see that $|\partial_s\gamma(0,s)|=|\partial_r\gamma(r,s)|=1$.
\begin{definition}\label{def.phi}
	We consider the function $s\mapsto \phi(s)$ such that
	\[\cos\phi(s)=\langle \mathbf{e}_3, \partial_r\gamma(0,s)\rangle\,,\quad \sin\phi(s)=\left\langle \mathbf{e}_3, \partial_s\gamma(0,s)\right\rangle \,.\]
	Moreover, we let
	\[K(s)=  \alpha_0^{\frac13}\beta(s)^{\frac23}
    E(s)^{\frac13}\, \]
    with
    \[E(s) = \alpha_0\sin^2\phi(s)+\cos^2\phi(s) \qquad \textrm{ and } \qquad \beta(s)=-\partial_r n_3(\gamma(s))\,.\]
\end{definition}
\begin{remark}
Note that $\phi\equiv 0$ corresponds to the case when $\Gamma$ lies in a plane orthogonal to $\mathbf{e}_3$. For instance, it happens in the case of an ellipsoid. In this case, $K$ is minimal where $\beta$, the transverse curvature to $\Gamma$, is minimal.
\end{remark}

It turns out that the influence of the shape of the boundary appears in the second order term of the asymptotic expansion of $\lambda_1(h)$. This term involves the Montgomery operator, acting on $L^2(\mathbb{R})$, defined for all $\xi\in\mathbb{R}$ by
\[\mathfrak{h}_\xi^{[2]}=-\partial_t^2+\left(\xi-\frac{t^2}{2}\right)^2\,.\]
The lowest eigenvalue of $\mathfrak{h}_\xi^{[2]}$ is denoted by $\mu_1^{[2]}(\xi)$. It is known that $\mu_1^{[2]}$ has a unique minimum, which is non-degenerate, not attained at infinity (see \cite{Helffer10} or even the generalization \cite{FS15}). We let
\[\Theta_0^{[2]}=\min_{\xi\in\mathbb{R}}\mu^{[2]}_1(\xi)\,.\]
Helffer and Morame obtained the following remarkable theorem in \cite[Theorem 1.2]{HM04}.
\begin{theorem}[\cite{HM04}]\label{thm.HM}
Under Assumption \ref{hyp.genericcancellation}, we have
\[\lambda_1(h)=\Theta_0 h+ K_{\min }\Theta_0^{[2]}h^{\frac43}+o(h^{\frac43})\,.\]
\end{theorem}
Fournais and Persson-Sundqvist have also established a four-term asymptotic expansion when $\Omega$ is a ball, see \cite[Theorem 1.1]{FP11}. In their situation, the crucial term, which determines the monotonicity of the eigenvalue with respect to $h$, is of order $h^2$.

\subsection{Main result: asymptotic simplicity}
Let us now describe our main result. It requires the following genericity assumption.

\begin{assumption}\label{hyp.generic}
The function $K$ has a unique minimum, which is non-degenerate.	
\end{assumption}
We let
\[ b^\Gamma(s,\sigma)=K(s)\mu^{[2]}_1\left(\frac{\alpha_0^{\frac13}\sigma}
{E(s)^{\frac23}\beta(s)^{\frac13}}\right)\,, \]
%,\quad A(s)=\delta_0\sin^2\phi(s)+\cos^2\phi(s)\,,\]
and we notice that $b^\Gamma$ has a unique minimum, which is non-degenerate. It is attained at a point $(s_{\min},\sigma_{\min})$, and we let $b^\Gamma_{\min}=b^\Gamma(s_{\min},\sigma_{\min})$. Note that
$b^\Gamma_{\min}=K_{\min}\Theta_0^{[2]}$ and that $s_{\min}$ is the point where the minimum of $K$ is reached and
\[\sigma_{\min}=\frac{E^{\frac23}(s_{\min})\beta^{\frac13}
(s_{\min})}{\alpha_0^{\frac13}}\xi_0^{[2]}\,.\]
The main result of this paper is the following theorem, extending to the order $o(h^{\frac53})$ the result of Helffer and Morame in \cite{HM04} and revealing the spectral gap.

\begin{theorem}\label{thm.main}
Under Assumptions \ref{hyp.genericcancellation} and \ref{hyp.generic}, there exist $d_0,d_1\in\mathbb{R}$ such that the following holds. Let $n\geq 1$. We have
\[\lambda_n(h)=\Theta_0h+  K_{\min}\Theta_0^{[2]}  h^{\frac43}+d_0h^{\frac32}+\left(d_1+\left(n-\frac12\right)
\sqrt{\det\mathrm{Hess}_{(s_{\min},\sigma_{\min})}b^\Gamma}\right)h^{\frac53}
+o(h^{\frac53})\,.\]
\end{theorem}

\bigskip
Let us make some comments and remarks on this theorem.
\begin{remark}
\begin{enumerate}[i)]~
\item Theorem \ref{thm.main} establishes that the low-lying eigenvalues are simple when $h$ is small enough, and that the spectral gap is of order $h^{\frac53}$. We can also notice the presence of the term of order $h^{\frac32}$.
\item Not only Theorem \ref{thm.main} is a more accurate description of the spectrum than the one recalled in Theorem \ref{thm.HM}, but the spirit of its proof is also of a different nature.  As we explain in Section \ref{sec.organization} below, our method strongly relies on microlocal analysis and projects a new light on the Helffer-Morame's results.
\item Actually, our strategy does not only provide us with a four-term asymptotics and it can also be used to establish a full asymptotic expansion in powers of $h^{\frac13}$. We can even see that, up to a local change of gauge near $s_{\min}$, the $n$-th normalized eigenfunction looks like, in the $(r,s,t)$ coordinates,
\[u_{\xi_0}(h^{-\frac12}t)v(h^{-\frac13}r)w_n(h^{-\frac{1}{6}}(s-s_{\min}))\,,\]
where
\begin{enumerate}[---]
\item $u_{\xi_0}$ is the positive normalized eigenfunction of $\mathfrak{h}_{\xi_0}$,
\item $v(r)=u_{\xi_0^{[2]}}\left(\frac{\delta_0^{\frac16}\beta^{\frac13}(s_{\min})}{A^{\frac13}(s_{\min})}r\right)$, where $u_{\xi_0^{[2]}}$ is the first normalized positive eigenfunction of the Montgomery operator $\mathfrak{h}^{[2]}_{\xi_0^{[2]}}$,
\item $w_n(s)=H_n\left(\left[\frac{K''(s_{\min})}{K(s_{\min})(\mu_1^{[2]})''(\xi_0^{[2]})}\right]^{\frac14}s\right)$, where $H_n$ is the $n$-th Hermite function.
\end{enumerate}
\item The analogous of Theorem \ref{thm.main} in two dimensions has been established by Fournais and Helffer in \cite[Theorem 1.1]{FH06}, and it was a strong extension of \cite[Section 10]{HM01} in the case of constant magnetic fields.
\end{enumerate}
\end{remark}

\subsection{Organization and strategy}\label{sec.organization}
Section \ref{sec.2} is devoted to recall why  the eigenfunctions associated with the low-lying eigenvalues are localized near $\Gamma$, see Propositions \ref{prop.loct} and \ref{prop.roughlocGamma}. In Section \ref{sec.rst}, we write the magnetic Laplacian in appropriate tubular coordinates near $\Gamma$. Once this is done, we prove in Section \ref{sec.optimal-r} that the eigenfunctions are exponentially localized near $\Gamma$ at a scale of order $h^{\frac13}$. This fact is a novelty with respect to the results in \cite{HM04} where such a decay of the eigenfunctions does not appear. For that purpose, we have to combine several arguments. We start with a partition of the unity and Taylor expansions of the metric and the magnetic potential (up to a local change of gauge) near $\Gamma$. This reveals the role of a model operator given by the quadratic form in \eqref{eq.modeloperator}. This operator cannot be analyzed anymore only with our space partition of unity. Up to a convenient rotation involving the geometric angles of the magnetic field with respect to the boundary, we get a new quadratic form \eqref{eq.Q}, which can be analyzed thanks to a partial Fourier transform exhibiting the role of the de Gennes operator. A microlocal partition of the unity is then used to relate \eqref{eq.Q} with a differential operator in two dimensions. This operator belongs to the class of Montgomery operators, which was studied in \cite{BHR16}. At this stage, a simple commutator estimate is done and the localization with respect to $r$ --- the distance to $\Gamma$ --- is established. This new and optimal localization allows to expand the vector potential and to neglect various terms smaller than the announced order of magnitude of the spectral gap, see Proposition \ref{prop.checkLm}. The operator, called $\check{\mathscr{L}}^{\mathrm{m}}_h$, is then rescaled in $t$ (at the scale $h^{\frac12}$) and in $r$ (at the scale $h^{\frac13}$). This natural rescaling gives a new operator $\breve{\mathscr{L}}^{\mathrm{m}}_h$ that invites us to consider the effective semiclassical parameter $\hbar=h^{\frac16}$. It also turns out that it is convenient to introduce the auxiliary parameter $\mu=\hbar^2$ and to transform, for instance, the formal subprincipal term $\hbar\beta(s)\frac{r^2}{2}$ into a principal-type term $\mu^{\frac12}\beta(s)\frac{r^2}{2}$. In \eqref{eq.symbol0}, $\breve{\mathscr{L}}^{\mathrm{m}}_h$ is seen as a pseudo-differential operator with operator symbol, whose principal symbol is given in \eqref{eq.n0good}. 

 There, we describe the subprincipal term (involving the $\hbar^2$ terms) and notice a remarkable algebraic/geometric cancellation (Lemma \ref{lem.annulr}), which will play an important role. Unfortunately, but not surprisingly, our pseudo-differential operator does not belong to a reasonable class to perform a microlocal dimensional reduction. For this reason, we establish a rough microlocalization result for the eigenfunctions by explaining why they are bounded in $\rho$ and $\mu\sigma$ ($\rho$ and $\sigma$ being the dual variables of $r$ and $s$), see Section \ref{sec.roughmicro}. By suitably truncating  quantities in the phase space, we are then led to study only operators with bounded $\mathscr{C}^\infty$ symbols as well as their derivatives.

 Section \ref{sec.parametrix} is where the Grushin-Sjöstrand machinery comes into play. Our strategy is inspired by our study of the magnetic tunnelling effect in \cite{BHR21}. There we used an adapation of a microlocal dimensional reduction method originally developped by Martinez \cite{M07} (and inspired by the Sjöstrand's works) and improved by Keraval in \cite{Keraval}. In Section \ref{sec.parametrix}, we use this method two times. First, we introduce the Grushin matrix of the operator symbol \eqref{eq.Grushinmatrix}. We construct an approximate inverse/parametrix \eqref{eq.parametrix1} and describe its "Schur complement" in Proposition \ref{prop.qpm}. This reveals a first effective pseudo-differential operator $\mathrm{Op}^W_{\hbar} a^{\mathrm{eff}}_\hbar$ given by \eqref{eq.aeff}, which describes the spectrum of the initial operator, see Proposition \ref{prop.specreducaeff}. The rest of the paper is devoted to the spectral analysis of $\mathrm{Op}^W_{\hbar} a^{\mathrm{eff}}_\hbar$. The principal symbol looks essentially like
\[\mu_1(r,s,\rho,\sigma)=\mu^{[1]}_1(\xi_0+\tilde\rho\sin\phi-\tilde p\cos\phi)+(\tilde\rho\cos\phi+ \tilde p\sin\phi)^2\,,\]
with
\[\tilde\rho=\Xi_1(\rho)\,,\quad \tilde p=\Xi_2(\mu\sigma)-\mu^{\frac12}\beta(s)\frac{r^2}{2}\,,\]
for suitable localization functions $\Xi_1$ and $\Xi_2$. When considered as a function of $\tilde \rho$ and $\tilde p$, $\mu_1$ has a unique minimim at $(\tilde\rho,\tilde p)=(0,0)$, which is non-degenerate. This induces a microlocalization of the eigenfunctions near $(\tilde\rho,\tilde p)=(0,0)$ (see Proposition \ref{prop.rhop}) and allows to expand the symbol at a convient order $N$, giving \eqref{eq.aeffN}. The aim is then to analyse the spectrum of $\mathrm{Op}^W_{\hbar}a_{\hbar,N}^{\mathrm{eff}}$. In Lemma \ref{lem.microhsigma}, we prove that the eigenfunctions are microlocalized. More precisely, we show that they are localized where $\hbar\sigma$ is bounded. This allows us to replace $\Xi_2(\hbar^2\sigma)$ by $\hbar\Xi_2(\hbar\sigma)$. We are then reduced to the spectral analysis of an $\hbar$-pseudo-differential operator whose principal symbol is \[\mathfrak{a}_0=\frac{1}{2}(\mu_1^{\mathrm{dG}})''(\xi_0)(\rho\sin\phi- \hbar\hat p\cos\phi)^2+(\rho\cos\phi+ \hbar\hat p\sin\phi)^2\,,\quad \hat p=\Xi_2(\hbar\sigma)-\beta(s)\frac{r^2}{2}\,.\]
We can see this new operator as a $1$-pseudo-diffferential operator with respect to $r$, whose principal symbol is
\[\mathfrak{b}_0(r,s,\rho,\sigma)=\frac{1}{2}(\mu_1^{\mathrm{dG}})''(\xi_0)(\rho\sin\phi- \hat p\cos\phi)^2+(\rho\cos\phi+ \hat p\sin\phi)^2\,,\]
where we recall that $\mu_1^{\mathrm{dG}}$ is the lowest eigenvalue of the de Gennes operator. This suggests to introduce the final semiclassical parameter $\varepsilon=\hbar^2$, and to study an $\varepsilon$-pseudo-differential operator (with a symbol expanded in powers of $\varepsilon^{\frac12}$) whose principal symbol is
\[\mathfrak{c}_0=\frac{1}{2}(\mu^{\mathrm{dG}})''(\xi_0)(\rho\sin\phi- \check p\cos\phi)^2+(\rho\cos\phi+ \check p\sin\phi)^2\,,\quad \check p=\Xi_2(\sigma)-\beta(s)\frac{r^2}{2}\,.\]
For fixed $(s,\sigma)$, this is essentially the symbol of a Montgomery operator, up to a convenient rescaling, as explained in Section \ref{sec.finalgrushin}. This principal operator symbol is the first stone of a second Grushin reduction. We reduce the spectral analysis to the one of an $\varepsilon$-pseudo-differential operator in one dimension. Its principal symbol is
\[b^\Gamma(s,\sigma)=K(s)\mu^{[2]}_1\left(\frac{\alpha_0^{\frac13}
\Xi_2(\sigma)}{E(s)^{\frac23}
\beta(s)^{\frac13}}\right)\,.\]
Section \ref{sec.final} is devoted to the analysis of this final operator.

\section{A first reduction to a neighborhood of $\Gamma$} \label{sec.2}

This section is devoted to recall a result of rough localization of the eigenfunctions near the set $\Gamma$. We give the precise statement and some consequences just below and postpone the proof to Section \ref{subsec.roughloc}.
In the following, we consider
\[\Omega_\delta=\{x\in\Omega : \mathrm{dist}(x,\partial\Omega)<\delta\}\,,\]
for $\delta>0$, a suitable neighborhood of $\D \Omega$.

\begin{proposition}\label{prop.roughlocGamma}
	There exist $\delta_0>0$, $C>0$, $\alpha>0$ and $h_0>0$ such that, for all $h\in(0,h_0)$, $\delta\in(0,\delta_0)$ and all eigenpairs $(\lambda,\psi)$ with $\lambda\leq\Theta_0 h+Ch^{\frac43}$, we have
	\[\int_{\Omega_\delta}e^{2\alpha \mathrm{dist}(p(x),\Gamma)/h^{\frac14}}|\psi(x)|^2\dd x\leq C\|\psi\|^2\,.\]
	%where
%	\[\Omega_\delta=\{x\in\Omega : \mathrm{dist}(x,\partial\Omega)<\delta\}\,.\]
\end{proposition}

Proposition \ref{prop.roughlocGamma} has the following straightforward consequence. Let us consider the tubular neighborhood of $\Gamma$ given by
\[\Omega_{\delta_1,\delta_2}=\Omega_{\delta_1}\cap\Gamma_{\delta_2}\,,\quad \Gamma_{\delta_2}=\{x\in\Omega : \mathrm{dist}(x,\Gamma)<\delta_2\}\,,\]
with $\delta_1=h^{\frac12-\eta}$ and $\delta_2=h^{\frac14-\eta}$ for some
$\eta\in\left(0,\frac14\right)$. We denote by $\mathscr{L}_{h,\delta_1,\delta_2}$ the realization of $(-ih\nabla-\mathbf{A})^2$ on $L^2(\Omega_{\delta_1,\delta_2})$ with Neumann condition on $\partial\Omega\cap\Omega_{\delta_1,\delta_2}$ and Dirichlet condition on $\Omega_{\delta_1,\delta_2}\setminus\partial\Omega$.

\begin{proposition}\label{prop.1}
	For all $n\geq 1$, we have
	\[\lambda_n(\mathscr{L}_{h,\delta_1,\delta_2})-Ce^{-ch^{-\eta}}\leq\lambda_n(\mathscr{L}_h)\leq\lambda_n(\mathscr{L}_{h,\delta_1,\delta_2})\,.\]
\end{proposition}

\begin{remark}
Up to a slight modification of $\Omega_{\delta_1,\delta_2}$, we may assume that $\Omega_{\delta_1,\delta_2}$ is smooth to avoid regularity problems near the "corners" of $\Omega_{\delta_1,\delta_2}$.
\end{remark}

 \subsection{First localization near the boundary} \label{subsec.roughloc}
 The following three results are well-known, see \cite[Theorems 2.1 \& 3.1]{HM96}.
 \begin{lemma}\label{lem.min-interior}
 	For all $\psi\in H^1_0(\Omega)$ and for all $h>0$,
 	\[\mathscr{Q}_h(\psi)\geq \int_{\Omega}h|\psi|^2\dd x\,.\]	
 \end{lemma}
Lemma \ref{lem.min-interior} and the fact that $\Theta_0<1$ imply that the eigenfunctions are exponentially localized near the boundary at a scale $h^{\frac12}$.

 \begin{proposition}\label{prop.loct}
 	Consider $\alpha\in(0,\sqrt{1-\Theta_0})$. There exist $C, h_0>0$ such that, for all $h\in(0,h_0)$, and all eigenpairs $(\lambda,\psi)$ suh that $\lambda\leq\Theta_0 h+o(h)$, we have
 	\[\int_{\Omega}e^{2\alpha\frac{\mathrm{d}(x,\partial\Omega)}{h^{\frac12}}}|\psi(x)|^2\dd x\leq C\|\psi\|^2\,.\]	
 \end{proposition}

 Let $\eta\in\left(0,\frac12\right)$, and consider $\delta=h^{\frac12-\eta}$. We consider $\mathscr{L}_{h,\delta}=(-ih\nabla-\mathbf{A})^2$ the operator with magnetic Neumann condition on $\partial\Omega$ and Dirichlet condition on $\partial\Omega_\delta\setminus\partial\Omega$. We have then the following direct consequence.
 \begin{proposition}
 	\[\lambda_n(\mathscr{L}_{h,\delta})-Ce^{-ch^{-\eta}}\leq \lambda_n(\mathscr{L}_{h,\delta}) \leq \lambda_n(\mathscr{L}_{h})\,.\]
 \end{proposition}
This implies that we can focus on the spectral analysis of $\mathscr{L}_{h,\delta}$.

\subsection{Rough localization near $\Gamma$}
We can now explain why the eigenfunctions are localized near $\Gamma$ as stated in Proposition \ref{prop.roughlocGamma}. This deserves some geometric preliminaries.
\begin{definition} \label{def.theta}
For $\delta$ small enough, and for all $x\in\Omega_\delta$, we can consider the projection $p(x)$ $\partial\Omega$ of $x$ defined by 
\[\min_{y\in\partial\Omega}\mathrm{dist}(x,y)=|x-p(x)|\,.\]
We also define $\theta(x)$ in $[-\frac\pi2,\frac\pi2]$ the oriented angle (see Figure \ref{fig:theta}) between $\mathbf{B}(p(x))$ and the tangent plane to $\D \Omega$ at $p(x)$ by
\[\mathbf{B}(p(x))\cdot\mathbf{n}(p(x))= \sin(\theta(x))\,.\]
\end{definition}

\begin{remark}
Note that $\theta(p(x)) = \theta(x)$ for all $x\in\Omega_\delta$ and that $\Gamma$ is then defined by $\set{x \in \D \Omega : \theta(x) = 0}$.
\end{remark}

\definecolor{zzttqq}{rgb}{0.6,0.2,0}
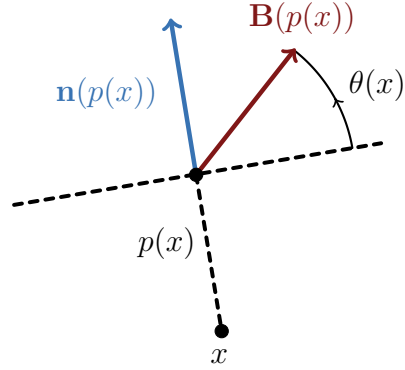
\begin{figure}[ht!]  \label{fig:theta}
\begin{tikzpicture}[line cap=round,line join=round,x=1cm,y=1cm]
\draw[line width=2pt, smooth,samples=1000,domain=0.1:0.43] plot[parametric] function{3.67*t**(3)+0*t**(2)+9.07*t-3.95,-10.77*t**(3)+0*t**(2)+7.73*t-1.05};
\draw[line width=2pt, smooth,samples=1000,domain=0.43:0.75] plot[parametric] function{-14.06*t**(3)+22.62*t**(2)-0.54*t-2.59,3.04*t**(3)-17.62*t**(2)+15.22*t-2.11};
\draw [->,line width=1.7pt, color=webblue] (0.19,1.41) -- (-0.15,3.49);
\draw [line width=1.5pt,dashed] (-2.23,1.01)-- (2.65,1.83);
\draw [shift={(0.19,1.41)},line width=0.8pt,->, samples=500]   plot[domain=0.17:0.5,variable=\t]({1*2.1076052761368764*cos(\t r)+0*2.1076052761368764*sin(\t r)},{0*2.1076052761368764*cos(\t r)+1*2.1076052761368764*sin(\t r)}) ;
\draw [shift={(0.19,1.41)},line width=0.8pt,-, samples=500]   plot[domain=0.5:0.9,variable=\t]({1*2.1076052761368764*cos(\t r)+0*2.1076052761368764*sin(\t r)},{0*2.1076052761368764*cos(\t r)+1*2.1076052761368764*sin(\t r)}) ;
\draw [line width=1.5pt, dashed] (0.19,1.41)-- (0.53,-0.67);
\draw [line width=1.7pt,->, color=webred] (0.19,1.41)-- (1.51,3.09);
\draw (2.6,2.6) node {$\theta(x)$};
\draw (1.6,3.5) node {\textcolor{webred}{$\mathbf{B}(p(x))$}};
\draw (-1,2.5) node {\textcolor{webblue}{$\mathbf{n}(p(x))$}};
\draw (0.5,-1) node {$x$};
\draw (-0.2,0.5) node {$p(x)$};
\draw [fill=black] (0.19,1.41) circle (2.5pt);
\draw [fill=black] (0.53,-0.67) circle (2.5pt);
\end{tikzpicture}
\caption{The angle $\theta(x)$}
\end{figure}
When we consider a small neighborhood of a point $x_0$ in $\partial\Omega$, we guess that the magnetic Laplacian acts as the magnetic Neumann Laplacian on a half-space with constant magnetic field $\mathbf{B}(x_0)$ of associated oriented angle $\theta(x_0)$. This model operator has been studied in \cite{HM02} and it is known that its spectrum is $[\mathfrak{s}(\theta(x_0)),+\infty[$, where $\mathfrak{s} : [-\frac\pi2,\frac\pi2] \to [\Theta_0,1]$ is a continuous and even function, strictly increasing (and also analytic) on $(0,\frac\pi2)$ and satisfying
\begin{equation}\label{eq.stheta}
\mathfrak{s}(\theta)\geq\Theta_0 +c\min(1,|\theta|)\,.
\end{equation}
for some $c>0$ and where we recall that $\Theta_0$ was defined in \eqref{eq:bigthetazero}.

\bigskip
In the "non-flat" case, the following proposition is then established in \cite[Proposition 9.1.2 \& Theorem 9.4.3]{FH10}. For completeness, we recall the proof in Appendix \ref{sec.app} (by using the tubular coordinates described in Section \ref{sec.coordnearboundary}). Here $\mathscr{Q}_{h,\delta}$ denotes the quadratic form associated with $\mathscr{L}_{h,\delta}$.
\begin{proposition}[\cite{FH10}] \label{prop.roughlowerbound}
There exist $C,h_0>$ such that for all $h\in(0,h_0)$ and all $\psi$ in the form domain of $\mathscr{Q}_{h,\delta}$, we have
	\[\mathscr{Q}_{h,\delta}(\psi)\geq \int_{\Omega_\delta}h\mathfrak{s}(\theta(x))|\psi(x)|^2\dd x-Ch^{\frac54}\|\psi\|^2\,.\]
\end{proposition}

An important consequence of Proposition \ref{prop.roughlowerbound} is the rough localization of the first eigenfunctions near $\Gamma$ stated in Proposition \ref{prop.roughlocGamma}.

\begin{proof}[Proof of Proposition \ref{prop.roughlocGamma}]
For $x \in \Omega_\delta$ with $\delta$ small, we let $\varphi(x)=\mathrm{dist}(p(x),\Gamma)$. Let us consider $(\lambda,\psi)$ an eigenpair of $\mathscr{L}_{h,\delta}$ with $\lambda\leq\Theta_0 h+Ch^{\frac43}$. A computation gives	\[\mathscr{Q}_{h,\delta}(e^{\alpha\varphi/h^{1/4}}\psi)
=\lambda\|e^{\alpha\varphi/h^{1/4}}\psi\|^2
+\alpha^2h^{\frac32}\|e^{\alpha\varphi/h^{1/4}}\psi\nabla\varphi\|^2\,,\]
so that, using Proposition \ref{prop.roughlowerbound}, \[\int_{\Omega_\delta}\left(h\mathfrak{s}(\theta(x))-Ch^{\frac54}-\alpha^2h^{\frac32}|\nabla\varphi(x)|^2
-\lambda\right)e^{2\alpha(x)\varphi/h^{\frac14}}|\psi(x)|^2\dd x\leq 0 \,.\]
Thus, \[\int_{\Omega_\delta}\left(\mathfrak{s}(\theta(x))-\Theta_0-Ch^{\frac14}\right)
e^{2\alpha\varphi(x)/h^{\frac14}}|\psi(x)|^2\dd x\leq 0\,.\]	
By using Assumption \ref{hyp.genericcancellation}, $\partial\Omega\ni x\mapsto\mathbf{B}\cdot\mathbf{n}(x) $ vanishes linearly on $\Gamma$ and thus, for some $c_0>0$,
\[|\theta(x)| = |\theta(p(x))| \geq c_0\varphi(x)\,.\]
	By \eqref{eq.stheta}, for some other $c_1>0$ we get $\mathfrak{s}(\theta(x))\geq\Theta_0+c_1\varphi(x)$ so that
$$
\int_{\Omega_\delta}\left(\varphi(x)/h^{\frac14}-C\right)
e^{2\alpha\varphi(x)/h^{\frac14}}|\psi(x)|^2\dd x\leq 0\,.
$$	
The conclusion of the proposition easily follows by splitting the previous integral into two parts, considering the set $\set{x \in \Omega_\delta : \varphi(x)/h^{\frac14} \geq 2C}$ and its complementary.
\end{proof}

%%%%%%%%%%%%%%%%%%%%%%%%%%%%%%%%%%%%%%%%%%%%%%%%%%%%%%%%%%%%%%%%%%%%%%%%%%%%%%%%
%%%%%%%%%%%%%%%%%%%%%%%%%%%%%%%%%%%%%%%%%%%%%%%%%%%%%%%%%%%%%%%%%%%%%%%%%%%%%%%%
%%%%%%%%%%%%%%%%%%%%%%%%%%%%%%%%%%%%%%%%%%%%%%%%%%%%%%%%%%%%%%%%%%%%%%%%%%%%%%%%
%%%%%%%%%%%%%%%%%%%%%%%%%%%%%%%%%%%%%%%%%%%%%%%%%%%%%%%%%%%%%%%%%%%%%%%%%%%%%%%%
%%%%%%%%%%%%%%%%%%%%%%%%%%%%%%%%%%%%%%%%%%%%%%%%%%%%%%%%%%%%%%%%%%%%%%%%%%%%%%%%
%%%%%%%%%%%%%%%%%%%%%%%%%%%%%%%%%%%%%%%%%%%%%%%%%%%%%%%%%%%%%%%%%%%%%%%%%%%%%%%%
\section{An optimal localization near $\Gamma$}\label{sec.optimalloc}
Thanks to Proposition \ref{prop.1}, we can focus on the spectral analysis of $\mathscr{L}_{h,\delta_1,\delta_2}$. Near $\Gamma$, we can use adapted coordinates $(r,s,t)$ based on the geodesics of $\partial\Omega$ where we recall that $r$ is the geodesic distance to $\Gamma$ in $\partial\Omega$ and $t= \mathrm{dist}(p(x),x)$ for $x \in \Omega_\delta$. Such coordinates are precisely described in Section \ref{sec.coordnearboundary} where the considerations of \cite{HM04} are slightly revisited. In Section \ref{sec.rst}, we reformulate our magnetic problem in these coordinates and rewrite it using a nice and suitable gauge transformation in Section \ref{sec.gauge}. Using these preparations, Section \ref{sec.optimal-r} is then devoted to prove Proposition \ref{prop.optimal-r}, which tells that the eigenfunctions of $\mathscr{L}_{h,\delta_1,\delta_2}$ are exponentially localized with respect to $r$ at the scale $h^{\frac13}$ (and not only at the rough scale $h^{\frac14}$ as shown in the previous section).

%%%%%%%%%%%%%%%%%%%%%%%%%%%%%%%%%%%%%%%%%%%%%%%%%%%%%%%%%%%%%%%%%%%%%%%%%%%%%%%%
%%%%%%%%%%%%%%%%%%%%%%%%%%%%%%%%%%%%%%%%%%%%%%%%%%%%%%%%%%%%%%%%%%%%%%%%%%%%%%%%
\subsection{Coordinates near $\Gamma$}\label{sec.coordnearboundary}
In this section we build rigourously a new chart of coordinates near $\Gamma$ and give some geometric properties.

%%%%%%%%%%%%%%%%%%%%%%%%%%%%%%%%%%%%%%%%%%%%%%%%%%%%%%%%%%%%%%%%%%%%%%%%%%%%%%%%
\subsubsection{Geodesic coordinates near $\Gamma$ in $\D \Omega$}\label{sec.geocoord}

Let us first consider the parametrization $s \mapsto \gamma(s)$ of $\Gamma$ by arc length. Let us explain how to push it by the geodesic flow on $\partial\Omega$. In the case of a surface embedded in $\R^3$, this can be done quite elementarily and we recall it now.

Denoting by $K$ the second fundamental form of $\partial\Omega$ associated to the Weingarten map defined by
\[\forall U,V\in T_x\partial\Omega\,,\quad K_x(U,V)=\langle \dd \mathbf{n}_x(U),V\rangle\,,\]
we can consider the ODE with parameter $s$ and unknown $r \mapsto \gamma(r,s)$
\[\partial^2_r\gamma(r,s)=-K(\partial_r\gamma(r,s), \partial_r\gamma(r,s))\mathbf{n}(\gamma(r,s))\,,\]
with initial conditions
\[\gamma(0,s)=\gamma(s)\,,\quad \partial_r\gamma(0,s)=\gamma'(s)^{\perp}\,,\]
where $\perp$ is understood in the tangent space and so that $(\partial_r\gamma,\partial_s\gamma,\mathbf{n})$ is a \emph{direct} orthonormal basis. This ODE has a unique smooth solution $(-a,a)\times[0,L[\ni(r,s)\mapsto\gamma(r,s)$ where $a>0$ is chosen small enough.

\begin{lemma}
	$\gamma$ is valued in $\partial\Omega$.
\end{lemma}
\begin{proof}
	We first notice that \[\partial_r\langle\partial_r\gamma,\mathbf{n}\rangle=-K(\partial_r\gamma,\partial_r\gamma)+\langle\dd\mathbf{n}(\partial_r\gamma),\partial_r\gamma\rangle=0\,.\]
	Thus, by using the initial condition, $\partial_r\gamma\in T_{x}\partial\Omega$.
	Now, consider
	\[\partial_r\langle \partial_s\gamma,\mathbf{n}\rangle=\langle\partial_r\partial_s\gamma,\mathbf{n}\rangle+\langle\partial_s\gamma,\dd\mathbf{n}(\partial_r\gamma)\rangle=-\langle\partial_r\gamma,\dd\mathbf{n}(\partial_s\gamma)\rangle+\langle\partial_s\gamma,\dd\mathbf{n}(\partial_r\gamma)\rangle=0\,,\]
	where we used $\partial_r\gamma\in T_{x}\partial\Omega$ and the symmetry of the Weingarten map. Using the initial condition, we see that $\partial_s\gamma\in T_x\partial\Omega$.
\end{proof}
\begin{lemma} \label{lem.normar}
	We have $|\partial_r\gamma(r,s)|=1$ and $\langle\partial_r\gamma,\partial_s\gamma\rangle=0$.
\end{lemma}
\begin{proof}
	We have \[\langle\partial^2_r\gamma,\partial_r\gamma\rangle=0\,,\]
	and thus $r\mapsto|\partial_r\gamma(r,s)|$ is constant and equals $1$ (due to the normalization of the initial condition). This also implies that
	\[\langle\partial_s\partial_r\gamma,\partial_r\gamma\rangle=0\,.\]
	We deduce that
	\[\partial_r\langle\partial_r\gamma,\partial_s\gamma\rangle=0\,,\]
	and the conclusion follows by considering the initial condition.
\end{proof}
These considerations show that $(r,s) \mapsto \gamma(r,s)$ is a local chart in $\D \Omega$ near $\Gamma$ for which $\Gamma$ is given by $\set{(r,s) : r=0}$. This chart also allows us to easily identify geometric quantities. We directly get the following lemma.
\begin{lemma}\label{lem.metrict=0}
	In this chart, the first fundamental form on $\partial\Omega$ is given by the matrix
	\[g(r,s)=\begin{pmatrix}
		1&0\\
		0&\alpha(r,s)
	\end{pmatrix}
	\,,\quad \alpha(r,s)=|\partial_s\gamma(r,s)|^2\,,\]	
\end{lemma}
Note that on $\Gamma$ some simplifications occurs. Recalling that the geodesic curvature $\kappa_g(s)$ and normal curvature $\kappa_n(s)$ for $\Gamma$ are defined at a point $\gamma(s) = \gamma(0,s)$ by
$$
\partial^2_s\gamma(s)=-\kappa_g(s) \D_s \gamma(s)\times\mathbf{n}(\gamma(s))+\kappa_n(s) \mathbf{n}(\gamma(s))\,,
$$
we get the following.

\begin{lemma}\label{lem.alpha=1}
	For all $s \in [0,2L]$, we have $\alpha(0,s)=1$ and $\partial_s\alpha(0,s)=0$. Moreover we have $\partial_r\alpha(0,s)=-2\kappa_g(s)$.
\end{lemma}

\begin{proof}
	The first two equalities come from $|\partial_s\gamma(0,s)|=1$. The last one follows from	\[\partial_r\alpha(r,s)=2\langle\partial_r\partial_s\gamma(r,s),\partial_s\gamma(r,s)\rangle
=-2\langle\partial_r\gamma(r,s),\partial^2_s\gamma(r,s)\rangle\]
	Thus, when $r=0$
	\[ \partial_r\alpha(0,s)=2\kappa_g(s)\langle\partial_r\gamma(s,0),
\D_s\gamma(s,0)\times\mathbf{n(\gamma(s,0)}\rangle
=-2\kappa_g(s)\,.\]
\end{proof}

%%%%%%%%%%%%%%%%%%%%%%%%%%%%%%%%%%%%%%%%%%%%%%%%%%%%%%%%%%%%%%%%%%%%%%%%%%%%%%%%
\subsubsection{Coordinates near $\Gamma$ in $\Omega_\delta$}
In this section, we choose $\delta$ sufficiently small and consider the local chart of the boundary $(r,s) \mapsto \gamma(r,s)$ defined globally near $\Gamma$. We consider the associated tubular coordinates
\[y=(r,s,t)\mapsto \Phi(r,s,t) =\gamma(s,t)-t\mathbf{n}(\gamma(s,t))=x\,.\]
The map $\Phi$ is a smooth (global in a neighborhood of $\Gamma$) diffeomorphism from $\{t=0\}$ to $\partial\Omega$.
The differential of $\Phi$ can be written as
\begin{equation}\label{eq.dPhiy}
\dd\Phi_{y}=[(\mathrm{Id}-t\dd\mathbf{n})(\partial_r\gamma),(\mathrm{Id}-t\dd\mathbf{n})(\partial_s\gamma),-\mathbf{n}]\,,
\end{equation}
and the Euclidean metrics becomes
\begin{equation} \label{eq:G}
G=(\dd\Phi)^{\mathrm{T}}\dd\Phi=
\begin{pmatrix}
g&0\\
0&1
\end{pmatrix}
\,,
\end{equation}
with the following extended definition of $g$
\[
g(r,s,t)=
\begin{pmatrix}
\|(\mathrm{Id}-t\dd\mathbf{n})(\partial_r\gamma)\|^2&\langle(\mathrm{Id}-t\dd\mathbf{n})(\partial_r\gamma),(\mathrm{Id}-t\dd\mathbf{n})(\partial_s\gamma)\rangle\\
\langle(\mathrm{Id}-t\dd\mathbf{n})(\partial_r\gamma),(\mathrm{Id}-t\dd\mathbf{n})(\partial_s\gamma)\rangle&\|(\mathrm{Id}-t\dd\mathbf{n})(\partial_s\gamma)\|^2
\end{pmatrix}
\,.\]
We recover on $\set{t=0}$ the matrix $g(r,s)=g(r,s,0)$  of the first fundamental form of $\partial\Omega$ as defined in Lemma \ref{lem.metrict=0}.

%%%%%%%%%%%%%%%%%%%%%%%%%%%%%%%%%%%%%%%%%%%%%%%%%%%%%%%%%%%%%%%%%%%%%%%%%%%%%%%%
%%%%%%%%%%%%%%%%%%%%%%%%%%%%%%%%%%%%%%%%%%%%%%%%%%%%%%%%%%%%%%%%%%%%%%%%%%%%%%%%
\subsection{Magnetic Laplacian in the new coordinates}\label{sec.rst}

In this section, we give the new expression of the magnetic Laplacian in the coordinates defined in the preceding section and exhibit an adapted change of gauge. This expression will be of crucial help for proving the refined localization in Section \ref{sec.optimal-r}.

%%%%%%%%%%%%%%%%%%%%%%%%%%%%%%%%%%%%%%%%%%%%%%%%%%%%%%%%%%%%%%%%%%%%%%%%%%%%%%%%
\subsubsection{The magnetic form in tubular coordinates}
We consider the $1$-form
\[\sigma=\mathbf{A}\cdot\dd x=\sum_{\ell=1}^3 A_\ell\dd x_\ell\,.\]
Its exterior derivative is the magnetic $2$-form
\[\omega=\dd \sigma=\sum_{1\leq k<\ell\leq 3}(\partial_k A_{\ell}-\partial_{\ell}A_k)\dd x_k\wedge\dd x_{\ell}\,,\]
which can also be written as
\[\omega=B_3\dd x_1\wedge \dd x_2-B_2\dd x_1\wedge \dd x_3+B_1\dd x_2\wedge \dd x_3\,.\]
Note also that
\[\forall U,V\in\R^3\,,\quad \omega(U,V)=[U,V,\mathbf{B}]=\langle U\times V,\mathbf{B}\rangle\,.\]
Let us now consider the effect of the change of variables $\Phi(y) = x$. We have
\begin{equation}\label{eq.tildeA}
\Phi^*\sigma=\sum_{j=1}^3 \tilde A_j \dd y_j\,,\quad \tilde{\mathbf{A}}=(\dd\Phi)^{\mathrm T}\circ\mathbf{A}\circ\Phi\,,
\end{equation}
and
\[\Phi^*\omega=\Phi^*\dd\sigma=\dd (\Phi^*\sigma)=[\cdot,\cdot,\nabla\times\tilde{\mathbf{A}}]\,.\]
This also gives that, for all $U,V\in\R^3$,
\[[\dd\Phi(U),\dd\Phi(V),\mathbf{B}]=[U,V,\nabla\times\tilde{\mathbf{A}}]\,,\quad
\mbox{ or } \det\dd\Phi[\cdot,\cdot,\dd\Phi^{-1}(\mathbf{B})]=[\cdot,\cdot,\nabla\times\tilde{\mathbf{A}}]\,,\]
so that,
\[\nabla\times\tilde{\mathbf{A}}=(\det\dd\Phi)\, \dd\Phi^{-1}(\mathbf{B})\,.\]
Note then that using \eqref{eq:G} we get
\begin{equation}\label{eq.coordinatesnewB}
|g|^{-\frac12}\nabla\times\tilde{\mathbf{A}}=\mathcal{B}\,.
\end{equation}

The coordinates of vector $(\bbb_1,\bbb_2,\bbb_3)$ of $\mathcal{B}(y)=\dd\Phi^{-1}_y(\mathbf{B}(x))$ correspond then to the coordinates of $\mathbf{B}(y)$ in the image of the canonical basis by $\dd\Phi_y$.
Note that until now, computations were valid for a wider class of change of variables than the one built in Section \ref{sec.coordnearboundary}. Let us now focus on this case when $y=(r,s,t)\mapsto \Phi(r,s,t) =\gamma(s,t)-t\mathbf{n}(\gamma(s,t))=x$.

We get that
\begin{equation}\label{eq.e3}
\mathbf{B}=\mathbf{e}_3=\mathrm{d}\Phi(\mathcal{B})=\mathcal{B}_1(\mathrm{Id}-t\mathrm{d}\mathbf{n})(\partial_r\gamma)+\mathcal{B}_2(\mathrm{Id}-t\mathrm{d}\mathbf{n})(\partial_s\gamma)-\mathcal{B}_3\mathbf{n}\,,
\end{equation}
so that, on the boundary $\set{(r,s,0)}$ (near $\Gamma$) we have
\begin{equation*}
\mathbf{e}_3=\mathcal{B}_1\partial_r\gamma+\mathcal{B}_2\partial_s\gamma-\mathcal{B}_3\mathbf{n}\,,
\end{equation*}
and in particular
\begin{equation}\label{eq.norm1}
	\mathcal{B}^2_1(r,s,0)+\alpha(r,s)\mathcal{B}^2_2(r,s,0)+\mathcal{B}_3^2(r,s,0)=1\,.
\end{equation}
When we are on $\Gamma = \set{(0,s,0)}$, we can describe how the magnetic field is tangent to the boundary, depending on the curvilinear coordinate $s$ and since $\mathbf{e}_3$ is orthogonal to $\mathbf{n}(0,s,0)$ we have
\begin{equation} \label{lem.e3r=t=0}
\mathbf{e}_3=\mathcal{B}_1(0,s,0)\partial_r\gamma(0,s)+\mathcal{B}_2(0,s,0)\partial_s\gamma(0,s)\,.
\end{equation}
By Lemma \ref{lem.normar}
	\[\mathcal{B}_1(0,s,0)=\langle \mathbf{e}_3, \partial_r\gamma\rangle\,,\qquad \mathcal{B}_2(0,s,0)=\left\langle \mathbf{e}_3, \partial_s\gamma\right\rangle \,.\]
%\end{lemma}
and the following degeneracy result holds.

\begin{lemma}
	We have $\partial_t\mathcal{B}_3(r,s,0)=0$.
\end{lemma}
\begin{proof}
Taking the derivative with respect to $t$ of \eqref{eq.e3}, we get, when $t=0$, \[\mathcal{B}_1\dd\mathbf{n}(\partial_r\gamma)+\mathcal{B}_2\dd\mathbf{n}(\partial_s\gamma)=\partial_t\mathcal{B}_1\partial_r\gamma+\partial_t\mathcal{B}_2\partial_s\gamma+\partial_t\mathcal{B}_3\mathbf{n}\,.\]
	Thus, $\partial_t\mathcal{B}_3(r,s,0)=0$.
\end{proof}

For further use let us eventually give also the expressions of the coordinates of $\bbb(r,s,0)$ with respect to the angles $\phi(s)$ defined in Definition \ref{def.phi} and $\theta(x)$ defined in Definition \ref{def.theta}. The natural extension of their definition as functions of $(r,s)$ for any $y=(r,s,0)$ of the (image of ) the boundary is the following, and is linked to the normalization property \eqref{eq.norm1}.

\begin{definition} \label{def.phithetars}
We consider $\phi(r,s) \in (-\pi, \pi)$ and $\theta(r,s) \in (-\frac\pi2, \frac\pi2)$ the angles defined by
\begin{equation*}
\begin{split}
& \mathcal{B}_1(r,s,0)=  \cos\phi(r,s)\cos\theta(r,s)\,,  \quad\alpha^{\frac12}(r,s)\mathcal{B}_2(r,s,0)=\sin\phi(r,s)\cos\theta(r,s)\,, \\
& \mathcal{B}_3(r,s,0)=\sin\theta(r,s)\,
\end{split}
\end{equation*}
\end{definition}

Note that this definition is consistent with Definitions \ref{def.phi} and \ref{def.theta} and the associated figures, possibly after identifying any point $x$ of the boundary with its image in coordinates $(r,s,0)$.

%%%%%%%%%%%%%%%%%%%%%%%%%%%%%%%%%%%%%%%%%%%%%%%%%%%%%%%%%%%%%%%%%%%%%%%%%%%%%%%%
\subsubsection{The magnetic Laplacian in tubular coordinates}\label{sec.quadraticy}
Recall that the quadratic form associated with $\mathscr{L}_h$ is given in the original coordinates by
\[\mathscr{Q}_h(\psi)=\int_{\Omega}|(-ih\nabla-\mathbf{A})\psi|^2\dd x\,,\quad \forall \psi\in H^1(\Omega)\,.\]
If the support of $\psi$ is close enough to the boundary, we may express $\mathscr{Q}_h(\psi)$ in the local chart given by $\Phi(y) = x$. Letting $\tilde\psi(y)=\psi\circ\Phi(y)$, we have then
\[\mathscr{Q}_h(\psi)=\int \langle G^{-1}(-ih\nabla_y-\tilde{\mathbf{A}}(y))\tilde\psi,(-ih\nabla_y-\tilde{\mathbf{A}}(y))\tilde\psi \rangle |g|^{\frac12}\dd y\,.\]
In the Hilbert space $L^2(|g|^{\frac12}\dd y)$, the operator locally takes the form
\begin{equation}\label{eq.magLaptilde}
|g|^{-\frac12}(-ih\nabla_y-\tilde{\mathbf{A}}(y))\cdot|g|^{\frac12}G^{-1}(-ih\nabla_y-\tilde{\mathbf{A}}(y))\,.
\end{equation}

The expression of the operator $\mathscr{L}_{h,\delta_1,\delta_2}$ in coordinates $(r,s,t)$ of $y$  after transformation of \eqref{eq.magLaptilde} is then
\begin{multline}\label{eq.tildeLhd}
\widetilde{\mathscr{L}}_{h,\delta_1,\delta_2}  = |g|^{-\frac12}(hD_r-\tilde A_1)|g|^{\frac12}g^{11}(hD_r-\tilde A_1)+|g|^{-\frac12}(hD_s-\tilde A_2)|g|^{\frac12}g^{22}(hD_s-\tilde A_2)\\
 +|g|^{-\frac12}(hD_r-\tilde A_1)|g|^{\frac12}g^{12}(hD_s-\tilde A_2)+|g|^{-\frac12}(hD_s-\tilde A_2)|g|^{\frac12}g^{12}(hD_r-\tilde A_1)\\
 +|g|^{-\frac12}(hD_t)|g|^{\frac12}(hD_t)\,,
\end{multline}
where the coefficients can be described as follows
\begin{equation}\label{eq.approxmetric}
\begin{split}
|g(r,s,t)|^{\frac12}&=\alpha(r,s)^{\frac12}+tk(r,s))+\mathscr{O}(t^2),\qquad
g^{11}(r,s,t)=1+tk_1(r,s)+\mathscr{O}(t^2)\,,\\
g^{22}(r,s,t)&=\alpha(r,s)^{-1}+tk_2(r,s)+\mathscr{O}(t^2)\,,\qquad
g^{12}(r,s,t)=tk_{12}(r,s)+\mathscr{O}(t^2)\,,
\end{split}
\end{equation}
where the functions $k$, $k_1$, $k_2$, and $k_{12}$ are smooth.

\subsection{A change of gauge}\label{sec.gauge}
The following proposition provides us with an appropriate gauge reducing the third coordinate of the vector potential to zero.

\begin{proposition}\label{prop.changegauge}
There exists a smooth function $\phi=\phi(r,s,t)$, on $(-a,a)\times [0,2L)\times (0,\delta)$, $2L$-periodic with respect to $s$, such that
	\[\begin{split}
		\tilde A_1-\partial_1\phi&=\int_0^t[|g|^{\frac12}\mathcal{B}_2](r,s,\tau)\dd\tau\,,\\
		\tilde A_2-\partial_2\phi&= \seq{f}+\!\!\int_0^r [|g|^{\frac12}\mathcal{B}_3](\rho,s,t)\dd\rho+\int_0^r\int_0^t\partial_2[|g|^{\frac12}\mathcal{B}_2](\rho,s,\tau)\dd\rho\dd\tau-\!\!\int_0^t [|g|^{\frac12}\mathcal{B}_1](0,s,\tau) \dd\tau\,,\\
		\tilde A_3-\partial_3\phi&=0\,,
	\end{split}\]	
with $L$ being the half-length of $\Gamma$ and
\[\seq{f}=\frac{1}{|\Gamma|}\int_{\partial\Omega^+}n_3\,\mathrm{d}S\,,\]
where $\partial\Omega^+=\{x\in\partial\Omega : n_3(x)=\mathbf{n}(x)\cdot\mathbf{e}_3>0\}$ and $\mathrm{d}S$ is the usual surface measure on $\partial\Omega$ induced by the Euclidean measure in $\mathbb{R}^3$.

\end{proposition}

\begin{proof}
	Let us recall \eqref{eq.coordinatesnewB}.
	We have
	\[\begin{split}
		\partial_2\tilde{A}_3-\partial_3\tilde A_2&=|g|^{\frac12}\mathcal{B}_1\,,\\
		\partial_3\tilde{A}_1-\partial_1\tilde A_3&=|g|^{\frac12}\mathcal{B}_2\,,\\
		\partial_1\tilde{A}_2-\partial_2\tilde A_1&=|g|^{\frac12}\mathcal{B}_3\,.
	\end{split}\]
	Considering
	\[\psi_1(r,s,t)=\int_0^t\tilde A_3(r,s,\tau)\mathrm{d}\tau\,,\]
	and $\widehat{\mathbf{A}}=\widetilde{\mathbf{A}}-\nabla\psi_1$, we see that $\hat A_3=0$ and that $\nabla\times\widehat{\mathbf{A}}=|g|^{\frac12}\mathcal{B}$. Thus,
		\[\begin{split}
	-\partial_3\hat A_2&=|g|^{\frac12}\mathcal{B}_1\,,\\
	\partial_3\hat{A}_1&=|g|^{\frac12}\mathcal{B}_2\,,\\
	\partial_1\hat{A}_2-\partial_2\hat A_1&=|g|^{\frac12}\mathcal{B}_3\,.
	\end{split}\]
The second equation provides us with
	\[\hat A_1(r,s,t)=\int_0^t[|g|^{\frac12}\mathcal{B}_2](r,s,\tau)\dd\tau+\varphi_1(r,s)\,,\]
whereas the last equation gives
\[\hat A_2(r,s,t)=\int_0^r [|g|^{\frac12}\mathcal{B}_3](\rho,s,t)\dd\rho+\int_0^r\partial_2\hat A_1(\rho,s,t)\dd\rho+\varphi_2(s,t)\,,\]	
where $\varphi_1$ and $\varphi_2$ are smooth functions that are $2L$-periodic with respect to $s$. We get
\begin{multline*}
\hat A_2(r,s,t)=\int_0^r [|g|^{\frac12}\mathcal{B}_3](\rho,s,t)\dd\rho+\int_0^r\int_0^t\partial_2[|g|^{\frac12}\mathcal{B}_2](\rho,s,\tau)\dd\rho\dd\tau\\
+\varphi_2(s,t)+\int_0^r\partial_2\varphi_1(\rho,s)\dd\rho\,.
\end{multline*}
Considering $\check{\mathbf{A}}=\hat{\mathbf{A}}-\nabla\psi_2$ with $\psi_2(r,s,t)=\int_0^r\varphi_1(\rho,s)\dd\rho$, we find that
	\[\check A_1(r,s,t)=\int_0^t[|g|^{\frac12}\mathcal{B}_2](r,s,\tau)\dd\tau\,,\]
\[
\check A_2(r,s,t)=\int_0^r [|g|^{\frac12}\mathcal{B}_3](\rho,s,t)\dd\rho+\int_0^r\int_0^t\partial_2[|g|^{\frac12}\mathcal{B}_2](\rho,s,\tau)\dd\rho\dd\tau+\varphi_2(s,t)\,,
\]
and $\check A_3=0$. Then, using that $\nabla \cdot (|g|^{\frac12} \bbb) = 0$ from \eqref{eq.coordinatesnewB}, we get
\begin{equation*}
\begin{split}
|g|^{\frac12}\mathcal{B}_1=-\partial_3\check A_2
&=-\int_0^r \partial_3[|g|^{\frac12}\mathcal{B}_3](\rho,s,t)\dd\rho-\int_0^r\partial_2[|g|^{\frac12}\mathcal{B}_2](\rho,s,t)\dd\rho-\partial_3\varphi_2(s,t)\\
&=\int_0^r \partial_1[|g|^{\frac12}\mathcal{B}_1](\rho,s,t)\dd\rho-\partial_3\varphi_2(s,t)\\
&=-[|g|^{\frac12}\mathcal{B}_1](0,s,t)+[|g|^{\frac12}\mathcal{B}_1](r,s,t)-\partial_3\varphi_2(s,t)\,,
\end{split}
\end{equation*}
so that
\[\partial_3\varphi_2(s,t)=-[|g|^{\frac12}\mathcal{B}_1](0,s,t)\,.\]
Thus, there exists $a$ 2L-periodic function $s \mapsto f(s)=\varphi_2(s,0)$ such that
\[\varphi_2(s,t)= f(s) -\int_0^t[|g|^{\frac12}\mathcal{B}_1](0,s,\tau)\dd\tau\,.\]
We can now perform a last change of gauge by considering the following $2L$-periodic function
\[\psi_3(s)=\int_0^{ s }\left(f( \sigma )-\seq{f}\right)\dd \sigma\,,\quad \seq{f}=\frac{1}{2L}\int_0^{2L}f(\sigma)\dd \sigma\,.\]
We let $\breve{\mathbf{A}}=\check{\mathbf{A}}-\nabla\psi_3=\widetilde{\mathbf{A}}-\nabla(\psi_1+\psi_2+\psi_3)$ and we have
\[\breve{A}_1(r,s,t)=\int_0^t[|g|^{\frac12}\mathcal{B}_2](r,s,\tau)\dd\tau\,,\]
\[
\breve{A}_2(r,s,t)=\seq{f}+\int_0^r [|g|^{\frac12}\mathcal{B}_3](\rho,s,t)\dd\rho+\int_0^r\int_0^t\partial_2[|g|^{\frac12}\mathcal{B}_2](\rho,s,\tau)\dd\rho\dd\tau-\int_0^t[|g|^{\frac12}\mathcal{B}_1](0,s,\tau)\dd\tau\,,
\]
and $\breve{A}_3=0$.

Notice that
\[ f(s)=\check A_2(0,s,0)=\hat{A}_2(0,s,0)=\tilde{A}_2(0,s,0)=\mathbf{A}(\gamma(s))\cdot\gamma'(s)\,,\]
where we used \eqref{eq.dPhiy} and \eqref{eq.tildeA}. This shows that
\[ \seq{f}=\frac{1}{|\Gamma|}\int_{\Gamma}\mathbf{A}\cdot \overrightarrow{\dd\ell}=\frac{1}{|\Gamma|}\int_{\partial\Omega^+}\mathbf{B}\cdot\mathbf{n}\,\dd S=\frac{1}{|\Gamma|}\int_{\partial\Omega^+}n_3\dd S\,,\]
thanks to the Ostrogradski-Stokes formula and the fact that $\mathbf{B}=\mathbf{e}_3$.
\end{proof}
\begin{remark}\label{rem.tildeA}
Thanks to Proposition \ref{prop.changegauge}, the spectral analysis of $\mathscr{L}_{h,\delta}$ can be done by assuming that
	\[\begin{split}
	\tilde A_1&=\int_0^t[|g|^{\frac12}\mathcal{B}_2](r,s,\tau)\dd\tau\,,\\
	\tilde A_2&= \seq{f}+\int_0^r [|g|^{\frac12}\mathcal{B}_3](\rho,s,t)\dd\rho+\int_0^r\int_0^t\partial_2[|g|^{\frac12}\mathcal{B}_2](\rho,s,\tau)\dd\rho\dd\tau-\int_0^t [|g|^{\frac12}\mathcal{B}_1](0,s,\tau) \dd\tau\,,\\
	\tilde A_3&=0\,.
\end{split}\]
In fact, by considering the $2L$-periodic functions $(e^{\frac{im}{2L}s})_{m\in\mathbb{Z}}$, we can even replace $\seq{f}$ by $\seq{f}-h\frac{m}{2L}$. There exists a unique $(m_h,\mathfrak{f}_h)\in\mathbb{Z}\times\left(-\frac{1}{2L},\frac{1}{2L}\right)$ such that $\seq{f}=m_h\frac{h}{2L}+h\mathfrak{f}_h$. Thus, we assume, as we may, that $\seq{f}=h\mathfrak{f}_h$.
\end{remark}

\subsection{Optimal localization near $\Gamma$}\label{sec.optimal-r}
Now comes the most important result of the section: the refined exponential localization of the eigenfunctions near $r=0$, at the scale $h^{\frac13}$, to be compared to the rough one at scale $h^{\frac14}$ stated in Proposition \ref{prop.roughlocGamma}.

\begin{proposition}\label{prop.optimal-r}
For all eigenfunctions $\varphi$ (of $\mathscr{L}_{h,\delta_1,\delta_2}$) associated with $\lambda\leq \Theta_0h+Ch^{\frac43}$, we have
 \[\int_{\Omega_{\delta_1,\delta_2}}e^{2r(x)/h^{\frac13}}|\varphi(x)|^2\dd x\leq C\|\varphi\|^2\,,\]
 and
 \[\mathscr{Q}_{h,{\delta_1,\delta_2}}(e^{r/h^{\frac13}}\varphi)\leq Ch\|\varphi\|^2\,\]
 where $\delta_1=h^{\frac12-\eta}$ and $\delta_2=h^{\frac14-\eta}$ for some $\eta\in\left(0,\frac{1}{12}\right)$.
\end{proposition}
\begin{proof}
We can write
\[\lambda\|e^{r/h^{\frac13}}\varphi\|^2=\Re\langle\mathscr{L}_{h,\delta_1,\delta_2}\varphi,e^{2r/h^{\frac13}}\varphi\rangle=\mathscr{Q}_{h,\delta_1,\delta_2}(e^{r/h^{\frac13}}\varphi)-h^{\frac43}\|e^{r/h^{\frac13}}(\nabla r)\varphi\|^2\,.\]
Since $\lambda\leq \Theta_0h+Ch^{\frac43}$, we get
\begin{equation} \label{eq.eqvp}
\mathscr{Q}_{h,\delta_1,\delta_2}(\psi)\leq(\Theta_0 h+\tilde Ch^{\frac43})\|\psi\|^2\,,\quad \psi=e^{r/h^{\frac13}}\varphi\,.
\end{equation}		
Let us use a (finite) partition of the unity $\sum_j \chi_j^2 = 1$ of $\Omega_{\delta_1, \delta_2}$ with balls of size $h^{\frac13}$. We can write
\begin{equation} \label{eq.partition} \mathscr{Q}_{h,\delta_1,\delta_2}(\psi)=\sum_{j}\mathscr{Q}_{h,\delta_1,\delta_2}(\psi_j)-h^2\sum_{j}\|\psi\nabla\chi_j\|^2\geq \sum_{j}\mathscr{Q}_{h,\delta_1,\delta_2}(\psi_j)-Ch^{\frac43}\|\psi\|^2\,
\end{equation}
where $ \psi_j = \chi_j \psi$. First notice that for the interior balls, \ie when  $\mathrm{supp}\,\psi_j\cap\partial\Omega=\emptyset$, we have
\begin{equation} \label{eq.interior}
\mathscr{Q}_{h,\delta_1,\delta_2}(\psi_j)  \geq  h\|\psi_j\|^2\,.
\end{equation}
From now on, we can focus on indices such that $\mathrm{supp}\,\psi_j\cap\partial\Omega \neq\emptyset$ and consider for further use $y_j = (r_j, s_j,0)$ a point in this intersection. We also define $\tilde{\Omega}_{\delta_1, \delta_2}$ the image of $\Omega_{\delta_1, \delta_2}$ by $x \mapsto y=(r,s,t)$ and $\tilde{\psi}(y) = \psi(x)$. Notice here that for all $(r,s,t) \in \tilde{\Omega}_{\delta_1, \delta_2}$, we have
\[
r= \ooo(h^{\frac14-\eta})\, \qquad t= \ooo(h^{\frac12-\eta})\,.
\]
We have then by Taylor expansion
	\begin{multline*}
	\mathscr{Q}_{h,\delta_1,\delta_2}(\psi_j)\geq  \int_{V_{\delta}}|g|^{\frac12}(r,s,0)\langle G^{-1}(r,s,0)(-ih\nabla_y-\tilde{\mathbf{A}})\tilde\psi_j,(-ih\nabla_y-\tilde{\mathbf{A}})\tilde\psi_j\rangle\dd y\\
	-C\int_{ \tilde{\Omega}_{\delta_1, \delta_2} } t|(-ih\nabla_y-\tilde{\mathbf{A}})\psi_j|^2\dd y\,,
	\end{multline*}
which implies that
	\begin{equation*}
(1+Ch^{\frac12-\eta})\mathscr{Q}_{h,\delta_1,\delta_2}(\psi_j)\geq  \int_{ \tilde{\Omega}_{\delta_1, \delta_2} }|g|^{\frac12}(r,s,0)\langle G^{-1}(r,s,0)(-ih\nabla_y-\tilde{\mathbf{A}})\tilde\psi_j,(-ih\nabla_y-\tilde{\mathbf{A}})\tilde\psi_j\rangle\dd y\,,
\end{equation*}
and then, using the explicit expression of $\tilde{\mathbf{A}}$ (see Remark \ref{rem.tildeA}) and of $|g|$,
	\begin{multline*}
(1+Ch^{\frac12-\eta})\mathscr{Q}_{h,\delta_1,\delta_2}(\psi_j)\\
\geq  \int_{ \tilde{\Omega}_{\delta_1, \delta_2} }\alpha^{\frac12}(r,s)\left(|h\partial_t\tilde\psi_j|^2+|(-ih\partial_r-\tilde A_1)\tilde\psi_j|^2+\alpha^{-1}|(-ih\partial_s-\tilde A_2)\tilde\psi_j|^2\right)\dd y\,.
\end{multline*}
Since the balls are of radius $h^{\frac13}$, this yields
	\begin{multline*}
(1-Ch^{\frac13})^{-1}(1+Ch^{\frac12-\eta})\mathscr{Q}_{h,\delta_1,\delta_2}(\psi_j)\\
\geq  \int_{\tilde{\Omega}_{\delta_1, \delta_2} }\alpha^{\frac12}_j\left(|h\partial_t\tilde\psi_j|^2+|(-ih\partial_r-\tilde A_1)\tilde\psi_j|^2+\alpha^{-1}_j|(-ih\partial_s-\tilde A_2)\tilde\psi_j|^2\right)\dd y\,,
\end{multline*}
where $\alpha_j:=\alpha(y_j)$. Up to a local change of gauge eliminating $\seq{f}$ (see Remark \ref{rem.tildeA}) and using a Taylor expansion at $y_j$, we may write
\[\begin{split}
\tilde A_1&=t[|g|^{\frac12}\mathcal{B}_2](y_j)+\mathscr{O}(t^2+|r-r_j|t+|s-s_j|t)\\ \tilde A_2&=(r-r_j)[|g|^{\frac12}\mathcal{B}_3](y_j)-t[|g|^{\frac12}\mathcal{B}_1](y_j)+R_j(r,s)+\mathscr{O}(t|s-s_j|+t^2)\\
\tilde A_3&=0\,,
\end{split}\]
with
\begin{equation} \label{eq.defRj}
R_j(r,s)=q_{rr,j}(r-r_j)^2+q_{rs,j}(r-r_j)(s-s_j).
\end{equation}
Using Definition \ref{def.phithetars}, we can write
\[\alpha_j^{\frac12}\mathcal{B}_1(y_j)=\alpha^{\frac12}_j\cos\phi_j\cos\theta_j\,,\quad \alpha_j^{\frac12}\mathcal{B}_2(y_j)=\sin\phi_j\cos\theta_j\,,\quad \alpha_j^{\frac12}\mathcal{B}_3(y_j)=\alpha_j^{\frac12}\sin\theta_j\,,\]
where $\phi_j= \phi(r_j,s_j)$ and $\theta_j=\theta(r_j,s_j)$,
so that
\[\begin{split}
\tilde A_1&=t\sin\phi_j\cos\theta_j+\mathscr{O}(t^2+|r-r_j|t+|s-s_j|t),\\ \tilde A_2&=(r-r_j)\alpha^{\frac12}_j \sin\theta_j-t\alpha^{\frac12}_j\cos\phi_j\cos\theta_j+R_j(r,s)+\mathscr{O}(t|s-s_j|+t^2),\\
\tilde A_3&=0\,.
\end{split}\]
Using then the standard inequality $(a+b)^2 \geq (1-\eps) a^2- \eps^{-1} b^2$ we get, for all $\eps\in(0,1)$,
\begin{equation*}
(1-Ch^{\frac13})^{-1}(1+Ch^{\frac12-\eta})\mathscr{Q}_{h,\delta}(\psi_j)
\geq (1-\epsilon)Q_j(\tilde\psi_j)-C\epsilon^{-1}\left(h^{\frac23}\|t\tilde\psi_j\|^2+\|t^2\tilde\psi_j\|^2\right)\,,
\end{equation*}
where
\begin{multline} \label{eq.Qj}
Q_j(\psi):=
\alpha_j^{\frac12}\int_{\tilde{\Omega}_{\delta_1, \delta_2} }\Big(|h\partial_t\psi|^2+|(-ih\partial_r-t\sin\phi_j\cos\theta_j)\psi|^2\\
+|(-ih\alpha^{-\frac12}_j\partial_s+t\cos\phi_j\cos\theta_j-(r-r_j)\sin\theta_j-\alpha^{-\frac12}_jR_j)\psi|^2\Big)\dd y\,.
\end{multline}
We get, with $\epsilon=h^{\frac13}$,
\begin{equation*}
	(1+Ch^{\frac13})\mathscr{Q}_{h,\delta_1,\delta_2}(\psi_j)
	\geq Q_j(\tilde\psi_j)-C\left(h^{\frac13}\|t\tilde\psi_j\|^2+h^{-\frac13}\|t^2\tilde\psi_j\|^2\right)\,,
\end{equation*}
so that using $t = \ooo(h^{\frac12-\eta})$
\begin{equation*}
(1+Ch^{\frac13}) \mathscr{Q}_{h,\delta_1,\delta_2}(\psi_j)
	\geq Q_j(\tilde\psi_j)-C\left(h^{\frac13}\|t\tilde\psi_j\|^2+h^{1-2\eta-\frac13}\|t\tilde\psi_j\|^2\right)\,,
\end{equation*}
and eventually
\begin{equation} \label{eq.boundary1}
(1+Ch^{\frac13}) \mathscr{Q}_{h,\delta_1,\delta_2}(\psi_j)
	\geq Q_j(\tilde\psi_j)-Ch^{\frac13}\|t\tilde\psi_j\|^2\,.
\end{equation}
Let us now gather \eqref{eq.eqvp}, \eqref{eq.partition}, \eqref{eq.interior}, and \eqref{eq.boundary1}. Denoting by $J$ the set of the indices related to the balls intersecting the boundary and $I=\complement J$ the set of indices related to the balls strictly inside $\Omega$, we get
\[-Ch^{\frac13}\sum_{j \in J} \|t\tilde\psi_j\|^2+(1-Ch^{\frac13})\sum_{j \in J }Q_j(\tilde\psi_j)+\sum_{j \in I }h\|\psi_j\|^2\leq(\Theta_0 h+\tilde Ch^{\frac43})\|\psi\|^2\,,\]
so that
\[-Ch^{\frac13} \|t\psi\|^2+(1-Ch^{\frac13})\sum_{j \in J }Q_j(\tilde\psi_j)+\sum_{j\in I}h\|\psi_j\|^2\leq(\Theta_0 h+\tilde Ch^{\frac43})\|\psi\|^2\,,\]
which, combined with the Agmon estimates with respect to $t$ (see Proposition \ref{prop.loct}), gives
\[(1-Ch^{\frac13})\sum_{j\in J}Q_j(\tilde\psi_j)+\sum_{j\in I}h\|\psi_j\|^2\leq(\Theta_0 h+\tilde Ch^{\frac43})\|\psi\|^2\,.\]
This can also be written as
\[\sum_{j\in J}Q_j(\tilde\psi_j)+\sum_{j\in I}h\|\psi_j\|^2\leq(\Theta_0 h+ Ch^{\frac43})\|\psi\|^2\,,\]
and, taking $c= 1-\Theta_0 >0$ and using the partition again, we get
\begin{equation} \label{eq.synth2}
\sum_{j\in I}ch\|\psi_j\|^2+\sum_{j\in J}(Q_j(\tilde\psi_j)-\Theta_0h\|\tilde\psi_j\|^2)-\tilde C h^{\frac43}\|\tilde\psi_j\|^2)\leq 0\,.
\end{equation}
Let us now look at $Q_j$ and more precisely at the term involving $R_j$ (defined in \eqref{eq.defRj}) in  \eqref{eq.Qj}. Since $\mathcal{B}_3$ vanishes linearly on $r=0$ and that $s$ is the coordinate along $r=0$, we get, by using the size of $\tilde{\Omega}_{\delta_1, \delta_2}$, that
\[
q_{rs,j} = \ooo\sep{h^{\frac14-\eta}}
\]
Thus, we can deal with the term $q_{rs,j}(r-r_j)(s-s_j)$ by using the size of the balls:
\[Q_j(\tilde\psi_j)\geq (1-h^{\frac13})Q^0_j(\tilde\psi_j)-
\underbrace{Ch^{-\frac13}h^{\frac43}h^{\frac12-2\eta}}_{=o(h^{\frac43})  \textrm{ since } \eta \leq \frac{1}{12}}
\|\tilde\psi_j\|^2\,,\]
where
\begin{multline}  \label{eq.defQ0}
Q^0_j(\psi)=
\alpha_j^{\frac12}\int_{\tilde{\Omega}_{\delta_1, \delta_2}}\Big(|h\partial_t\psi|^2+|(-ih\partial_r-t\sin\phi_j\cos\theta_j)\psi|^2\\
+|(-ih\alpha^{-\frac12}_j\partial_s+t\cos\phi_j\cos\theta_j-(r-r_j)\sin\theta_j-\alpha^{-\frac12}_jq_{rr,j}(r-r_j)^2)\psi|^2\Big)\dd y\,.
\end{multline}
Thus, from \eqref{eq.synth2}, we get the following improved inequality:
\begin{equation}\label{eq.localization}
\sum_{j\in I}ch\|\psi_j\|^2+\sum_{j\in J}(Q^0_j(\tilde\psi_j)-\Theta_0h\|\tilde\psi_j\|^2-\tilde C h^{\frac43}\|\tilde\psi_j\|^2)\leq 0\,.
\end{equation}

Let us now have a look at the term $Q^0_j(\tilde\psi_j)$. To this end, we introduce
$R >0$ sufficiently large, and the following splitting
\begin{equation} \label{eq.splitJ}
J_\leq = \{j \in J : |r_j| \leq Rh^{\frac13}\},  \qquad J_\geq = \{j \in J : |r_j| > Rh^{\frac13}\}
\end{equation}
depending on the geodesic distance to $\Gamma$ of the point $y_j$. Recall here that the
size of the balls in the partition $(\chi_j)_j$ is independent of $R$.

The point is to get a convenient lower bound on $Q^0_j(\tilde\psi_j)$ depending on $j\in J$. Lemma \ref{lem.qj}, whose statement and proof are postponed to the end of this section (to avoid interrupting the proof and help seeing how its gives the conclusion), provides us with such a lower bound.

We first get that for the $j \in J_\geq$   we have
\[
Q^0_j(\tilde\psi_j)-\Theta_0 h\|\tilde\psi_j\|^2\geq \tilde c_0 Rh^{\frac43}\|\tilde\psi_j\|^2\,.
\]
Besides, from Lemma \ref{lem.qj} again, we have, for all $j \in J$,
\[
Q^0_j(\tilde\psi_j)-\Theta_0 h\|\tilde\psi_j\|^2 \geq -C_Rh^{\frac43}\|\tilde\psi_j\|^2\,.
\]
Taking $R$ large enough, we can then write from \eqref{eq.localization} that
\begin{equation*}\label{eq.localization2+}
\sum_{j\in I}ch\|\psi_j\|^2+\sum_{j \in J_\geq}\tilde{c}_0 \frac{R}{2} h^{\frac43}\|\tilde\psi_j\|^2\leq\\ \sum_{j\in J_\leq}\tilde C_R h^{\frac43}\|\tilde\psi_j\|^2\,,
\end{equation*}
using that $h^{\frac43} \leq h$. Forgetting the dependence on $R$ in the constants, we get that there is constant $C$ such that
\begin{equation*}\label{eq.localization2+}
\sum_{j\in I}ch\|\psi_j\|^2+\sum_{j \in J_\geq} \|\tilde\psi_j\|^2  +\sum_{j \in J_\leq} \|\tilde\psi_j\|^2\leq \sum_{j\in J_\leq}C \|\tilde\psi_j\|^2\,,
\end{equation*}
Now since for indices $j \in J_\leq$ and the fact that  $\psi = e^{r/h^{\frac13}} \phi$, we get that
\[
\sum_{j\in J_\leq}C \|\tilde\psi_j\|^2\, \leq C' \norm{\phi}^2
\]
so that with \eqref{eq.localization2+} we get
\[
\norm{\psi}^2 \leq C'' \norm{\phi}^2\,.
\]
This proves the (first) Agmon inequality in Proposition \ref{prop.optimal-r}.
The second inequality in Proposition \ref{prop.optimal-r} is a direct consequence of the first one and \eqref{eq.eqvp}.

\end{proof}
As we just saw, the proof of Proposition \ref{prop.optimal-r} will be complete once the following lemma is established.  We keep using the notations of the preceding proof. As we shall see, the proof is quite delicate and uses many changes of variables, of gauge, as well as commutators in order to reveal some hidden ellipticity related to model operators.

\begin{lemma} \label{lem.qj}
Let  $0<\eta\leq \frac{1}{12}$. Then there exists $\tilde{c}_0>0$ such that, for all $R$ sufficiently large, there exists a constant $C_R$ such that for all $\psi$ smooth and supported in balls of index $j \in J_\geq$, \ie such that $|r_j|\geq Rh^{\frac13}$ we have
\[
Q^0_j(\psi)-\Theta_0 h\|\psi\|^2\geq \tilde c_0 Rh^{\frac43}\|\psi\|^2\, %- Ch^{\frac43}\|\tilde{\underline{\chi}}_{1,h}\tilde\psi_j\|^2.
\]
Besides for functions supported in balls of index $j\in J$, we have
\[
Q^0_j(\psi) - \Theta_0 h\|\psi\|^2\geq -C_R h^{\frac43}\|\psi\|^2\,.
\]
\end{lemma}

\begin{proof}
With a first rescaling $\alpha_j^{1/2}  s \mapsto s$ (and keeping the same letter for the variable $s$ as well as for the domain $\tilde{\Omega}_{\delta_1, \delta_2}$ after this change), we have $Q^0_j(\psi)=\check Q^0_j(\check\psi)$ with
\begin{multline*}
\check Q^0_j(\check\psi)=
\int_{ \tilde{\Omega}_{\delta_1, \delta_2} }\Big(|h\partial_t\check\psi|^2+|(-ih\partial_r-t\sin\phi_j\cos\theta_j)\check\psi|^2\\
+|(-ih\partial_s+t\cos\phi_j\cos\theta_j-(r-r_j)\sin\theta_j-\frac{a_j}{2}(r-r_j)^2)\check\psi|^2\Big)\dd y\,,\quad 2a_j=\alpha^{-\frac12}_jq_{rr,j}\,.
\end{multline*}
From the definition of $q_{rr,j}$ (see \eqref{eq.defRj} and above) and Assumption \ref{hyp.genericcancellation}, we have  $\D_r \bbb_3 >0$ so that $a_j$ is uniformly positive. This allows the following algebraic simple computation
\[\frac{a_j}{2}(r-r_j)^2+(r-r_j)\sin\theta_j=\frac{a_j}{2}\left(r-r_j+a_j^{-1}\sin\theta_j\right)^2-\frac{\sin^2\theta_j}{2a_j}\,.\]
This suggests the change of variable
\begin{equation}\label{eq.translation}
\tilde r=r-r_j+a_j^{-1}\sin\theta_j\,.
\end{equation}
We observe that, there exists $c>0$ such that, for all the $j$ such that $|r_j|\geq Rh^{\frac13}$, we have
\[|a_j^{-1}\sin\theta_j|\geq cRh^{\frac13}\,.\]
In particular, since the balls of the partition have common radius $h^{\frac13}$ independent of $R$, we get that for $R>0$ large enough, on the support of $\psi$ (or of $\check\psi$ which has the same scale), we have
\[|\tilde r|\geq\frac{cRh^{\frac13}}{2}\,.\]
This will be of crucial use later.
Let us use the translation \eqref{eq.translation} and a local change of gauge associated to the conjugation with $e^{is \frac{\sin^2\theta_j}{2a_j}}$ to remove the constant $-\frac{\sin^2\theta_j}{2a_j}$. We get that
\[\check Q^0_j(\check \psi)=\check Q_j^{\mathrm{Tr}}(\psi^{\mathrm{Tr}}_j)\,,\]
with
\begin{multline}\label{eq.modeloperator}
\check Q^{\mathrm{Tr}}_j(\psi^{\mathrm{Tr}})=
\int_{\tilde{\Omega}_{\delta_1, \delta_2} }\Big(|h\partial_t\psi|^2+|(-ih\partial_r-t\sin\phi_j\cos\theta_j)\psi^{\mathrm{Tr}}|^2\\
+|(-ih\partial_s+t\cos\phi_j\cos\theta_j-\frac{a_j}{2}\tilde r^2)\psi^{\mathrm{Tr}}|^2\Big)\dd y\,,
\end{multline}
where we keep the notation $\tilde{\Omega}_{\delta_1, \delta_2}$ and the measure $\dd y$ since the new ones have the same properties as the original ones. In the same spirit, we write from now on $r$ instead of $\tilde{r}$.

\bigskip
It appears that  the quadratic form $\check Q^{\mathrm{Tr}}_j$ can be rewritten with the help of the de Gennes operator. For exhibiting this property we perform some algebraic transformations. We first let
\[
 P_t = -ih\partial_t\,,  \qquad P_r=-ih\partial_r\,, \qquad P_s=-ih\partial_s-\frac{a_j}{2}  r^2 ,
\]
and we notice that operator $\check L_j^{\mathrm{Tr}}$ associated to the form $\check Q^{\mathrm{Tr}}_j$ can be written as
\begin{equation*}
\begin{split}
\check L_j^{\mathrm{Tr}}
&:= P_t^2 +(P_r-t\sin\phi_j\cos\theta_j)^2+\left(P_s+t\cos\phi_j\cos\theta_j\right)^2\\
&=P_t^2 + P_r^2+P_s^2+t^2\cos^2\theta_j-2t\sin\phi_j\cos\theta_jP_r+2t\cos\phi_j\cos\theta_j P_s\\
&=P_t^2 + P_r^2+P_s^2+\cos^2\theta_j\left(t-\frac{\sin\phi_j}{\cos\theta_j}P_r+\frac{\cos\phi_j}{\cos\theta_j}P_s\right)^2-(\sin\phi_j P_r-\cos\phi_j P_s)^2\\
&=P_t^2 + \cos^2\phi_jP_r^2+\sin^2\phi_jP_s^2+\cos\phi_j\sin\phi_j(P_rP_s+P_sP_r)\\
& \qquad \qquad \qquad +\cos^2\theta_j\left(t-\frac{\sin\phi_j}{\cos\theta_j}P_r+\frac{\cos\phi_j}{\cos\theta_j}P_s\right)^2\\
&=P_t^2 + (\cos\phi_j P_r+\sin\phi_j P_s)^2+\left(t\cos\theta_j-\sin\phi_jP_r+\cos\phi_jP_s\right)^2\,.
\end{split}
\end{equation*}
Let us first assume that $\sin\phi_j\neq 0$ and use now the notation $p_\sharp=-ih\partial_\sharp$. This allows to consider a new change of gauge:
\begin{equation}  \label{eq.sinphizero}
\begin{split}
&e^{\frac{ia r^3\cos\phi_j}{6h\sin\phi_j}}\left[(P_r-t\sin\phi_j\cos\theta_j)^2+\left(P_s+t\cos\phi_j\cos\theta_j\right)^2\right]e^{-\frac{ia r^3\cos\phi_j}{6h\sin\phi_j}}\\
&=\left(\cos\phi_j p_r+\sin\phi_j p_s-a\frac{r^2}{2}\frac{\cos^2\phi_j}{\sin\phi_j}-a\frac{r^2}{2}\sin\phi_j\right)^2+\left(t\cos\theta_j-\sin\phi_jp_r+\cos\phi_jp_s\right)^2\\
&=\left(\cos\phi_j p_r+\sin\phi_j p_s-\frac{ar^2}{2\sin\phi_j}\right)^2+\left(t\cos\theta_j-\sin\phi_jp_r+\cos\phi_jp_s\right)^2.
\end{split}
\end{equation}
The following change of variables (a rotation) appears naturally
\[r=u\cos\phi_j +v\sin\phi_j \,,\quad s=u\sin\phi_j-v\cos\phi_j\,,\]
so that on the dual side
\[\partial_u=\cos\phi_j\partial_r+\sin\phi_j\partial_s\,,\quad \partial_v=\sin\phi_j\partial_r-\cos\phi_j\partial_s\,.\]
Using these changes of variable and of gauge, we see that the operator $\check L_j^{\mathrm{Tr}}$ is unitarily equivalent to the following operator in variables $(t,u,v)$:
\[
p_t^2 + \left(p_u-\frac{ar(u,v)^2}{2\sin\phi_j}\right)^2 +\left(t\cos\theta_j-p_v\right)^2\,,
\]
where now $r$ is a function of $(u,v)$.
 It will be useful to notice for further use that
\begin{equation} \label{eq.comruv}
\left[\partial_v,\frac{ar(u,v)^2}{2\sin\phi_j}\right]=-a r(u,v)\,.
\end{equation}
which is of absolute value of order at least $Rh^{\frac13}$ on the support of involved functions.
Let us consider then the associated quadratic form
\begin{equation}\label{eq.Q}
Q(\psi)= \left\|\left(p_u-\frac{ar(u,v)^2}{2\sin\phi_j}\right)\psi\right\|^2+\|p_t\psi\|^2
+\|\left(t\cos\theta_j-p_v\right)\psi\|^2\,.
\end{equation}
We observe that the last two terms are related to the de Gennes operator. To exhibit it, we first do a change of variable and introduce a new temporary semiclassical parameter
\[
\hat{t} = \hat{h}^{-1} t, \qquad  \partial_{\hat{t}} = \hat{h} \partial_t, \qquad     \hat{h}= \sep{\frac{h}{\cos\theta_j}}^{1/2}\,,
\]
which allows to write
\[
Q(\psi) =
\left\|\left(p_u-\frac{ar(u,v)^2}{2\sin\phi_j}\right)\psi\right\|^2 + h\cos\theta_j \sep{ \|\partial_{\hat{t}}\psi\|^2 + \|(\hat{t}+i\hat{h} \partial_v)\psi\|^2}.
\]
Performing a $\hat{h}$-partial Fourier transform in variable $v$, and denoting $\hat\psi$ the corresponding Fourier transform of $\psi$ in the last two terms gives then
\[
Q(\psi) =
\left\|\left(p_u-\frac{ar(u,v)^2}{2\sin\phi_j}\right)\psi\right\|^2 + h\cos\theta_j \sep{ \|\partial_{\hat{t}}\hat{\psi}\|^2 +  \|(\hat{t}-\xi)\hat\psi\|^2},
\]
so that
\[Q(\psi)\geq \left\|\left(p_u-\frac{ar(u,v)^2}{2\sin\phi_j}\right)\psi\right\|^2
+h\cos\theta_j\left\langle\mu_1^{\mathrm{dG}}
\left(\xi\right)\hat\psi,\hat\psi\right\rangle\,.\]
Let us consider a quadratic partition of the unity $\Xi_1^2+\Xi_2^2=1$ such that $\Xi_1=1$ near $\xi_0$. Then from the properties of $\mu_1^{\mathrm{dG}}$, there exists $c>0$ such that
\[\mu_1^{\mathrm{dG}}(\xi)-\Theta_0\geq c\sep{(\xi-\xi_0)^2\Xi^2_1(\xi)+\Xi^2_2(\xi)}\,.\]
It follows that
\begin{equation}\label{eq.Q>}
Q(\psi)\geq \left\|\left(p_u-\frac{ar(u,v)^2}{2\sin\phi_j}\right)\psi\right\|^2+h\Theta_0\cos\theta_j\|\psi\|^2
+ ch\|(\xi-\xi_0)\Xi_1\hat\psi\|^2+ch\|\Xi_2\hat\psi\|^2\,.
\end{equation}
Let us now study $ P_u:=p_u-\frac{ar(u,v)^2}{2\sin\phi_j}$. Denoting $\Xi_\sharp^w$ the fourier multiplier associated to $\Xi_\sharp$, we have the localization formula (see, for instance, \cite[Prop. 4.8]{Ray})
\[
\|P_u \psi\|^2= \sum_{j=1,2}\left(\|P_u(\Xi^w_j\psi)\|^2-\|[P_u,\Xi^w_j]\psi\|^2+\Re\langle P_u \psi,[[P_u,\Xi^w_j],\Xi^w_j]\psi\rangle\right)\,.
\]
This gives
\begin{multline*}
 Q(\psi)-\Theta_0 h\cos\theta_j\|\psi\|^2\geq\|P_u\Xi_1^w\psi\|^2+ ch\|(\xi-\xi_0)\Xi_1\hat\psi\|^2+ ch\|\Xi_2\hat\psi\|^2\\ 
-\sum_{j=1,2}\|[P_u,\Xi^w_j]\psi\|^2-|\langle P_u \psi,[[P_u,\Xi^w_j],\Xi^w_j]\psi\rangle|\,.
\end{multline*} 
 From \eqref{eq.comruv} we have exact estimates on the commutators, and remembering that the Fourier transform in $v$ is at the scale $\hat{h} \sim h^{\frac12}$, we find 
\begin{multline*}
 Q(\psi)-\Theta_0 h\cos\theta_j\|\psi\|^2\geq\|P_u\Xi_1^w\psi\|^2+ ch\|(\xi-\xi_0)\Xi_1\hat\psi\|^2+ ch\|\Xi_2\hat\psi\|^2\\ 
-Ch\sum_{j=1,2}\left(\|r (\Xi'_j)^w \psi\|^2+h\|(\Xi''_j)^w\psi\|^2+\|[r(\Xi'_j)^w,\Xi^w_j]\psi\|^2\right)\,.
\end{multline*}
Using a commutator between $r$ and the derivatives of $\Xi_j^w$ we get that
\begin{multline*}
Q(\psi)-\Theta_0 h\cos\theta_j\|\psi\|^2\geq\|P_u\Xi_1^w\psi\|^2+ ch\|(\xi-\xi_0)\Xi_1\hat\psi\|^2+ ch\|\Xi_2\hat\psi\|^2\\
-Ch\| r \psi\|^2-Ch^2\|\psi\|^2.
\end{multline*}
This inequality is applied to functions $\psi$ supported in $\{|r|\leq Ch^{\frac14-\eta}\}$ and thus, for $\eta$ small enough ($\eta<\frac{1}{12}$ is indeed sufficient), we get
\begin{multline*}
Q(\psi)-\Theta_0 h\cos\theta_j\|\psi\|^2\geq  \|P_u\Xi_1^w\psi\|^2+ ch\|(-i\hat{h} \partial_{v}-\xi_0)\Xi_1^w\psi\|^2+ ch\|\Xi_2^w\psi\|^2
-Ch^{\frac43}\|\psi\|^2\,.
\end{multline*}
Using the commutation property \eqref{eq.comruv} together with the abstract operator inequality
$A^2 + B^2 \geq \pm i [A,B]$ applied to the first two terms, we get then
\begin{equation*}
Q(\psi)-\Theta_0 h\cos\theta_j\|\psi\|^2\geq \pm \tilde c h  \langle r\Xi_1^w\psi, \Xi_1^w\psi\rangle+ ch\|\Xi_2^w\psi\|^2 -Ch^{\frac43}\|\psi\|^2\,.
\end{equation*}
Let us now introduce a partition of the unity $\chi_\leq^2 + \chi_\geq^2 = 1$ with $\mathrm{supp}\chi_\leq\subset\{|r|\leq \frac{R}{2}h^{\frac13}\}$.
We get that
\begin{equation*}
Q(\psi)-\Theta_0 h\cos\theta_j\|\psi\|^2\geq \bar c R h^{\frac43}  \|\chi_\geq \Xi_1^w \psi\|^2+ ch\|\Xi_2^w\psi\|^2 -Ch^{\frac43}\|\psi\|^2\,. 
\end{equation*}
Now observe that due to the fact that $\psi$ is supported in balls of index  $j \in J_\geq$, \ie such that $|r_j|\geq Rh^{\frac13}$ we have for $R$ sufficiently large that
\[
\|\chi_\leq \Xi_1^w \psi\|^2 \leq C_N h^N \|\psi\|^2, \qquad \forall N\in \N,
\]
with constants $C_N$ independent of $R\geq 1$. This allows to write that
\begin{equation*}
Q(\psi)-\Theta_0 h\cos\theta_j\|\psi\|^2\geq \tilde c R h^{\frac43} \|\Xi_1^w \psi\|^2+ ch\|\Xi_2^w\psi\|^2 -Ch^{\frac43}\|\psi\|^2\,. 
\end{equation*}
So for all index $j$ such that $\sin \phi_j \neq 0$, using $h^{\frac43} = o(h)$ and assuming $h$ small enough, we get
\begin{equation*}
Q(\psi)-\Theta_0 h\cos\theta_j\|\psi\|^2\geq \tilde c R h^{\frac43} \|\psi\|^2 -C h^{\frac43} \|\psi\|^2\,. 
\end{equation*}
Using that $|\theta_j| \leq C h^{\frac14-\eta}$, we get $\cos \theta_j \geq 1- C h^{\frac12-2\eta}$ which gives for $R$ large enough
\begin{equation*}
Q(\psi)-\Theta_0 h\|\psi\|^2\geq \tilde c R h^{\frac43} \|\psi\|^2 -C' h^{\frac43} \|\psi\|^2\,\geq \tilde{c}_0 R h^{\frac43}  \|\psi\|^2 \,.
\end{equation*}
This gives
\[
Q^0_j(\tilde\psi_j)-\Theta_0 h\|\tilde\psi_j\|^2\geq  \tilde{c}_0 Rh^{\frac43}\|\tilde\psi_j\|^2\,,
\]
for all $j \in J_\geq$, \ie such that $|r_j|\geq Rh^{\frac13}$ and such that $\sin\phi_j\neq 0$.
The case when $\sin\phi_j =  0$ is easier since we do not need the change of gauge in \eqref{eq.sinphizero} and the rotation procedure, and we skip the proof. This is the first
inequality in Lemma \ref{lem.qj}. The second inequality is much easier since it does not use the gain provided by \eqref{eq.comruv} on indices $j \in J_\geq$,  and we also skip its proof. The proof of Lemma \ref{lem.qj} is complete.
\end{proof}
\begin{remark}\label{rem.optimalrt}
In the coordinates $(r,s,t)$, the localization estimates can be written as follows. In terms of the eigenfunctions $\varphi$ of $\widetilde{\mathscr{L}}_{h,\delta_1,\delta_2}$, associated with $\lambda\leq\Theta_0h+Ch^{\frac43}$, we have
\[\int_{0<t<\delta_1\,,\\|r|<\delta_2}e^{2\alpha t/h^{1/2}}|\varphi|^2\dd r\dd s\dd t\leq C\|\varphi\|^2\,,\]
and
\[\int_{0<t<\delta_1\,,\\|r|<\delta_2}e^{2|r|/h^{1/3}}|\varphi|^2\dd r\dd s\dd t\leq C\|\varphi\|^2\,.\]
\end{remark}

%%%%%%%%%%%%%%%%%%%%%%%%%%%%%%%%%%%%%%%%%%%%%%%%%%%%%%%%%%%%%%%%%%%%%%%%%%%%%
%%%%%%%%%%%%%%%%%%%%%%%%%%%%%%%%%%%%%%%%%%%%%%%%%%%%%%%%%%%%%%%%%%%%%%%%%%%%%%%
%%%%%%%%%%%%%%%%%%%%%%%%%%%%%%%%%%%%%%%%%%%%%%%%%%%%%%%%%%%%%%%%%%%%%%%%%%%%
%%%%%%%%%%%%%%%%%%%%%%%%%%%%%%%%%%%%%%%%%%%%%%%%%%%%%%%%%%%%%%%%%%%%%%%%%%%%%%
%%%%%%%%%%%%%%%%%%%%%%%%%%%%%%%%%%%%%%%%%%%%%%%%%%%%%%%%%%%%%%%%%%%%%%%%%%%%
%%%%%%%%%%%%%%%%%%%%%%%%%%%%%%%%%%%%%%%%%%%%%%%%%%%%%%%%%%%%%%%%%%%%%%%%%%%%%%
\section{A new operator near $\Gamma$}\label{sec.4}

Propositions \ref{prop.loct} and \ref{prop.optimal-r} tell us that the eigenfunctions are localized near $t=0$ (at the scale $h^{\frac12}$) and near $r=0$ (at the scale $h^{\frac13}$).
We shall use these two results to perform refined approximations of operator $\widetilde{\mathscr{L}}_{h,\delta_1,\delta_2}$ involving Taylor expansions, as well as other modifications (change of gauge, inserting cutoff functions, and eventually rescaling). The final result will be used in Section \ref{sec.parametrix} for solving the eigenvalues problem.

\subsection{Taylor expansions and a change of gauge}  \label{sec.nicepseudo}
Enlightened by the localization properties in Propositions \ref{prop.loct} and \ref{prop.optimal-r}, we first perform a Taylor expansion of the vector potential near $r=t=0$ (\emph{i.e.}, geometrically speaking, near $\Gamma$). It is given in the following lemma, whose proof is a straightforward computation using the expression of $\tilde{\mathbf{A}}$ in Remark \ref{rem.tildeA}, where we recall $\seq{f}=h\mathfrak{f}_h$.
\begin{lemma}\label{lem.approxmodel}
We have
\[\begin{split}
\tilde A_1(r,s,t) =& [|g|^{\frac12}\mathcal{B}_2](0,s,0)t+\partial_1[|g|^{\frac12}\mathcal{B}_2](0,s,0)rt
+\frac12\partial_3[|g|^{\frac12}\mathcal{B}_2](0,s,0)t^2\\
&+\frac12\partial^2_1[|g|^{\frac12}\mathcal{B}_2](0,s,0)r^2t
+\mathscr{O}(|r|^3t+|r|t^2+t^3)\,,
\end{split}\]
and
\[\begin{split}
\tilde A_2(r,s,t)
= & h\mathfrak{f}_h-[|g|^{\frac12}\mathcal{B}_1](0,s,0)t
-\frac12\partial_3[|g|^{\frac12}\mathcal{B}_1](0,s,0)t^2 \\
& +\left(\partial_2[|g|^{\frac12}\mathcal{B}_2]
    +\partial_3[|g|^{\frac12}\mathcal{B}_3]\right)(0,s,0)rt
+\frac12\partial_1[|g|^{\frac12}\mathcal{B}_3](0,s,0)r^2\\
& +\frac16\partial_1^2[|g|^{\frac12}\mathcal{B}_3](0,s,0)r^3
+\frac12\left(\partial_1\partial_3[|g|^{\frac12}\mathcal{B}_3]
+\partial_1\partial_2[|g|^{\frac12}\mathcal{B}_2]\right)(0,s,0)r^2t  \\
& +\mathscr{O}(r^4+|r|^3t+|r|t^2+t^3)\,.
\end{split}\]	
\end{lemma}
Therefore, using Definition \ref{def.phithetars},
\[
[|g|^{\frac12}\mathcal{B}_2](0,s,0) = \sin \phi(s), \qquad [|g|^{\frac12}\mathcal{B}_1](0,s,0) = \cos \phi(s), \qquad |g|^{\frac12}(0,s,0) = \alpha(s) = 1
\]
and recalling the notation
$$
\beta(s)=\partial_1[|g|^{\frac12}\mathcal{B}_3](0,s,0)
$$
we get, up to some remainders that we shall control later, a model vector potential, polynomial with respect to $(r,t)$, defined by
\begin{eqnarray}\label{eq.tildeAm}
%\begin{split}
&\widetilde{A}^{\mathrm{m}}_1(r,s,t)=t\sin\phi(s)+u_1(s)rt+v_1(s)r^2t+w_1(s) t^2\,, \nonumber \\
& \widetilde{A}^{\mathrm{m}}_2(r,s,t)=h\mathfrak{f}_h-t\cos\phi(s)+\beta(s)\frac{r^2}{2}+u_2(s) rt+v_2(s)r^2t+w_2(s)t^2+d(s)r^3\, \label{eq.Am}
%\end{split}
\end{eqnarray}
where we do not give the exact expression of $d(s)$ and the other terms are defined as follows
\begin{equation}\label{eq.u1u2}
\begin{array}{ll}
u_1(s)=\partial_1[|g|^{\frac12}\mathcal{B}_2](0,s,0), &u_2(s)=\left(\partial_2[|g|^{\frac12}\mathcal{B}_2]
 +\partial_3[|g|^{\frac12}\mathcal{B}_3]\right)(0,s,0), \\
v_1(s)=\frac12\partial^2_1[|g|^{\frac12}\mathcal{B}_2](0,s,0),
& v_2(s)=\frac12\left(\partial_1\partial_3[|g|^{\frac12}\mathcal{B}_3]
+\partial_1\partial_2[|g|^{\frac12}\mathcal{B}_2]\right)(0,s,0), \\
 w_1(s)=\frac12\partial_3[|g|^{\frac12}\mathcal{B}_2](0,s,0), \quad
& w_2(s)=-\frac12\partial_3[|g|^{\frac12}\mathcal{B}_1](0,s,0)\,.	
\end{array}
\end{equation}
 Note that $\beta(s)=\partial_1\mathcal{B}_3(0,s,0)$
since $|g|(0,s,0)=\alpha(0,s)=1$ and $\mathcal{B}_3(0,s,0)=0$. Recalling that $\mathcal{B}_3(r,s,0)=-\mathbf{n}(\gamma(r,s))\cdot\mathbf{e}_3$ and Assumption \ref{hyp.genericcancellation}, we get that
\[
\beta(s)\geq c_0>0\,.
\]
\begin{remark}
When $\phi\equiv 0$, \emph{i.e.}, when $\Gamma$ lies in a plane, there is a quite explicit description of $\beta$. Indeed, in this case, $\mathbf{e}_3=-\partial_r\gamma(0,s)$ so that
\[\partial_r(-\mathbf{n}\cdot\mathbf{e}_3)=\dd\mathbf{n}(\partial_r\gamma,\partial_r\gamma)=\dd\mathbf{n}(\mathbf{e}_3,\mathbf{e}_3)\,,\]
which is positive when $\Omega$ is strictly convex. This is nothing but the curvature of $\partial\Omega$ in the direction of the magnetic field. The place where $\beta$ is minimal is where the magnetic field is the most tangent to the boundary.
\end{remark}

For reasons that will become clear later, we shall now perform a triple change of gauge.
The first one is associated with  the unique $2L$-periodic function $F$ of variable $s$ with
\[F'(s)=-\xi_0\cos\phi(s)+ F_0, \qquad F_0 = \frac{\xi_0}{2L}\int_0^{2L}\cos\phi(s)\dd s\,.
\]
The second one is associated to the function $G(r,s)$  with
\[
G(r,s) = r \xi_0 \sin \phi(s)\,.
\]
These changes will allow us later to center our problem at $\xi_0$, which is the  frequence naturally associated to the de Gennes operator.
Performing a last linear change of gauge associated with $H(s)$ where
\[
H(s) = F_0 - m \frac{2\pi}{2L} s
\]
for a suitable
$m \in \Z$ allow us to replace $h\mathfrak{f}_h - F_0$ by $h\tilde{\mathfrak{f}}_h $ with $\tilde{\mathfrak{f}}_h \in \left(-\frac{1}{2L},\frac{1}{2L}\right)$. The resulting operator is then
\begin{equation}   \label{eq.smalltildeLhd}
\tilde{\mathscr{L}}_h = e^{-i(F+G+H)/h^{\frac12}} \widetilde{\mathscr{L}}_{h,\delta_1,\delta_2}  e^{i(F+G+H)/h^{\frac12}}
\end{equation}
the spectrum of $\tilde{\mathscr{L}}_h $ and of $\widetilde{\mathscr{L}}_{h,\delta_1,\delta_2}$ are of course the same. Remembering \eqref{eq.tildeAm} (see also \eqref{eq.tildeLhd}), we can replace $\widetilde{A}^{\mathrm{m}}_1$ and $\widetilde{A}^{\mathrm{m}}_1$ by
\[\tilde{A}^{\mathrm{m}}_1(r,s,t)=(t-h^{\frac12}\xi_0)\sin\phi(s)+u_1(s)rt+v_1(s)r^2t+w_1(s) t^2\,,\]
and
\begin{multline}\tilde{A}^{\mathrm{m}}_2(r,s,t)=h\tilde{\mathfrak{f}}_h+(h^{\frac12}\xi_0-t)\cos\phi(s)
+\beta(s)\frac{r^2}{2}+u_2(s) rt \\
-  h^\frac12r\phi'\cos\phi  +v_2(s)r^2t+w_2(s)t^2+d(s)r^3\,.
\end{multline}

\subsection{Truncating variables in the right scales}

Propositions \ref{prop.loct} and \ref{prop.optimal-r} tell us that the eigenfunctions are localized near $t=0$ at the scale $h^{\frac12}$ and near $r=0$ at the scale $h^{\frac13}$. Jointly with the Taylor approximations, this leads to define the following approximation $\check{\mathscr{L}}^{\mathrm{m}}_h$ (and the associated quadratic form $\check{\mathscr{Q}}^{\mathrm{m}}_{h}$) of operator $\tilde{\mathscr{L}}_h$ in \eqref{eq.smalltildeLhd} (and therefore $\widetilde{\mathscr{L}}_{h,\delta_1,\delta_2}$ defined in \eqref{eq.tildeLhd}).

In order to build this operator,  we first introduce a smooth cutoff function $\chi_0$ equalling $1$ near $0$ and
remembering \eqref{eq.approxmetric}, we also introduce cutoff variables and the weight
\begin{equation}\label{eq.cutoffvariables}
\check r=r\chi_0(r/h^{1/3-\eta})\,,\quad \check t=t\chi_{0}(t/h^{1/2-\eta})\,,\quad a(r,s,t)=\alpha(\check r,s)^{\frac12}+k(\check r,s)\check t\,.
\end{equation}
Introducing then the new domain of integration \[\mathcal{V}=\mathbb{R}\times[0,2L)\times\R_+\,.\] 
We define then $\check{\mathscr{Q}}^{\mathrm{m}}_{h}$ as follows
\begin{multline*}
\check{\mathscr{Q}}^{\mathrm{m}}_{h}(\psi)=\int_{\mathcal{V}}\left( (1+\check tk_1(\check r,s))|(hD_r-\check A^{\mathrm{m}}_1)\psi|^2+ (\alpha(\check r,s)^{-1}+\check tk_2(\check r,s))|(hD_s-\check A^{\mathrm{m}}_2)\psi|^2\right)a\dd y\\
+\int_{\mathcal{V}} |hD_t\psi|^2a\dd y+2\Re\langle \check tk_{12} (hD_r-\check A^{\mathrm{m}}_1)\psi,(hD_s-\check A^{\mathrm{m}}_2)\psi\rangle_{L^2(a\dd y)}+\int_{\mathcal{V}}hV\left(\frac{r}{h^{\frac13-\frac\eta2}}\right)|\psi|^2a\dd y\,.
\end{multline*}
Here, $\check A^{\mathrm{m}}_1$ and $\check A^{\mathrm{m}}_2$ are defined as
\begin{multline} \label{eq.goodpotential}
 \qquad \qquad \check A^{\mathrm{m}}_1(r,s,t)=(t-h^\frac12\xi_0)\sin\phi(s)+u_1(s)\check r\check t+v_1(s)\check r^2\check t+w_1(s) \check t^2\,, \\
 \check A^{\mathrm{m}}_2(r,s,t)=h\mathfrak{f}_h-(t-h^\frac12\xi_0)\cos\phi(s)+\beta(s)\frac{\check r^2}{2}+u_2(s) \check r\check t-h^\frac12\check r \phi' \cos \phi \\ + v_2(s)\check r^2\check t+w_2(s)\check t^2+d(s)\check r^3\,
\end{multline}
and $V$ is a smooth even non-negative potential equal to $0$ in a neighborhood of $0$, equal to $1$ outside the unit ball and radially increasing. In the following we denote by $\check{\mathscr{L}}^{\mathrm{m}}_h$ the operator associated with $\check{\mathscr{Q}}^{\mathrm{m}}_{h}$.

Note that we do not put cutoff functions for the terms linear with respect to $t$, since we shall use later the de Gennes operator in variable $t$. The presence of $V$ is essentially artificial. Let us explain this. By inserting a cutoff function in the term involving $\beta$ (which is responsible for the localization with respect to $r$), we a priori lose the localization of the eigenfunctions of  $\check{\mathscr{L}}^{\mathrm{m}}_h$ with respect to $r$ at the scale $h^{\frac13}$. That is why we add a confining potential to keep this localization property. At the end of the analysis, all these cutoff functions will be removed: we introduce them here only in order to be able to use pseudodifferential tools in convenient classes of symbols.

We can now compare the spectra of $\widetilde{\mathscr{L}}_{h,\delta_1,\delta_2}$, $\tilde{\mathscr{L}}_{h}$ and $\check{\mathscr{L}}^{\mathrm{m}}_h$.

\begin{proposition}\label{prop.checkLm}
For all $n\geq 1$,
\[\lambda_n(\widetilde{\mathscr{L}}_{h,\delta_1,\delta_2}) =\lambda_n(\tilde{\mathscr{L}}_{h})= \lambda_n(\check{\mathscr{L}}^{\mathrm{m}}_h)+o(h^{\frac53})\,.\]	
\end{proposition}
\begin{proof}
As already noticed, the  first equality is a direct consequence of the change of gauge
in \eqref{eq.smalltildeLhd}. We only prove a lower bound for $\lambda_n(\tilde{\mathscr{L}}_{h})$, the lower bound following from quite similar considerations.	Let us consider
\[\mathscr{E}_N(h)=\underset{1\leq j\leq N}{\mathrm{span}} \chi_h\psi_{j,h}\,,\]
where $(\psi_{j,h})$ is an orthonormal family of eigenfunctions associated with the familly of eigenvalues $(\lambda_j( \tilde{\mathscr{L}}_{h}))$, and where $\chi_h=\chi_h(r,t)$ is a smooth cutoff function supported in $\{|r|< h^{1/3-}\}\cap\{0<t<h^{\frac12-}\}$. By using the Agmon estimates, we see that $\mathscr{E}_N(h)$ is of dimension $N$ for $h$ small enough and that, for all $\psi\in \mathscr{E}_N(h)$,
\[\tilde{\mathscr{Q}}_{h}(\psi)\leq \lambda_N( \tilde{\mathscr{L}}_{h}) \|\psi\|^2_{L^2(|g|^{\frac12}\dd y)}+\mathscr{O}(h^\infty)\|\psi\|^2\,.\]
Since $\lambda_N(\tilde{\mathscr{L}}_h)=\mathscr{O}(h)$, the Taylor expansion $a$ of $|g|^{\frac12}$ and Agmon estimates with respect to $t$ (see Proposition \ref{prop.loct} or Remark \ref{rem.optimalrt}) allow to replace $|g|^{\frac12}$ by $a$ and we get
\begin{equation}\label{eq.ubQtilde}
 \tilde{\mathscr{Q}}_{h}(\psi)\leq \lambda_N(\tilde{\mathscr{L}}_h) \|\psi\|^2_{L^2(a\dd y)}+Ch^2\|\psi\|^2_{L^2(a\dd y)}\,.
\end{equation}
Then, by approximating the terms of the metrics (see \eqref{eq.approxmetric}), we get
\begin{multline*}
 \tilde{\mathscr{Q}}_{h}(\psi)\geq
\int_{ \tilde{\Omega}_{\delta_1, \delta_2}}\left( (1+ tk_1( r,s))|(hD_r-\tilde A_1)\psi|^2+ (\alpha^{-1}+ tk_2( r,s))|(hD_s-\tilde A_2)\psi|^2\right)a\dd y\\
+\int_{ \tilde{\Omega}_{\delta_1, \delta_2}} |hD_t\psi|^2a\dd y+2\Re\langle  tk_{12} (hD_r- \tilde A_1)\psi,(hD_s- \tilde A_2)\psi\rangle_{L^2(a\dd y)}-Ch^2\|\psi\|^2_{L^2(a\dd y)}\,.
\end{multline*}
By approximating the vector potential (from Lemma \ref{lem.approxmodel} and noticing that the change of gauge \eqref{eq.smalltildeLhd} does not affect the error terms), we find that
\begin{multline*}
 \tilde{\mathscr{Q}}_{h}(\psi)  \geq
\int_{\tilde{\Omega}_{\delta_1, \delta_2}}\left( (1+ tk_1( r,s))|(hD_r- {\tilde{A}}^{\mathrm{m}}_1)\psi|^2+ (\alpha^{-1}+ tk_2( r,s))|(hD_s-{\tilde{A}}^{\mathrm{m}}_2)\psi|^2\right)a\dd y\\
+\int_{\tilde{\Omega}_{\delta_1, \delta_2}} |hD_t\psi|^2a\dd y+2\Re\langle  tk_{12} (hD_r- {\tilde{A} }^{\mathrm{m}}_1)\psi,(hD_s- {\tilde{A}}^{\mathrm{m}}_2)\psi\rangle_{L^2(a\dd y)}\\
-Ch^2\|\psi\|^2_{L^2(a\dd y)}-C\|(-ih\nabla-\mathbf{\tilde{A}}^{\mathrm{m}})\psi\|(\|r^3 t\psi\|+\|r t^2\psi\|+\|t^3\psi\|+\|r^4\psi\|)\,.
\end{multline*}
By using the Agmon estimates with respect to $r$ and $t$ (see Remark \ref{rem.optimalrt}), and
the fact that $\|(-ih\nabla-\mathbf{{\tilde{A}}}^{\mathrm{m}})\psi\|\leq Ch^{\frac12}\|\psi\|$, we get
\begin{multline*}
\tilde{\mathscr{Q}}_{h}(\psi) \geq
\int_{ \tilde{\Omega}_{\delta_1, \delta_2} }\left( (1+ tk_1( r,s))|(hD_r- {\tilde{A}}^{\mathrm{m}}_1)\psi|^2+ (\alpha^{-1}+ tk_2( r,s))|(hD_s-{\tilde{A}}^{\mathrm{m}}_2)\psi|^2\right)a\dd y\\
+\int_{ \tilde{\Omega}_{\delta_1, \delta_2}} |hD_t\psi|^2a\dd y+2\Re\langle  tk_{12} (hD_r- {\tilde{A}}^{\mathrm{m}}_1)\psi,(hD_s- {\tilde{A}}^{\mathrm{m}}_2)\psi\rangle_{L^2(a\dd y)}\\
-Ch^{\frac{11}{6}}\|\psi\|^2_{L^2(a\dd y)}\,.
\end{multline*}
Then, we can insert cutoff functions in the coefficients of the metrics and of the vector potential (up to exponentially small remainders), and we infer that, for all $\psi\in\mathscr{E}_N(h)$,
\[ \check{\mathscr{Q}}^{\mathrm{m}}_{h}(\psi)-Ch^{\frac{11}{6}}\|\psi\|^2_{L^2(a\dd y)}\leq \tilde{\mathscr{Q}}_{h}(\psi)\,.\]
With \eqref{eq.ubQtilde}, this gives
\[ \check{\mathscr{Q}}^{\mathrm{m}}_{h}(\psi)-Ch^{\frac{11}{6}}\|\psi\|^2_{L^2(a\dd y)}\leq \lambda_N( \tilde{\mathscr{L}}_{h}) \|\psi\|^2_{L^2(a\dd y)}\,.\]
The min-max theorem implies that
\[\lambda_N(\check{\mathscr{L}}^{\mathrm{m}}_h)\leq \lambda_N(\tilde{\mathscr{L}}_{h})+Ch^{\frac{11}{6}}\,\]
and the lower bounds in the statement follows since $h^{\frac{11}{6}} = o(h^{\frac{5}{3}})$.

The converse inequality of the proposition can be obtained by similar arguments since the eigenfunctions associated with $\check{\mathscr{L}}^{\mathrm{m}}_h$ are exponentially localized at the scales $h^{\frac12}$ and $h^{\frac13-\eta}$ with respect to $t$ and $r$, respectively, as could be shown following exactly the same procedure leading to Remark \ref{rem.optimalrt}.
\end{proof}

\subsection{A rescaling and first pseudodifferential properties}\label{sec.rescalingcheckbreve}
We can now focus on the spectral analysis of $\check{\mathscr{L}}^{\mathrm{m}}_h$. Actually, it will be convenient to consider a rescaled version of $\check{\mathscr{L}}^{\mathrm{m}}_h$. We let
\[t=h^{\frac12}\breve{t}\,,\quad r=h^{\frac13}\breve{r}\,,\quad \hbar=h^{\frac16}\,,\quad \mu=\hbar^2\,,\]
and we divide the operator by $h$. This gives a new operator $\mathscr{N}_\hbar$ (which will be at the center of our coming analysis) given by
\begin{multline*}
 \mathscr{N}_\hbar = a^{-1}_\hbar(\hbar D_r-\breve{A}^{\mathrm{m}}_1)a_\hbar a_{\hbar,1}(\hbar D_r-\breve{A}^{\mathrm{m}}_1)+a^{-1}_\hbar(\mu\hbar D_s-\breve{A}^{\mathrm{m}}_2)a_\hbar a_{\hbar,2}(\mu\hbar D_s-\breve{A}^{\mathrm{m}}_2) \\
+a^{-1}_\hbar D_{\breve t}a_\hbar D_{\breve t}
+V(h^{\frac\eta2}\breve r)
+\hbar^3\breve{t}\chi_{h,3}a_{\hbar}^{-1}\Big[(\hbar D_r-\breve{A}^{\mathrm{m}}_1)a_\hbar k_{12}(\hbar^2\breve{r}\chi_{h,1},s)(\hbar^2\mu D_s-\breve{A}^{\mathrm{m}}_2)\\
+(\hbar D_r-\breve{A}^{\mathrm{m}}_2)a_\hbar k_{12}(\hbar^2\breve{r}\chi_{h,1},s)(\hbar^2\mu D_s-\breve{A}^{\mathrm{m}}_1)\Big]
\,,
\end{multline*}
where
\[\begin{split}
a_\hbar(\breve{r},s,\breve{t})&=\alpha^{\frac12}(\hbar^2\breve{r}\chi_{h,1},s)+\hbar^3\chi_{h,3}\breve{t}k(\hbar^2\breve{r}\chi_{h,1},s)\,,\\ a_{\hbar,1}(\breve{r},s,\breve{t})&=1+\hbar^3\breve{t}\chi_{h,3}k_1(\hbar^2\breve{r}\chi_{h,1},s)\,,\\ a_{\hbar,2}(\breve{r},s,\breve{t})&=\alpha^{-1}(\hbar^2\breve{r}\chi_{h,1},s)+\hbar^3\chi_{h,3}\breve{t}k_2(\hbar^2\breve{r}\chi_{h,1},s)\,, \end{split}\]
and
\[\breve{A}^{\mathrm{m}}_1= (\breve t-\xi_0) \sin\phi(s)+u_1(s)\hbar^2\chi_{h,1}\chi_{h,3}\breve r\breve t+\hbar^4v_1(s)\chi^2_{h,1}\chi_{h,3}\breve r^2\breve t+\hbar^3w_1(s)\chi_{h,3}\breve t^2\,,\]
\begin{multline*}
\breve A^{\mathrm{m}}_2=\hbar^3\mathfrak{f}_h- (\breve t-\xi_0) \cos\phi(s)+\mu^{\frac12}\beta(s)\chi^2_{h,1}\frac{\breve r^2}{2}+u_2(s)\chi_{h,1}\chi_{h,3} \hbar^2\breve r\breve t - \hbar^2 \chi_{h,1}\breve r \phi' \cos\phi \\ + v_2(s)\hbar^4\chi^2_{h,1}\chi_{h,3}\breve r^2\breve t
+w_2(s)\hbar^3\chi^2_{h,3}\breve t^2
+d(s)\hbar^3\chi^2_{h,1}\breve r^3\,.
\end{multline*}
In the lines above, the cutoff functions $\chi_{h,j}$ (at the scale $h^{-\eta}$) are reminiscent of \eqref{eq.cutoffvariables} and the index $j$ refers to the variable ($j=1$ corresponds to $\breve r$ and $j=3$ to $\breve t$). 

In the following, we drop the \it breve \rm accents and consider (mostly) $\mu$ as a parameter. The operator $\mathscr{N}_\hbar$ can be seen as an $\hbar$-pseudodifferential operator with operator symbol:
\begin{equation}\label{eq.symbol0}
\mathscr{N}_\hbar=\mathrm{Op}^{\mathrm{W}}_\hbar(n_\hbar)\,,
\end{equation}
with $n_\hbar(r,s,\rho,\sigma)=n_0+\hbar n_1+\hbar^2 n_2+\ldots$, where the principal symbol is
\begin{equation}\label{eq.n0good}
n_0=D^2_t+(\rho-(t-\xi_0)\sin\phi(s))^2
+\left(\mu\sigma+(t-\xi_0)\cos\phi(s)-\mu^{\frac12}\beta\chi_{h,1}^2\frac{r^2}{2}\right)^2+V_h(r)\,
\end{equation}
where $V_h(r) = V(r h^\frac\eta2)$ is the artificial confining potential coming from \eqref{eq.goodpotential} and where we recall that it can formally be replaced by $0$.

\begin{remark}\label{rem.explicationgauge}
Let us explore a little the structure of operator $\mathscr{N}_\hbar$ and the reason of the gauge change introduced in \eqref{eq.smalltildeLhd}. We observe that, when $\mu = 0$ and $V_h$ replaced by $0$, the operator-valued symbol $n_0$ is equal to
\begin{equation}
\begin{split}
& D^2_t+(\rho-(t-\xi_0)\sin\phi(s))^2+{(t-\xi_0)}^2\cos^2\phi(s) \\
& \qquad \qquad = D^2_t+(\xi_0-t + \rho \sin\phi(s))^2+ \rho^2\cos^2\phi(s),
\end{split}
\end{equation}
for which the minimum of the spectrum is $\Theta_0$, attained only when $\rho=0$ in virtue of the general properties of the de Gennes operator recalled below \eqref{eq:bigthetazero}. This was the main motivation of the preceding changes of gauge, including the shift by $\xi_0$, and will be in the core of the coming analysis.
\end{remark}
With Remark \ref{rem.explicationgauge} in mind, we now introduce
\begin{equation}\label{eq.p}
p:=\mu\sigma-\mu^{\frac12}\beta\chi_{h,1}^2\frac{r^2}{2}\,
\end{equation}
so that the principal operator valued symbol $n_0$ of $\mathscr{N}_\hbar$ takes the form
\begin{equation}\label{eq.n0verygood}
\begin{split}
n_0&=D_t^2+(\rho+(\xi_0-t)\sin\phi)^2+\left(\mu\sigma-(\xi_0-t)\cos\phi-\mu^{\frac12}\beta\chi_{h,1}^2\frac{r^2}{2}\right)^2+V_h\\
&=D_t^2+(\xi_0-t+\rho\sin\phi-p\cos\phi)^2+(\rho\cos\phi+p\sin\phi)^2+V_h\,.
\end{split}
\end{equation}
The subprincipal symbol is then
\[n_1=0\,,\]
and the next term in the expansion is
\begin{multline*}
n_2=2u_2((\xi_0-t)\cos\phi-p)\chi_{h,1}\chi_{h,3}rt+2(p-(\xi_0-t)\cos\phi) r\phi'\cos\phi \\
-2u_1 \chi_{h,1}\chi_{h,3}rt(\rho+(\xi_0-t)\sin\phi)
+r\chi_{h,1}\partial_r(\alpha^{-1})(p-(\xi_0-t)\cos\phi)^2\,.
\end{multline*}
Note that $n_2$ can also be written as
\begin{equation} \label{eq.n2n2tilde}
n_2=rL_\hbar+\tilde n_2\,,
\end{equation}
with
\begin{multline*}
L_\hbar=-2u_2(\xi_0-t)^2\chi_{h,1}\chi_{h,3}\cos\phi
+2u_1 \chi_{h,1}\chi_{h,3}(\xi_0-t)^2\sin\phi \\ +\chi_{h,1}\partial_r(\alpha^{-1})(\xi_0-t)^2\cos^2\phi
+(\xi_0-t)L_{h,1}\,,
\end{multline*}
with
\[L_{\hbar,1}=-2\chi_{h,1}\phi'\cos^2\phi+2\xi_0\chi_{h,1}\chi_{h,3}u_2\cos\phi-2u_1\xi_0\chi_{h,1}\chi_{h,3}\sin\phi\,,\]
and
\[\tilde n_2=-2rtpu_2\chi_{h,1}\chi_{h,3}+2p r\phi'\cos\phi+r\chi_{h,1}\partial_r(\alpha^{-1})\left(p^2-2p(\xi_0-t)\cos\phi \right)-2u_1 \chi_{h,1}\chi_{h,3}rt\rho\,.\]

Looking at the terms involving $(\xi_0-t)^2$ and forgetting the cutoff functions,
we notice an algebraic cancellation in $L_\hbar$.
\begin{lemma}\label{lem.annulr}
We have
\[-2u_2\cos\phi+2u_1\sin\phi+\cos^2\phi\,\partial_r(\alpha^{-1})=0\,.\]	
\end{lemma}
\begin{proof}
Recalling \eqref{eq.u1u2} and Lemma \ref{lem.metrict=0}, we see that
\[u_1=\frac12\sin\phi\,\partial_r\alpha+\partial_1\mathcal{B}_2\,,\quad u_2=-\frac12\cos\phi\,\partial_r\alpha-\partial_1\mathcal{B}_1\,,\]
where we also used Lemma \ref{lem.alpha=1}, \eqref{lem.e3r=t=0}  and Definition \ref{def.phi}. It follows that
\begin{equation*}
\begin{split}
-2u_2\cos\phi+2u_1\sin\phi+\cos^2\phi\,\partial_r(\alpha^{-1})&=\partial_r\alpha+2\cos\phi\,\partial_1\mathcal{B}_1+2\sin\phi\,\partial_1\mathcal{B}_2-\cos^2\phi\,\partial_r\alpha\\
&=2\mathcal{B}_1\,\partial_1\mathcal{B}_1+2\mathcal{B}_2\,\partial_1\mathcal{B}_2+\mathcal{B}_2^2\,\partial_r\alpha=0\,,
\end{split}
\end{equation*}
where we used \eqref{eq.norm1} (and its derivative with respect to $r$, at $r=0$).
\end{proof}
This cancellation will be used in Section \ref{sec.parametrix} when building an approximate parametrix for $\mathscr{N}_\hbar$. Anyway, a last preparation result is needed in order to be able to use pseudodifferential techniques. It is the aim of the next section.

\subsection{Microlocalization}\label{sec.roughmicro}
From \eqref{eq.n0verygood} and recalling Remark \ref{rem.explicationgauge}, we notice that the lowest eigenvalues of $\mathscr{N}_\hbar$ satisfy
\[\lambda_n(\mathscr{N}_\hbar)=\Theta_0+o(1)\,.\]
The corresponding eigenfunctions are actually microlocalized with respect to $\rho$ and $\mu\sigma$. Consider
\[\mathscr{N}^c_\hbar=\mathrm{Op}^{\mathrm{W}}_\hbar(n_\hbar^c)\,,\]
where $n_\hbar^c(r,s,\rho,\sigma)$ is obtained by replacing $\rho$ by $\Xi_1(\rho)$ where $\Xi_1\in S(1)$ is odd, increasing, and coincides with the identity near $\rho=0$, and by replacing $\mu\sigma$ by $\Xi_2(\mu\sigma)$ where $\Xi_2\in S(1)$ is  such that $\Xi_2(x)=x$ on $[-M,M]$ and equals $\pm 2M$ away from a compact set.

Recall here that $S(1)$ is the set of $\mathscr{C}^\infty$ bounded symbols as well as their derivatives. The functions $\Xi_1$ and $\Xi_2$ are also chosen so that, uniformly with respect to $s$ (or $\phi(s)$),
the function
\[(\rho,\mu\sigma)\mapsto\mu_1^{\mathrm{dG}}(\xi_0+\Xi_1(\rho)\sin\phi- \Xi_2(\mu\sigma)\cos\phi)+(\Xi_1(\rho)\cos\phi+ \Xi_2(\mu\sigma)\sin\phi)^2\,,\]
has still a unique minimum at $(0,0)$, which is non-degenerate and not attained at infinity.
Note this is possible in virtue of the general properties of $\mu_1^{\mathrm{dG}}$ recalled below
\eqref{eq:bigthetazero} and the fact that
\[
(\rho, \mu\sigma) \longmapsto (\rho \sin\phi- \mu\sigma \cos\phi, \rho\cos\phi+ \mu\sigma \sin\phi)
\]
is a rotation and therefore an isometry uniformly with respect to $s \in \Gamma$.
For the sake of shortness, we let
\begin{equation}\label{eq.trhotp}
\tilde\rho=\Xi_1(\rho)\,,\qquad \tilde p=\Xi_2(\mu\sigma)-\mu^{\frac12}\beta\chi_{h,1}^2\frac{r^2}{2}\,.
\end{equation}
With this notation, we have
\begin{equation*}
n^c_0(r,s,\rho,\sigma)=D_t^2+(\xi_0-t+\tilde\rho\sin\phi- \tilde p\cos\phi)^2+(\tilde\rho\cos\phi+ \tilde p\sin\phi)^2+V_h\,.
\end{equation*}

\begin{proposition}\label{prop.micro}
The low-lying spectra of $\mathscr{N}^c_\hbar$ and $\mathscr{N}_\hbar$ coincide modulo $\mathscr{O}(\hbar^\infty)$.	
\end{proposition}

\section{Parametrix constructions and spectral consequences}\label{sec.parametrix}

Thanks to our preparation results, we are now in position to analyze the eigenvalue problem.
We want now to identify the eigenvalues of $\mathscr{N}_\hbar^c$, and for this we shall
build an approximate inverse for an augmented inversible matrix of operator-valued operators.
This will be done in several steps, with two successive Grushin reductions, and will lead us to complete the proof of  Theorem \ref{thm.main}.

\subsection{Approximate parametrix}
We first apply the Grushin procedure to the operator-valued symbol $n_\hbar^c$.
We consider $z\in\mathbb{R}$ such that $|z-\Theta_0|\leq C\hbar^2$. We let
\begin{equation}\label{eq.Grushinmatrix}
\mathscr{P}_\hbar=
\begin{pmatrix}
n^c_\hbar-z&\cdot u_{r,s,\rho,\sigma}\\
\langle\cdot,u_{r,s,\rho,\sigma}\rangle&0
\end{pmatrix}
\,,
\end{equation}
where $u_{r,s,\rho,\sigma}=u^{\mathrm{dG}}_{\xi_0+\tilde\rho\sin\phi-\tilde p\cos\phi}$ is the first normalized positive eigenfunction of operator-valued symbol $n^c_0$ defined in Section \ref{sec.roughmicro}. Recall that $n^c_0$ is an operator in variable $t$ for which the phase space variables  $(r,s,\rho,\sigma)$ are considered as parameters. With these notations $\mathscr{P}_\hbar$ is an unbounded operator on $L^2(\R^+_t) \times \R$. We also consider its principal symbol
\[\mathscr{P}_0=
\begin{pmatrix}
n^c_0-z&\cdot u_{r,s,\rho,\sigma}\\
\langle\cdot,u_{r,s,\rho,\sigma}\rangle&0
\end{pmatrix}=\begin{pmatrix}
p_0&p_0^+\\
p^-_0&0
\end{pmatrix}
\,,\]
which is invertible (from its domain to $L^2(\R^+_t) \times \R$) uniformly with parameters thanks to the preceding section. Its inverse is the bounded operator given by
\[\mathscr{Q}_0
=\begin{pmatrix}
(n^c_0-z)^{-1}\Pi^\perp&\cdot u_{r,s,\rho,\sigma}\\
\langle\cdot,u_{r,s,\rho,\sigma}\rangle&z-\mu_1(r,s,\rho,\sigma)
\end{pmatrix}=\begin{pmatrix}
q_0&q_0^+\\
q^-_0&q_{0}^\pm
\end{pmatrix}
\,,\]
where
\begin{equation}\label{eq.mu1}
\mu_1(r,s,\rho,\sigma)=\mu_1^{\mathrm{dG}}(\xi_0+\tilde\rho\sin\phi-\tilde p\cos\phi)+(\tilde\rho\cos\phi+ \tilde p\sin\phi)^2+V_h\,,
\end{equation}
and $\Pi^\perp$ is the orthogonal projection on $\mathrm{span}(u_{r,s,\rho,\sigma})^\perp$.
By construction, we have
\[\mathscr{Q}^\mathrm{W}_0\mathscr{P}^\mathrm{W}_\hbar=\mathrm{Id}+\mathscr{O}(\hbar)\,.\]
We can then find  $\sep{ \mathscr{Q}_j}_{j=1, \cdots, 5}$ such that, at a formal level, we have
\begin{equation}\label{eq.parametrix1}
(\mathscr{Q}_0+\hbar\mathscr{Q}_1+\hbar^2\mathscr{Q}_2+\hbar^3\mathscr{Q}_3
+\hbar^4\mathscr{Q}_4)^\mathrm{W}\mathscr{P}^\mathrm{W}_\hbar=\mathrm{Id}+\mathscr{O}(\hbar^5)\,
\end{equation}
where the error term is possibly an unbounded operator in variable $t$ in this first formal approach.
We recall that the Moyal product is given in our setting by
\[\mathscr{Q}^{\mathrm{W}}_\hbar\mathscr{P}_\hbar^{\mathrm{W}}=(\mathscr{Q}_\hbar e^{\frac{\hbar}{2i}\square}\mathscr{P}_\hbar)^{\mathrm{W}}\,,\quad
 \square=\overleftarrow{\partial}_{(\rho,\sigma)} \overrightarrow{\partial}_{(r,s)}
-\overleftarrow{\partial}_{(r,s)}\overrightarrow{\partial}_{(\rho,\sigma)}\,.\]
Let us explain how to find $\mathscr{Q}_1$ and $\mathscr{Q}_2$. By using the Moyal product and a formal expansion in powers of $\hbar$, $\mathscr{Q}_1$ and $\mathscr{Q}_2$ must satisfy the following relations:
 \[\mathscr{Q}_1\mathscr{P}_0+\mathscr{Q}_0\mathscr{P}_1+\frac{1}{2i}\{\mathscr{Q}_0,\mathscr{P}_0\}=0\,,\]
 \[\mathscr{Q}_2\mathscr{P}_0+\mathscr{Q}_1\mathscr{P}_1+\mathscr{Q}_0\mathscr{P}_2
 +\frac{1}{2i}\left(\{\mathscr{Q}_1,\mathscr{P}_0\}+\{\mathscr{Q}_0,\mathscr{P}_1\}\right)
 -\frac18\mathscr{Q}_0\square^2\mathscr{P}_0=0\,,\]
where
 \[A\square^2B=\partial^2_{(\rho,\sigma)} A\partial^2_{(r,s)}B-2\partial^2_{{(r,s)},{(\rho,\sigma)}}A\partial^2_{{(r,s)},{(\rho,\sigma)}}B
 +\partial^2_{(r,s)}A\partial^2_{(\rho,\sigma)} B\,.\]
We get
\[\mathscr{Q}_1=-\mathscr{Q}_0\mathscr{P}_1\mathscr{Q}_0-\frac{1}{2i}\{\mathscr{Q}_0,\mathscr{P}_0\}\mathscr{Q}_0\,,\]
and
\[\mathscr{Q}_2=-\mathscr{Q}_1\mathscr{P}_1\mathscr{Q}_0-\mathscr{Q}_0\mathscr{P}_2\mathscr{Q}_0-\frac{1}{2i}\left(\{\mathscr{Q}_1,\mathscr{P}_0\}+\{\mathscr{Q}_0,\mathscr{P}_1\}\right)\mathscr{Q}_0+\frac18\left(\mathscr{Q}_0\square^2\mathscr{P}_0\right)\mathscr{Q}_0\,.\]
We will use the following notation
\[\mathscr{Q}=\begin{pmatrix}
q&q^+\\
q^-&q^\pm
\end{pmatrix}\,\]
and we shall sometimes denote the bottom right component $q^\pm$ by $\mathscr{Q}_\pm$.
Since the $\mathscr{P}_j$ are self-adjoint, so are the $\mathscr{Q}_j$. In particular,
\[\left(\{\mathscr{Q}_0,\mathscr{P}_0\}\mathscr{Q}_0\right)_\pm=0\,.\]
Since $n_1=0$, we have $\mathscr{P}_1=0$. It follows that
\[q^\pm_1=0\,.\]
Let us notice that $\mathscr{P}_0$ and $\mathscr{Q}_0$ do not Poisson-commute (in the sense that $\mathscr{Q}_0\square^k\mathscr{P}_0=0$ for all $k\geq 1$). However, if $\mu$ and  $V_h$ were replaced by $0$, this Poisson-commutation holds. Keeping this remark in mind and using similar considerations as above, we can write
\[\begin{split}
q_2^\pm&=(\mathscr{Q}_0\mathscr{P}_1\mathscr{Q}_0\mathscr{P}_1\mathscr{Q}_0)_\pm-(\mathscr{Q}_0\mathscr{P}_2\mathscr{Q}_0)_\pm+R_\hbar+\mu r_\hbar+\mathscr{O}(\mu^2)\\
&=-(\mathscr{Q}_0\mathscr{P}_2\mathscr{Q}_0)_\pm+R_\hbar+\mu r_\hbar+\mathscr{O}(\mu^2)\\
&=-\left\langle (rL_\hbar+\tilde n_2) u_{r,s,\rho,\sigma},u_{r,s,\rho,\sigma}\right\rangle+R_\hbar+\mu r_\hbar+\mathscr{O}(\mu^2)\,,
\end{split}\]
where we used \eqref{eq.n2n2tilde} and with $R_\hbar$ supported in $\{|r|\geq h^{-\eta/2}\}$ ($R_\hbar$ comes from the presence of $V_\hbar$). 

Then, with Lemma \ref{lem.annulr}, and the exponential decay of $u_{r,s,\rho,\sigma}$ with respect to $t$, we get
\begin{equation*}
q_2^\pm=-\langle \chi_{h,1}(rL_{1}(\xi_0-t)+\tilde n^0_2)u_{r,s,\rho,\sigma},u_{r,s,\rho,\sigma}\rangle+\tilde R_\hbar+\mu r_\hbar+\mathscr{O}(\mu^2)+\mathscr{O}(\hbar^\infty)\,,
\end{equation*}
where
\[L_{1}=-2\phi'\cos^2\phi+2\xi_0u_2\cos\phi-2u_1\xi_0\sin\phi\,,\]
and
\[\tilde n^0_2=-2rt\tilde pu_2+2\tilde p r\phi'\cos\phi+r\partial_r(\alpha^{-1})\left(\tilde p^2-2\tilde p(\xi_0-t)\cos\phi \right)-2u_1 rt\tilde\rho\,.\]
The terms $\mu r_\hbar$ and $\mathscr{O}(\mu^2)$ are in fact phantom terms since $\mu=\hbar^2$: they play at most at the order $\hbar^4$. In what follows, we implicitly include them in the lower order terms.

\begin{proposition}\label{prop.qpm}
Defining $q^\pm_\hbar(z)=(\mathscr{Q}_0+\hbar\mathscr{Q}_1+\hbar^2\mathscr{Q}_2
+\hbar^3\mathscr{Q}_3+\hbar^4\mathscr{Q}_4)_\pm$ and recalling $|z-\Theta_0|\leq C\hbar^2$, we have
\begin{multline*}
q^\pm_\hbar(z)=z-\mu_1(r,s,\rho,\sigma)-\hbar^2\langle \chi_{h,1}(rL_1(\xi_0-t)+\tilde n_2^0) u_{r,s,\rho,\sigma},u_{r,s,\rho,\sigma}\rangle\\
+\hbar^3q_{3,z}^\pm(r,s,\rho,\mu\sigma)+\hbar^4 q_{4,z}^\pm(r,s,\rho,\mu\sigma)+R_\hbar+\mathscr{O}(\hbar^\infty)\,,
\end{multline*}
where the $q_{j,z}^\pm$ belong to $S(1)$ and are analytic with respect to $z$.
\end{proposition}
Actually, since we are interested in the $z$ of the form
\[z=\Theta_0+\mathscr{O}(\hbar^2)\,,\]
this leads to the following effective symbol where we have fixed $z= \Theta_0$ in the ${\mathscr O}(\hbar^3)$ terms and removed the ${\mathscr O}(\hbar^\infty)$ and $R_\hbar$ terms,
\begin{equation}\label{eq.aeff}
a_{\hbar}^{\mathrm{eff}}(r,s,\rho,\sigma)=\mu_1(r,s,\rho,\sigma)+\hbar^2\langle \chi_{h,1}(rL_1(\xi_0-t)+\tilde n_2^0) u_{r,s,\rho,\sigma},u_{r,s,\rho,\sigma}\rangle-\hbar^3q_{3,\Theta_0}^\pm-\hbar^4 q_{4,\Theta_0}^\pm\,,
\end{equation}
and where we recall \eqref{eq.mu1}.

\subsection{Spectral consequences}
In this section, we explain how the symbols $q^{\pm}_\hbar(z)$ and $a_{\hbar}^{\mathrm{eff}}$ are related to the spectrum of $\mathscr{N}_\hbar^c$.
\begin{proposition}\label{prop.specreducaeff}
For all $n\geq 1$,
\[\lambda_n\left(\mathscr{N}_\hbar^c\right)=\lambda_n\left(\mathrm{Op}^W_\hbar a_{\hbar}^{\mathrm{eff}}\right)+\mathscr{O}(\hbar^5)\,.\]
\end{proposition}
\begin{proof}
With the choices of $\mathscr{Q}_j$, we get, for some integer $N$,
\[(\mathscr{Q}_0+\hbar\mathscr{Q}_1+\hbar^2\mathscr{Q}_2+\hbar^3\mathscr{Q}_3+\hbar^4\mathscr{Q}_4)^\mathrm{W}\mathscr{P}^\mathrm{W}_\hbar=\mathrm{Id}+\mathscr{O}_{L^2(\langle t\rangle^{N}\dd t  \times \C )\to L^2(\dd t)  \times \C }(\hbar^5)\,.\]		
We write
\[\mathrm{Op}^{\mathrm{W}}_\hbar\left(\mathscr{Q}_0+\hbar\mathscr{Q}_1+\hbar^2\mathscr{Q}_2+\hbar^3\mathscr{Q}_3+\hbar^4\mathscr{Q}_4\right)=\begin{pmatrix}
Q_\hbar&Q_\hbar^+\\
Q^-_\hbar&Q_\hbar^\pm
\end{pmatrix}\,,\quad \Pi=\mathrm{Op}^{\mathrm{W}}_\hbar(\langle\cdot,u_{r,s,\rho,\sigma}\rangle)\,.\]
Thus,
\[Q_\hbar(\mathscr{N}^c_\hbar-z)+Q_\hbar^+\Pi=\mathrm{Id}+\mathscr{O}_{ L^2(\langle t\rangle^{N}\dd t)\to L^2(\dd t) }(\hbar^5)\,,\quad Q^-_\hbar(\mathscr{N}^c_\hbar-z)+Q_\hbar^\pm\Pi=\mathscr{O}_{L^2(\langle t\rangle^{N}\dd t)\to \C }(\hbar^5)\,.\]
By means of the Calder\'on-Vaillancourt theorem with respect to the variables $(r,s)$, it follows that
\begin{equation}\label{eq.grushinrough0}
\|\psi\|\leq C\|\Pi\psi\|+C\|(\mathscr{N}^c_\hbar-z)\psi\|+C\hbar^5\|\langle t\rangle^N \psi\|\,,
\end{equation}
and
\[\|Q_\hbar^\pm(\Pi\psi)\|\leq C\|(\mathscr{N}^c_\hbar-z)\psi\|+C\hbar^5\|\langle t\rangle^N\psi\|\,.\]
We recall that $Q_\hbar^\pm=\mathrm{Op}^{\mathrm{W}}_\hbar(q_\hbar^\pm)$ is described in Proposition \ref{prop.qpm}. Since we are considering $z\in\mathbb{R}$ such that $|z-\Theta_0|\leq C\hbar^2$, we get from Proposition \ref{prop.qpm} again that
\begin{equation}\label{eq.grushinrough}
\|(z-\mathrm{Op}^W_\hbar a_{\hbar}^{\mathrm{eff}})\Pi\psi\|\leq C\|(\mathscr{N}^c_\hbar-z)\psi\|+C\hbar^5\|\langle t\rangle^N\psi\|+\|\mathrm{Op}_\hbar^{\mathrm{W}}(R_\hbar)\Pi\psi\|\,.
\end{equation}
Let us now consider \eqref{eq.grushinrough0} and \eqref{eq.grushinrough}. Applying \eqref{eq.grushinrough} to an eigenpair $(\lambda,\psi)$ with $\lambda=\Theta_0+\mathscr{O}(\hbar^2)$, we get
\[\|(\lambda-\mathrm{Op}^W_\hbar a_{\hbar}^{\mathrm{eff}})\Pi\psi\|\leq C\hbar^5\|\langle t\rangle^N\psi\|+\|\mathrm{Op}_\hbar^{\mathrm{W}}(R_\hbar)\Pi\psi\|\,.\]
From Proposition \ref{prop.checkLm}, the eigenfunctions of $\mathscr{N}_\hbar^c$ still satisfy Agmon estimates with respect to $t$ and $r$. This implies that
\[\|(\lambda-\mathrm{Op}^W_\hbar a_{\hbar}^{\mathrm{eff}})\Pi\psi\|\leq C\hbar^5\|\psi\|\,,\]
and, with \eqref{eq.grushinrough0}, we find
\[\|(\lambda-\mathrm{Op}^W_\hbar a_{\hbar}^{\mathrm{eff}})\Pi\psi\|\leq C\hbar^5\|\Pi\psi\|\,.\]
Since $\Pi\psi\neq 0$, the spectral theorem tells us that $\mathrm{dist}\left(\lambda, {\rm Sp} \left( \mathrm{Op}^W_\hbar a_{\hbar}^{\mathrm{eff}}\right)\right)\leq C\hbar^5$.
Let us consider
\[E_n(\hbar)=\ker(\mathscr{N}_\hbar^c-\lambda_n(\hbar))\,,\quad \lambda_n(\hbar)=\lambda_n\left(\mathscr{N}_\hbar^c\right)\,.\]
We have, for all $\psi\in E_n(\hbar)$,
\[\|\psi\|\leq C\|\Pi\psi\|\,.\]
In particular, $\dim\mathrm{ran}\,\Pi_{|E_n(\hbar)}=\dim E_n(\hbar)$. Since, for all $\psi\in E_n(\hbar)$,
\[\|(\lambda_n(\hbar)-\mathrm{Op}^W_\hbar a_{\hbar}^{\mathrm{eff}})\Pi\psi\|\leq C\hbar^5\|\Pi\psi\|\,,\]
the spectral theorem implies that there are at least $\dim E_n(\hbar)$ eigenvalues of $\mathrm{Op}^W_\hbar a_{\hbar}^{\mathrm{eff}}$ in the disc $D(\lambda_n(\hbar),C\hbar^5)$.

\bigskip
Conversely, we can check that $\lambda_n^{\mathrm{eff}}(\hbar):=\lambda_n\left(\mathrm{Op}^W_\hbar a_{\hbar}^{\mathrm{eff}}\right)=\Theta_0+\mathscr{O}(\hbar^2)$. We also observe that
\[\mathscr{P}^\mathrm{W}_\hbar(\mathscr{Q}_0+\hbar\mathscr{Q}_1+\hbar^2\mathscr{Q}_2+\hbar^3\mathscr{Q}_3+\hbar^4\mathscr{Q}_4)^\mathrm{W}=\mathrm{Id}+\mathscr{O}(\hbar^5)\,.\]	
This gives	
\[(\mathscr{N}_\hbar^c-z)Q^+_\hbar+\Pi Q_\hbar^\pm=\mathscr{O}(\hbar^5)\,,\quad \Pi^*Q_\hbar^+=\mathrm{Id}+\mathscr{O}(\hbar^5)\,.\]
We get
\[\|(\mathscr{N}_\hbar^c-z)Q^+_\hbar\psi\|\leq C\|Q_\hbar^\pm\psi\|+C\hbar^5\|\psi\|\,,\]
and then from Proposition \ref{prop.qpm}
\[\|(\mathscr{N}_\hbar^c-z)Q^+_\hbar\psi\|\leq C\|(z-\mathrm{Op}^W_\hbar a_{\hbar}^{\mathrm{eff}})\psi\|+C\hbar^5\|\psi\|\,.\]
We also have that
\[\|\psi\|\leq C\|Q_\hbar^+\psi\|\,,\]
which shows that $Q_\hbar^+$ is injective. Using again the spectral theorem, we see that there are at least $\dim \ker(\mathscr{N}_\hbar^c-\lambda^{\mathrm{eff}}_n(\hbar))$ eigenvalues of $\mathscr{N}_\hbar^c$ in the disc $D(\lambda^{\mathrm{eff}}_n(\hbar),C\hbar^5)$. The proof of Proposition \ref{prop.specreducaeff} is complete.
\end{proof}

\subsection{Spectral analysis of $\mathrm{Op}_\hbar^Wa_\hbar^{\mathrm{eff}}$}
This section is devoted to the description of the low-lying spectrum of $\mathrm{Op}_\hbar^Wa_\hbar^{\mathrm{eff}}$.
The principal symbol is
\[\mu_1(r,s,\rho,\sigma)=\mu^{\mathrm{dG}}(\xi_0+\tilde\rho\sin\phi-\tilde p\cos\phi)+(\tilde\rho\cos\phi+ \tilde p\sin\phi)^2+V_h(r)\,.\]
This implies that the eigenfunctions are microlocalized near $\tilde\rho=\tilde p=0$, where we recall that
\[\tilde\rho=\Xi_1(\rho)\,,\quad \tilde p=\Xi_2(\mu\sigma)-\mu^{\frac12}\beta(s)\chi_{h,1}\frac{r^2}{2}\,.\]

\begin{proposition}\label{prop.rhop}
Let $C>0$. For $\eta\in\left(0,\frac13\right)$, the following holds. Let us consider a smooth function $\zeta$ equalling $0$ near $0$ and $1$ away from a compact neighborhood of $0$. We consider  \[\zeta_{1,\hbar}(r,s,\rho,\sigma)=\zeta(\hbar^{-\eta}\tilde\rho)\,,\quad \zeta_{2,\hbar}(r,s,\rho,\sigma)=\zeta(\hbar^{-\eta}\tilde p)\,.\] 
There exists $\hbar_0>0$ such that, for all $\hbar\in(0,\hbar_0)$ and all normalized eigenfunctions $\psi$ of $\mathrm{Op}^W_\hbar a_{\hbar}^{\mathrm{eff}}$ associated with an eigenvalue $\lambda$ satisfying $\lambda\leq\Theta_0+C\hbar^2$, we have
\[\zeta_{j,\hbar}^W \psi=\mathscr{O}(\hbar^\infty)\|\psi\|\,.\]	
\end{proposition}
\begin{proof}
Let us consider the eigenvalue equation
\[\mathrm{Op}_\hbar^Wa_\hbar^{\mathrm{eff}}\psi=\lambda\psi\,.\]
We get
\[\left(\mathrm{Op}_\hbar^Wa_\hbar^{\mathrm{eff}}-\lambda\right)\zeta_{j,\hbar}^W\psi=[\mathrm{Op}_\hbar^Wa_\hbar^{\mathrm{eff}},\zeta_{j,\hbar}^W]\psi\,,\]
and then
\[\Re\langle\left(\mathrm{Op}_\hbar^Wa_\hbar^{\mathrm{eff}}-\lambda\right)\zeta_{j,\hbar}^W\psi,\zeta_{j,\hbar}^W\psi\rangle=\Re\langle[\mathrm{Op}_\hbar^Wa_\hbar^{\mathrm{eff}},\zeta_{j,\hbar}^W]\psi,\zeta_{j,\hbar}^W\psi\rangle\,.\]
Thanks to the Calder\'on-Vaillancourt theorem and support considerations, we get that
\[\Re\langle\left(\mathrm{Op}_\hbar^Wa_\hbar^{\mathrm{eff}}-\lambda\right)\zeta_{j,\hbar}^W\psi,\zeta_{j,\hbar}^W\psi\rangle\leq C\hbar^{1-\eta}\|\underline{\zeta}_{j,\hbar}^W\psi\|\|\zeta_{j,\hbar}^W\psi\|\,,\]
where $\underline{\zeta}_{j,\hbar}$ has a support slightly larger than the one of $\zeta_{j,\hbar}$.
Then, we observe that, for some $c_0>0$,
\[\mu_1(r,s,\rho,\sigma)\geq\Theta_0+c_0\min(\tilde\rho^2+\tilde p^2,1)\,.\]
Note that this implies that $\mu_1-\Theta_0 \gtrsim  \hbar^{2\eta}$ on the support of $\zeta_{j,\hbar}$. 
Therefore on this support we have
\[a_\hbar^{\mathrm{eff}}(r,s,\rho,\sigma)-\lambda\geq c_0\min(\tilde\rho^2+\tilde p^2,1)-C\hbar^2 \, \gtrsim \hbar^{2\eta} - C\hbar^2\,.\]
We can therefore apply the Fefferman-Phong inequality in $S^\eta(1)$ (see \cite{Bony} or \cite{Fermanian}) made of semiclassical symbols for which each derivation implies a loss of $\hbar^{-\eta}$, and we deduce that, for $\eta\in(0,\frac12)$,
\[\Re\langle\left(\mathrm{Op}_\hbar^Wa_\hbar^{\mathrm{eff}}-\lambda\right)
\zeta_{j,\hbar}^W\psi,\zeta_{j,\hbar}^W\psi\rangle\geq (c\hbar^{2\eta}-C\hbar^{2-4\eta})\|\zeta_{j,\hbar}^w\psi\|^2\,.\]
Using $\eta<\frac13$, we find that
\[\|\zeta_{j,\hbar}^W\psi\|\leq C\hbar^{1-3\eta}\|\underline{\zeta}_{j,\hbar}^W\psi\|\,.\]
By considering now $\underline{\zeta}_{j,\hbar}^W\psi$, we get the result by induction.
\end{proof}

\begin{remark}
Note that due to the form of $\Xi_1$, the microlocalization with respect to $\tilde\rho$ at the scale $\hbar^\eta$ also implies  that
\[\zeta(\hbar^{-\eta}\rho)^W\psi=\mathscr{O}(\hbar^\infty)\,.\]	
\end{remark}

Proposition \ref{prop.rhop} invites us to use a Taylor expansion of $a^{\mathrm{eff}}_\hbar$. Remember that the eigenfunction
\[u_{r,s,\rho,\sigma}=u^{\mathrm{dG}}_{\xi_0+\tilde\rho\sin\phi-\tilde p\cos\phi}\]
depends on $(\tilde\rho,\tilde p)$ only. For all $N\in\N$, let us consider the symbol
\begin{multline}\label{eq.aeffN}
a_{\hbar,N}^{\mathrm{eff}}(r,s,\rho,\sigma)={a}_0(r,s,\rho,\sigma)\\
+c_0(\tilde\rho,\tilde p)\sum_{k=3}^N\frac{(\mu^{\mathrm{dG}})^{(k)}(\xi_0)}{k!}(\tilde\rho\sin\phi-\tilde p\cos\phi)^k
+\hbar^2c_0(\tilde\rho,\tilde p)r\chi_{h,1}L_1\sum_{k=1}^N\frac{f^{(k)}(0)}{k!}(\tilde\rho\sin\phi-\tilde p\cos\phi)^k\\
 + \hbar^2c_0(\tilde{\rho},\tilde{p})r \chi_{h,1}  \sum_{1\leq \ell+\ell' \leq N} g_{\ell, \ell'}(r,s)  \tilde{\rho}^\ell \tilde{p}^{\ell'}
-\hbar^3q_{3,\Theta_0}^\pm-\hbar^4 q_{4,\Theta_0}^\pm\,,
\end{multline}
where 
\begin{enumerate}[---]
\item the principal term is
\[{a}_0=\frac{1}{2}(\mu^{\mathrm{dG}})''(\xi_0)(\tilde\rho\sin\phi-\tilde p\cos\phi)^2+(\tilde\rho\cos\phi+\tilde p\sin\phi)^2+V_h(r)\,,\]
\item $c_0$ is a smooth cutoff function equalling $1$ on a small enough neighborhood of $(0,0)$, 
\item the function $f$ is given by
\[f(\xi)=\langle (\xi_0-t)u^{\mathrm{dG}}_{\xi_0+\xi},u^{\mathrm{dG}}_{\xi_0+\xi}\rangle\,,\]
\item the sum involving $g_{\ell, \ell'}$ is the Taylor expansion of $\langle \tilde{n}^0 _2 u^{\mathrm{dG}}_{\xi_0+\xi},u^{\mathrm{dG}}_{\xi_0+\xi}\rangle$ in variables $\tilde\rho$ and $\tilde p$.

\end{enumerate}
Note that $f(0)=0$.

\begin{proposition}\label{prop.2}
Let $M,n\in\mathbb{N}$. There exists $N\in\mathbb{N}$ such that
\[\lambda_n\left(\mathrm{Op}^W_\hbar a_\hbar^{\mathrm{eff}}\right)=\Theta_0+\lambda_n\left(\mathrm{Op}^W_\hbar a_{\hbar,N}^{\mathrm{eff}}\right)+\mathscr{O}(\hbar^M)\,.\]	
\end{proposition}	

Let us now focus on the new operator $\mathrm{Op}^W_\hbar a_{\hbar,N}^{\mathrm{eff}}$.

\begin{remark}\label{rem.microprho}
It is clear that the eigenfunctions of this new operator associated with eigenvalues $\lambda=\Theta_0+\mathscr{O}(\hbar^2)$ are still microlocalized near $\tilde\rho=\tilde p=0$ at the scale $\hbar^{\frac12-\eta}$ using the ellipticity of $\tilde\rho^2+\tilde p^2$. They are also roughly localized near $r=0$ at scale $h^{-\eta/2}$ (due to the presence of the confining potential $V_h$ and the cutoff functions in $r$). The tildes and the $V_h$ can then be removed if necessary.
\end{remark}

The rough localization in $r$ is not sufficient since we want uniform bounds in $\hbar$ with respect to $r$. This is the aim of the following lemma where we prove that the (rescaled) variable $r$ lives at the scale $1$.
\begin{lemma}
Let us consider $n\geq 1$. There exist $\hbar_0>0$, $C>0$ such that, for all $\hbar\in(0,\hbar_0)$ and all eigenfunctions associated with $\lambda_n\left(\mathrm{Op}^W_\hbar a_{\hbar,N}^{\mathrm{eff}}\right)$, we have
\[\|r^4\psi\|\leq C\|\psi\|\,.\]
\end{lemma}
\begin{proof}
Let us write
\[\left(\mathrm{Op}^W_\hbar a_{\hbar,N}^{\mathrm{eff}}\right)\psi=\lambda_n\left(\mathrm{Op}^W_\hbar a_{\hbar,N}^{\mathrm{eff}}\right)\psi\,.\]
Let us consider a non-negative smooth function $\chi_+(r)$ equalling $1$ on $r\geq 1$ and $0$ for $r\leq\frac12$. We set $\chi_-(r)=\chi_+(-r)$ and we let $\chi=\chi_\pm$.
It follows that
\begin{equation}\label{eq.eei}
\langle\left(\mathrm{Op}^W_\hbar a_{\hbar,N}^{\mathrm{eff}}\right)\chi\psi,\chi\psi\rangle\leq C\hbar^2\|\chi\psi\|^2+C\hbar\|\tilde\rho^W\psi\|\|\psi\|\,.
\end{equation}
Note that
\begin{multline*}
\Big\|\big( \hbar^2c_0(\tilde\rho,\tilde p)r\chi_{h,1}L_1\sum_{k=1}^N\frac{f^{(k)}(0)}{k!}(\tilde\rho\sin\phi-\tilde p\cos\phi)^k\\
 + \hbar^2c_0(\tilde{\rho},\tilde{p})  r\chi_{\hbar,1} \sum_{1\leq \ell+\ell' \leq N} g_{\ell, \ell'}(r,s) \tilde{\rho}^\ell \tilde{p}^{\ell'} -\hbar^3q_{3,\Theta_0}^\pm-\hbar^4 q_{4,\Theta_0}^\pm\big)^W(\chi\psi)\Big\|=\mathscr{O}(\hbar^2)
\end{multline*}
since $|r\chi_{h,1}|\leq \hbar^{-\eta}$ and $\psi$ is microlocalized in $(\tilde\rho,\tilde p)$ at the scale $\hbar^{\frac12-\eta}$. The Calder\'on-Vaillancourt is used to control the last two terms.
Then, note that, due to the small support of $c_0$,
\begin{multline*}
\frac{1}{2}(\mu^{\mathrm{dG}})''(\xi_0)(\tilde\rho\sin\phi-\tilde p\cos\phi)^2+(\tilde\rho\cos\phi+\tilde p\sin\phi)^2
+c_0(\tilde\rho,\tilde p)\sum_{k=3}^N\frac{(\mu^{\mathrm{dG}})^{(k)}(\xi_0)}{k!}(\tilde\rho\sin\phi-\tilde p\cos\phi)^k\\
\geq d(\tilde \rho^2+\tilde p^2)\,.
\end{multline*}
Thanks to the Fefferman-Phong inequality (the symbols are bounded), we get
 \[\langle\left(\mathrm{Op}^W_\hbar a_{\hbar,N}^{\mathrm{eff}}\right)\chi\psi,\chi\psi\rangle\geq d\langle\mathrm{Op}^W_\hbar(\tilde\rho^2+\tilde p^2)\chi\psi,\chi\psi\rangle-D\hbar^2\|\psi\|^2\,.\]
Using again Remark \ref{rem.microprho}, we infer that
 \[\langle\left(\mathrm{Op}^W_\hbar a_{\hbar,N}^{\mathrm{eff}}\right)\chi\psi,\chi\psi\rangle\geq d\langle\mathrm{Op}^W_\hbar(\rho^2+ p^2)\chi\psi,\chi\psi\rangle-D\hbar^2\|\psi\|^2\,.\]
Notice that
\[\begin{split}\langle\mathrm{Op}^W_\hbar(\rho^2+ p^2)\chi\psi,\chi\psi\rangle&\geq |\langle[\rho^W,p^W]\chi\psi,\chi\psi\rangle|=\hbar\mu^{\frac12}\left|\int r\beta|\chi\psi|^2\right|=\hbar^2\left|\int r\beta|\chi\psi|^2\right|\\
&\geq \hbar^2\min\beta\left|\int r|\chi\psi|^2\right|\,.
\end{split}\]
From \eqref{eq.eei}, it follows that
\[ \hbar^2\min\beta\left|\int r|\chi\psi|^2\right|\leq \tilde D\hbar^2\|\psi\|^2+C\hbar\|\tilde\rho^W\psi\|\|\psi\|\,,\]
which gives, with the eigenvalue equation,
\[ \hbar^2\min\beta\int |r||\chi_{\pm}\psi|^2\leq  D\hbar^2\|\psi\|^2\,.\]
This gives that
\[\||r|^{\frac12}\psi\|\leq C\|\psi\|\,.\]
Iterating this process gives indeed a localization at any power of $r$, in particular the power $4$, and the result follows. We skip this iteration which follows from similar arguments involving perhaps some commutators between $r$ and $\rho$ of the  order ${\hbar^{1-\eta}}$.
\end{proof}

Let us now consider a smooth function $\chi_M$ such that $\chi_M$ equals $1$ away from $[-M,M]$ neighborhood of $0$ and equals $0$ on $[-M/2,M/2]$. We can slightly adapt the proof of the last lemma and get the following.

\begin{lemma}\label{lem.r4tronc}
Let us consider $n\geq 1$. There exist $\hbar_0>0$, $C>0$ such that, for all $\hbar\in(0,\hbar_0)$ and all eigenfunctions associated with $\lambda_n\left(\mathrm{Op}^W_\hbar a_{\hbar,N}^{\mathrm{eff}}\right)$, we have
\[\|r^4\mathrm{Op}^W_\hbar(\chi_M(\hbar\sigma))\psi\|\leq C\|\mathrm{Op}^W_\hbar(\chi_M(\hbar\sigma))\psi\|+\mathscr{O}(\hbar^\infty)\|\psi\|\,.\]
\end{lemma}

We have now everything in hand to prove a refined microlocalization of the eigenfunctions with respect to $\sigma$.
\begin{lemma}\label{lem.microhsigma}
 There exist $\hbar_0, M>0$ such that, for all $\hbar\in(0,\hbar_0)$, and all eigenfunctions $\psi$ of $\mathrm{Op}^W_\hbar a_{\hbar,N}^{\mathrm{eff}}$ associated with $\lambda=\mathscr{O}(\hbar^2)$, we have
\[\mathrm{Op}^W_\hbar(\chi_M(\hbar\sigma))\psi=\mathscr{O}(\hbar^\infty)\|\psi\|\,.\]
\end{lemma}
\begin{proof}
Let us write
\[\left(\mathrm{Op}^W_\hbar a_{\hbar,N}^{\mathrm{eff}}\right)\psi=\lambda\psi\,,\]
and observe that
\begin{multline}\label{eq.microlocchiM}
\langle\left(\mathrm{Op}^W_\hbar a_{\hbar,N}^{\mathrm{eff}}\right)\mathrm{Op}^W_\hbar(\chi_M(\hbar\sigma))\psi,\mathrm{Op}^W_\hbar(\chi_M(\hbar\sigma))\psi\rangle\\
\leq C\hbar^2\|\mathrm{Op}^W_\hbar(\chi_M(\hbar\sigma))\|^2+|\langle[\mathrm{Op}^W_\hbar a_{\hbar,N}^{\mathrm{eff}},\mathrm{Op}^W_\hbar(\chi_M(\hbar\sigma))]\psi, \mathrm{Op}^W_\hbar(\chi_M(\hbar\sigma))\psi\rangle|\,.
\end{multline}
As we did before, we have
\begin{multline*}
\langle\left(\mathrm{Op}^W_\hbar a_{\hbar,N}^{\mathrm{eff}}\right)\mathrm{Op}^W_\hbar(\chi_M(\hbar\sigma))\psi,\mathrm{Op}^W_\hbar(\chi_M(\hbar\sigma))\psi\rangle\\
\geq c(\|\tilde\rho^W \mathrm{Op}^W_\hbar(\chi_M(\hbar\sigma))\psi\|^2+\|\tilde p^W\mathrm{Op}^W_\hbar(\chi_M(\hbar\sigma))\psi\|^2)-C\hbar^2\|\mathrm{Op}^W_\hbar(\chi_M(\hbar\sigma))\psi\|^2\,.
\end{multline*}
By using the microlocalization with respect to $\rho$ and $r$, we have
\begin{multline*}
\langle\left(\mathrm{Op}^W_\hbar a_{\hbar,N}^{\mathrm{eff}}\right)\mathrm{Op}^W_\hbar(\chi_M(\hbar\sigma))\psi,\mathrm{Op}^W_\hbar(\chi_M(\hbar\sigma))\psi\rangle\\
\geq c(\|\rho^W \mathrm{Op}^W_\hbar(\chi_M(\hbar\sigma))\psi\|^2+\| p^W\mathrm{Op}^W_\hbar(\chi_M(\hbar\sigma))\psi\|^2)-C\hbar^2\|\mathrm{Op}^W_\hbar(\chi_M(\hbar\sigma))\psi\|^2+\mathscr{O}(\hbar^\infty)\|\psi\|^2\,,
\end{multline*}
where $p=\Xi_2(\mu\sigma)-\beta\frac{r^2}{2}$. We find
\begin{equation*}
\begin{split}
&\langle\left(\mathrm{Op}^W_\hbar a_{\hbar,N}^{\mathrm{eff}}\right)\mathrm{Op}^W_\hbar(\chi_M(\hbar\sigma))\psi,\mathrm{Op}^W_\hbar(\chi_M(\hbar\sigma))\psi\rangle\\
&\geq c\| p^W\mathrm{Op}^W_\hbar(\chi_M(\hbar\sigma))\psi\|^2-C\hbar^2\|\mathrm{Op}^W_\hbar(\chi_M(\hbar\sigma))\psi\|^2+\mathscr{O}(\hbar^\infty)\|\psi\|^2\\
&\geq \frac{c}{2}\|\mathrm{Op}^W_\hbar\Xi_2(\mu\sigma)\,\mathrm{Op}^W_\hbar(\chi_M(\hbar\sigma))\psi\|^2-C\hbar^2\|\mathrm{Op}^W_\hbar(\chi_M(\hbar\sigma))\psi\|^2+\mathscr{O}(\hbar^\infty)\|\psi\|^2\,,
\end{split}
\end{equation*}
where we used Lemma \ref{lem.r4tronc} to control the term involving $r^2$ in $p$. For $M$ large enough, there exists $d>0$ such that
\begin{equation*}
\langle\left(\mathrm{Op}^W_\hbar a_{\hbar,N}^{\mathrm{eff}}\right)\mathrm{Op}^W_\hbar(\chi_M(\hbar\sigma))\psi,\mathrm{Op}^W_\hbar(\chi_M(\hbar\sigma))\psi\rangle
\geq d\hbar^2\|\mathrm{Op}^W_\hbar(\chi_M(\hbar\sigma))\psi\|^2+\mathscr{O}(\hbar^\infty)\|\psi\|^2\,,
\end{equation*}
where we have used that, on the support of $\sigma \mapsto \chi_M(\hbar\sigma)$, we have $ \Xi_2(\mu\sigma) \geq M\hbar^2/2$ (see Section \ref{sec.roughmicro} and remember that $\mu = \hbar^2$).

Moreover, we have
\begin{multline*}
|\langle[\mathrm{Op}^W_\hbar a_{\hbar,N}^{\mathrm{eff}},\mathrm{Op}^W_\hbar(\chi_M(\hbar\sigma))]\psi, \mathrm{Op}^W_\hbar(\chi_M(\hbar\sigma))\psi\rangle|\leq C\hbar^3\|\mathrm{Op}^W_\hbar(\underline{\chi_M}(\hbar\sigma))\psi\|^2+\mathscr{O}(\hbar^\infty)\|\psi\|^2\,.
\end{multline*}
With \eqref{eq.microlocchiM}, we deduce that
\[\|\mathrm{Op}^W_\hbar({\chi_M}(\hbar\sigma))\psi\|^2\leq C\hbar\|\mathrm{Op}^W_\hbar(\underline{\chi_M}(\hbar\sigma))\psi\|^2+\mathscr{O}(\hbar^\infty)\|\psi\|^2\,.\]
By an induction argument, we deduce the result.
\end{proof}
The microlocalization established in Lemma \ref{lem.microhsigma} allows to replace $\Xi_2(\hbar^2\sigma)$ by $\hbar\Xi_2(\hbar\sigma)$ in $a_{\hbar,N}^{\mathrm{eff}}$. The rough localization with respect to $r$ (caused by $V_h$) also allows to remove the $\chi_{h,1}$ in the principal symbol, and the rough microlocalisation with respect to $\rho$ to replace $\tilde\rho$ by $\rho$. That is why we consider
\begin{multline*}
	\mathfrak{a}_{\hbar,N}^{\mathrm{eff}}(r,s,\rho,\sigma)=\mathfrak{a}_0(r,s,\rho,\sigma)\\
	+c_0(\rho,\hbar\hat p)\sum_{k=3}^N\frac{(\mu^{\mathrm{dG}})^{(k)}(\xi_0)}{k!}(\rho\sin\phi-\hbar\hat p\cos\phi)^k
	+\hbar^2c_0(\rho,\hbar\hat p)r\chi_{h,1}L_1\sum_{k=1}^N\frac{f^{(k)}(0)}{k!}(\rho\sin\phi-\hbar\hat p\cos\phi)^k\\
+  \hbar^2c_0(\rho,\hbar \hat{p}) r\chi_{\hbar,1} \sum_{1\leq \ell+\ell' \leq N} \hbar^{\ell'} g_{\ell, \ell'}(r,s)  \rho^\ell \hat{p}^{\ell'}
	-\hbar^3q_{3,\Theta_0}^\pm(r,s,\rho,\mu\sigma)-\hbar^4 q_{4,\Theta_0}^\pm(r,s,\rho,\mu\sigma)\,,
\end{multline*}
where
\[\mathfrak{a}_0=\frac{1}{2}(\mu^{\mathrm{dG}})''(\xi_0)(\rho\sin\phi- \hbar\hat p\cos\phi)^2+(\rho\cos\phi+ \hbar\hat p\sin\phi)^2\,,\]
with
\[\hat p=\Xi_2(\hbar\sigma)-\beta(s)\frac{r^2}{2}\,.\]

\begin{proposition}\label{prop.3}
For all $n\geq 1$, we have
\[\lambda_n\left(\mathrm{Op}^W_\hbar a_{\hbar,N}^{\mathrm{eff}}\right)=\lambda_n\left(\mathrm{Op}^W_\hbar \mathfrak{a}_{\hbar,N}^{\mathrm{eff}}\right)+\mathscr{O}(\hbar^\infty)\,.\]
\end{proposition}

\subsection{Changes of semiclassical parameters}
We can now focus on the new effective operator $\mathrm{Op}^W_\hbar \mathfrak{a}_{\hbar,N}^{\mathrm{eff}}$. Firstly note that it can be rewritten as an $1$-pseudo differential operator with respect to $r$, and whose symbol is
\begin{multline*}
\hbar^2\mathfrak{b}_{\hbar,N}^{\mathrm{eff}}(r,s,\rho,\sigma)=\hbar^2\mathfrak{b}_0(r,s,\rho,\sigma)+c_0(\upsilon\rho,\upsilon\hat p)\sum_{k=3}^N\hbar^k\frac{(\mu^{\mathrm{dG}})^{(k)}(\xi_0)}{k!}(\rho\sin\phi-\hat p\cos\phi)^k\\
+\hbar^2c_0(\upsilon\rho,\upsilon\hat p)r\chi_{h,1}L_1\sum_{k=1}^N\hbar^k\frac{f^{(k)}(0)}{k!}(\rho\sin\phi-\hat p\cos\phi)^k\\
-\hbar^3q_{3,\Theta_0}^\pm(r,s,\upsilon\rho,\mu\sigma) \\
 + \hbar^2c_0(\upsilon\rho,\upsilon \hat{p}) r\chi_{\hbar,1} \sum_{1\leq \ell+\ell' \leq N} \hbar^{\ell+\ell'} g_{\ell, \ell'}(r,s)  \rho^\ell \hat{p}^{\ell'}-\hbar^4 q_{4,\Theta_0}^\pm(r,s,\upsilon\rho,\mu\sigma)\,,
\end{multline*}
with $\upsilon=\hbar$ and where
\[\mathfrak{b}_0(r,s,\rho,\sigma)=\frac{1}{2}(\mu^{\mathrm{dG}})''(\xi_0)(\rho\sin\phi- \hat p\cos\phi)^2+(\rho\cos\phi+ \hat p\sin\phi)^2\,.\]
In other words, we have the relation
\begin{equation}\label{eq.c}
\mathrm{Op}^W_\hbar \mathfrak{a}_{\hbar,N}^{\mathrm{eff}}=\hbar^2\mathrm{Op}^W_{\hbar,s,\sigma}\mathfrak{B}^{\mathrm{eff}}_\hbar\,,\qquad
\mathfrak{B}^{\mathrm{eff}}_\hbar(s,\sigma)=\mathrm{Op}^W_{1,r,\rho}\mathfrak{b}_{\hbar,N}^{\mathrm{eff}}\,.
\end{equation}
The notation $\upsilon$ is only introduced to avoid the ambiguity when expanding the operator in powers of $\hbar$.

Secondly, by using the new semiclassical parameter $\varepsilon=\hbar^2$ with respect to $s$, and we write
\begin{equation}\label{eq.d}
\mathrm{Op}^W_{\hbar,s,\sigma}\mathfrak{B}^{\mathrm{eff}}_\hbar=\mathrm{Op}^W_{\varepsilon}\mathfrak{C}^{\mathrm{eff}}_{\varepsilon}\,,
\end{equation}
where the symbol of this operator is
\begin{multline*}
\mathfrak{c}_{\varepsilon,N}^{\mathrm{eff}}(r,s,\rho,\sigma)=\mathfrak{c}_0(r,s,\rho,\sigma)+c_0(\upsilon\rho,\upsilon\check p)\sum_{k=3}^N\varepsilon^{\frac{k-2}{2}}\frac{(\mu^{\mathrm{dG}})^{(k)}(\xi_0)}{k!}(\rho\sin\phi-\check p\cos\phi)^k\\
+c_0(\upsilon\rho,\upsilon\check p)r\chi_{h,1}L_1\sum_{k=1}^N\varepsilon^{\frac{k}{2}}\frac{f^{(k)}(0)}{k!}(\rho\sin\phi-\check p\cos\phi)^k\\
 + \hbar^2c_0(\upsilon \rho,\upsilon \hat{p})r\chi_{\hbar,1}  \sum_{1\leq \ell+\ell' \leq N} \eps^{\frac{\ell+\ell'}{2}} g_{\ell, \ell'}(r,s)  \rho^\ell \hat{p}^{\ell'}
-\varepsilon^{\frac12} q_{3,\Theta_0}^\pm(r,s,\upsilon\rho,\sigma)-\varepsilon q_{4,\Theta_0}^\pm(r,s,\upsilon\rho,\sigma)\,,
\end{multline*}
with
\[\mathfrak{c}_0=\frac{1}{2}(\mu^{\mathrm{dG}})''(\xi_0)(\rho\sin\phi- \check p\cos\phi)^2+(\rho\cos\phi+ \check p\sin\phi)^2\,,\]
where
\[\check p=\Xi_2(\sigma)-\beta(s)\frac{r^2}{2}\,.\]

\subsection{A final Grushin reduction}\label{sec.finalgrushin}
Note that, by completing a square, we can write that
\begin{multline*}
 \mathfrak{c}_0(r,s,\rho,\sigma)=E(s)\left(\rho+\left(1-\frac{( \mu_1^{\mathrm{dG}})''(\xi_0)}{2}\right)\frac{\cos\phi(s)\sin\phi(s)}{E(s)}\hat{p}\right)^2 \\
+\frac{\mu''(\xi_0)}{2E(s)}\left(\Xi_2(\sigma)-\beta(s)\frac{r^2}{2}\right)^2\,,
\end{multline*}
with
\[ E(s) =\frac{\mu''(\xi_0)}{2}\sin^2\phi(s)+\cos^2\phi(s)\,.\]
The domain of $\mathfrak{C}_0^{\mathrm{eff}}=\mathrm{Op}^W_{1,r,\rho}\mathfrak{c_0}$ does not depend on $(s,\sigma)$, and it is given by
\[\mathrm{Dom}(\mathfrak{C}_0^{\mathrm{eff}})=\{\psi\in H^2(\mathbb{R}) : r^4\psi\in L^2(\mathbb{R})\}\,.\]
After, rescaling in $r$, we see that the lowest eigenvalue of $\mathrm{Op}^W_{1,r,\rho}\mathfrak{c_0}$ is
\[ b_0 (s,\sigma):=K(s)\mu^{[2]}_1\left(\frac{\delta_0^{\frac13}
\Xi_2(\sigma)}{E(s)^{\frac23}\beta(s)^{\frac13}}\right)\,,\quad K(s)=\delta_0^{\frac13}\beta(s)^{\frac23}E(s)^{\frac13}\,.\]
The function $b_0$ has a unique minimum, which is non-degenerate and not attained at infinity. We denote by $w_{s,\sigma}$ an associated positive normalized eigenfunction.

We are now interested in the eigenvalues of $\mathrm{Op}^W_{\varepsilon}\mathfrak{C}^{\mathrm{eff}}_{\varepsilon,N}$. Let us fix $\delta>0$ and consider $z\in\mathbb{R}$ such that $|z-\min M_0|\leq\delta$. Then, we consider
\[\mathfrak{M}_{\varepsilon,z}(s,\sigma)=\begin{pmatrix}
\mathfrak{C}_\varepsilon^{\mathrm{eff}}(s,\sigma)-z&\cdot w_{s,\sigma}\\
\langle\cdot,w_{s,\sigma}\rangle&0	
\end{pmatrix}\,,\qquad \mathfrak{M}_{0,z}(s,\sigma)=\begin{pmatrix}
\mathfrak{C}_0^{\mathrm{eff}}(s,\sigma)-z&\cdot w_{s,\sigma}\\
\langle\cdot,w_{s,\sigma}\rangle&0	
\end{pmatrix}\,.\]
When $\delta$ is small enough, $\mathcal{M}_{0,z} : \mathrm{Dom}(\mathrm{Op}_{1,r,\rho}\mathfrak{c}_0)\times\mathbb{C}\to L^2(\mathbb{R})$ is bijective with inverse given by
\[\mathfrak{N}_{0,z}(s,\sigma)=\begin{pmatrix}
(\mathfrak{C}_0^{\mathrm{eff}}(s,\sigma)-z)^{-1}\Pi^\perp&\cdot w_{s,\sigma}\\
\langle\cdot,w_{s,\sigma}\rangle&z-b_0(s,\sigma)
\end{pmatrix}\,,\]
where $\Pi^\perp$ is the orthogonal projection on $(\mathrm{span} w_{s,\sigma})^\perp$.

As we did with the first dimensional reduction, we can find $\mathfrak{N}_{1,z}$ and $\mathfrak{N}_{2,z}$ such that
\[\mathrm{Op}^W_{\varepsilon,s,\sigma}\left(\mathfrak{N}_{0,z}+\varepsilon^{\frac12}\mathfrak{N}_{1,z}+\varepsilon\mathfrak{N}_{2,z}\right)\mathrm{Op}^W_{\varepsilon,s,\sigma}(\mathfrak{M}_{\varepsilon,z})=\mathrm{Id}+\mathscr{O}(\varepsilon^{\frac32})\,.\]
Let us write
\[\left(\mathfrak{N}_{0,z}+\varepsilon^{\frac12}\mathfrak{N}_{1,z}+
\varepsilon\mathfrak{N}_{2,z}\right)_{\pm}=
z- b(s,\sigma)-\varepsilon^{\frac12} b_{1,z}(s,\sigma)-\varepsilon  b_{2,z}(s,\sigma)\,.\]
We are interested in $z$ varying in the following range
\[z=\min b_0+\zeta\varepsilon^{\frac12}+\mathscr{O}(\varepsilon)\,,\]
where $\zeta\in\mathbb{R}$ is determined in Proposition \ref{prop.final} below. Applying again the Grushin method, we get the following proposition.

\begin{proposition}\label{prop.4}
There exist three functions $M_0 = b_0$, $M_1$ and $M_2$ belonging to $S(1)$ such that the following holds. Let $n\geq 1$. We have
\[\lambda_n\left(\mathrm{Op}^W_{\varepsilon,s,\sigma} \mathfrak{C}_\varepsilon^{\mathrm{eff}}\right)=\lambda_n\left(\mathcal{M}_\varepsilon\right)+\mathscr{O}(\varepsilon^{\frac32})\,,\quad \mathcal{M}_\varepsilon=\mathrm{Op}^W_{\varepsilon,s,\sigma}(M_0+\varepsilon^{\frac12}M_1+\varepsilon M_2)\,.\]
\end{proposition}

\subsection{Analysis of the ultimate effective operator}\label{sec.final}
\begin{proposition}\label{prop.final}
There exists $d_1\in\mathbb{R}$ such that the following holds. Let $n\geq 1$. We have
\begin{equation*}
\lambda_n(\mathcal{M}_\varepsilon)=M_0(0,\sigma_0)+\varepsilon^{\frac12}M_1(0,\sigma_0)+(2n-1)\frac{\varepsilon}{2}\sqrt{\det\mathrm{Hess}_{(0,\sigma_0)} M_0}+d_1\varepsilon+o(\varepsilon)\,.
\end{equation*}	
\end{proposition}
\begin{proof}
Since $M_0$ has a non-degenerate minimum at $(0,\sigma_0)$, we can write
\begin{multline*}
M(s,\sigma)=M_0(0,\sigma_0)+\varepsilon^{\frac12}M_1(0,\sigma_0)+Q_0(s,\sigma-\sigma_0)+\varepsilon^{\frac12}\left(s\partial_s M_1(0,\sigma_0)+(\sigma-\sigma_0)\partial_\sigma M_1(0,\sigma_0)\right)\\
+\varepsilon M_2(0,\sigma_0)+\varepsilon^{\frac12}Q_1(s,\sigma-\sigma_0)
+\mathscr{O}(|(s,\sigma-\sigma_0)|^3)+\mathscr{O}(\varepsilon|(s,\sigma-\sigma_0)|)\,,
\end{multline*}
where $Q_0=\frac12\mathrm{Hess}_{(0,\sigma_0)} M_0=\begin{pmatrix}
a&0\\
0&b
\end{pmatrix}>0$ and $Q_1=\frac12\mathrm{Hess}_{(0,\sigma_0)} M_1$.

We have
\begin{equation*}
\begin{split}
&Q_0(s,\sigma-\sigma_0)+\varepsilon^{\frac12}\left(s\partial_s M_1(0,\sigma_0)+(\sigma-\sigma_0)\partial_\sigma M_1(0,\sigma_0)\right)\\
=&as^2+\varepsilon^{\frac12}s\partial_s M_1(0,\sigma_0)+b(\sigma-\sigma_0)^2+\varepsilon^{\frac12}(\sigma-\sigma_0)\partial_\sigma M_1(0,\sigma_0)\\
=&a\left(s+\varepsilon^{\frac12}\frac{\partial_s M_1(0,\sigma_0)}{2a}\right)^2+b\left(\sigma-\sigma_0+\varepsilon^{\frac12}\frac{\partial_\sigma M_1(0,\sigma_0)}{2b}\right)^2\\
&-\varepsilon\left(\frac{[\partial_s M_1(0,\sigma_0)]^2}{4a}+\frac{[\partial_\sigma M_1(0,\sigma_0)]^2}{4b}\right)\,.
\end{split}
\end{equation*}
By using the translation in the phase space
\[\tilde S=(s,\sigma-\sigma_0)+\varepsilon^{\frac12}S_0\,,\qquad S_0=\left(\frac{\partial_s M_1(0,\sigma_0)}{2a},\frac{\partial_\sigma M_1(0,\sigma_0)}{2b}\right)\,,\]
we can write
\begin{multline*}
M(s,\sigma)=M_0(0,\sigma_0)+\varepsilon^{\frac12}M_1(0,\sigma_0)+Q_0(\tilde s,\tilde\sigma)+\varepsilon^{\frac12}Q_1(\tilde s,\tilde\sigma)\\
+\varepsilon\left(-\frac{[\partial_s M_1(0,\sigma_0)]^2}{4a}-\frac{[\partial_\sigma M_1(0,\sigma_0)]^2}{4b}+M_2(0,\sigma_0)\right)+\mathscr{O}(\varepsilon^{\frac32}+|\tilde S|^3+\varepsilon|\tilde S|^2+\varepsilon |\tilde S|)\,.
\end{multline*}
Since we are interested in the eigenvalues $\lambda$ such that $\lambda\leq M_0(0,\sigma_0)+\varepsilon^{\frac12}M_1(0,\sigma_0)+C\varepsilon$. The eigenfunctions associated with such eigenvalues are microlocalized near $(\tilde s,\tilde\sigma)=(0,0)$. By implementing, for instance, a Birkhoff normal form (see \cite{Sj92}), the result follows.

\end{proof}

\subsection{Proof of Theorem \ref{thm.main}}
Our main theorem is a consequence of the succession of propositions and relations:
\begin{enumerate}
\item [---] (Section \ref{sec.2}) Proposition \ref{prop.1},
\item [---] (Section \ref{sec.4}) Proposition \ref{prop.checkLm}, 
%Section \ref{sec.rescalingcheckbreve}, \eqref{eq.a}, \eqref{eq.b}, 
Proposition \ref{prop.micro},
\item [---] (Section \ref{sec.parametrix}) Proposition \ref{prop.specreducaeff}, Proposition \ref{prop.2}, Proposition \ref{prop.3}, \eqref{eq.c}, \eqref{eq.d}, Proposition \ref{prop.4}, and Proposition \ref{prop.final},
\item [---] $\hbar=h^{\frac16}$ and $\varepsilon=\hbar^2$.
\end{enumerate}

\appendix

\section{Proof of Proposition \ref{prop.roughlowerbound}}\label{sec.app}

	Let us use a partition of the unity with balls of radius $h^{\rho}$. We have
\[\mathscr{Q}_{h,\delta}(\psi)\geq\sum_{j}\mathscr{Q}_{h,\delta}(\psi_j)-Ch^{2-2\rho}\|\psi\|^2 \,,\quad \psi_j=\chi_j\psi\,.\]

If $\mathrm{supp}\,\psi_j\cap\partial\Omega=\emptyset$, then, by Lemma \ref{lem.min-interior}, we have
\[\mathscr{Q}_{h,\delta}(\psi_j)\geq h\int_{\Omega_\delta} |\psi_j|^2\dd x\geq h\int_{\Omega_\delta} \mathfrak{s}(\theta(p(x)))|\psi_j|^2\dd x \,,\]
for some constant $C>0$ independent of $j$, and where we used that $\sigma(\theta)\leq 1$.

Consider now the $j$ such that $\mathrm{supp}\,\psi_j\cap\partial\Omega\neq\emptyset$.

On $\Omega_\delta$, we can use locally tubular coordinates $y=(r,s,t)$. By using the considerations and notation of Section \ref{sec.quadraticy}, and by freezing the metrics, we get
\[\mathscr{Q}_{h,\delta}(\psi_j)\geq (1-Ch^{\rho}) \int_{0<t<\delta}|g|^{\frac12}(y_j)\langle G^{-1}(y_j)(-ih\nabla_y-\tilde{\mathbf{A}})\tilde\psi_j,(-ih\nabla_y-\tilde{\mathbf{A}})\tilde\psi_j\rangle\dd y\,.\]
Then, we write the Taylor approximation, on the support of $\tilde\psi_j$,
\[\tilde{\mathbf{A}}(y)=\underbrace{\tilde{\mathbf{A}}(y_j)+d\tilde{\mathbf{A}}(y_j)(y-y_j)}_{=:\tilde{\mathbf{A}}^{\rm lin}_j(y)}+\mathscr{O}(|y-y_j|^2)\,,\]
where $\mathscr{O}$ is uniform with respect to $j$. We have $\nabla\times\tilde{\mathbf{A}}(y)=\nabla\times\tilde{\mathbf{A}}(y_j)+\mathscr{O}(|y-y_j|)$. Thanks to the Young inequality, we get that
\begin{multline*}
(1-Ch^{\rho})^{-1}\mathscr{Q}_{h,\delta}(\psi_j)
\geq (1-\varepsilon)\int_{0<t<\delta}|g|^{\frac12}(y_j)\langle G^{-1}(y_j)(-ih\nabla_y-\tilde{\mathbf{A}}^{\mathrm{lin}}_j)\tilde\psi_j,(-ih\nabla_y-\tilde{\mathbf{A}}^{\mathrm lin}_j)\tilde\psi_j\rangle\dd y
\\	-Ch^{4\rho}\varepsilon^{-1}\int_{0<t<\delta}|\tilde\psi_j|^2|g_j|^{\frac12} \,.
\end{multline*}
Then, we perform a linear change of variables $y=G^{-\frac12}(y_j)z$, and we get
\begin{multline*}\int_{0<t<\delta}|g|^{\frac12}(y_j)\langle G^{-1}(y_j)(-ih\nabla_y-\tilde{\mathbf{A}}^{\mathrm{lin}}_j)\tilde\psi_j,(-ih\nabla_y-\tilde{\mathbf{A}}^{\mathrm lin}_j)\tilde\psi_j\rangle\dd y\\
=\int_{0<t<\delta}\| (-ih\nabla_z-\check{\mathbf{A}}^{\mathrm{lin}}_j)\check\psi_j\|^2\dd z\,,
\end{multline*}
where $\check\psi_j(z)=|g|^{\frac14}(y_j)\tilde\psi_j(G^{-\frac12}(y_j)z)$ and $\check{\mathbf{A}}^{\mathrm{lin}}_j(z)=G^{-\frac12}(y_j)\tilde{\mathbf{A}}^{\mathrm{lin}}_j(G^{-\frac12}(y_j)z)$. Note that the corresponding (constant) magnetic fields are related through
\[\nabla_z\times\check{\mathbf{A}}^{\mathrm{lin}}_j=|g(y_j)|^{-\frac12} G^{\frac12}(y_j)\nabla_y\times \tilde{\mathbf{A}}^{\mathrm{lin}}_j=|g(y_j)|^{-\frac12} G^{\frac12}(y_j)\nabla_y\times\tilde{\mathbf{A}}(y_j)\,.\]
Then, we notice that (see \eqref{eq.coordinatesnewB})
\[\langle |g(y_j)|^{-\frac12} G^{\frac12}(y_j)\nabla_y\times\tilde{\mathbf{A}}(y_j),\mathbf{e}_3\rangle=\langle\mathcal{B}_j,\mathbf{e}_3\rangle=-\mathbf{B}(x_j)\cdot\mathbf{e}_3\,.\]
 In the same way, we get
\[\|\nabla_z\times\check{\mathbf{A}}^{\mathrm{lin}}_j\|^2=\|\mathbf{B}(x_j)\|=1\,.\]
This implies that
\[\int_{0<t<\delta}\| (-ih\nabla_z-\check{\mathbf{A}}^{\mathrm{lin}}_j)\check\psi_j\|^2\dd z\geq h\mathfrak{s}(\theta(p(x_j)))\|\check\psi_j\|^2\,.\]
Thus,
\begin{multline*}\int_{0<t<\delta}|g|^{\frac12}(y_j)\langle G^{-1}(y_j)(-ih\nabla_y-\tilde{\mathbf{A}}^{\mathrm{lin}}_j)\tilde\psi_j,(-ih\nabla_y-\tilde{\mathbf{A}}^{\mathrm lin}_j)\tilde\psi_j\rangle\dd y\\
\geq (1-Ch^{\rho})\int_{0<t<\delta}h\mathfrak{s}(\theta(p(x)))|\psi_j|^2\dd x
\end{multline*}
It follows that
\begin{equation*}
(1-Ch^{\rho})^{-1}\mathscr{Q}_{h,\delta}(\psi_j)
\geq (1-\varepsilon)\int_{0<t<\delta}h\mathfrak{s}(\theta(p(x)))|\psi_j|^2\dd x	-Ch^{4\rho}\varepsilon^{-1}\|\psi_j\|^2\,.
\end{equation*}
We choose $\varepsilon=h^{-\frac12+2\rho}$ and get
\begin{equation*}
\mathscr{Q}_{h,\delta}(\psi_j)
\geq \int_{0<t<\delta}h\mathfrak{s}(\theta(p(x)))|\psi_j|^2\dd x	-C(h^{\frac12+2\rho}+h^{1+\rho})\|\psi_j\|^2\,.
\end{equation*}
This, with the choice $\rho=\frac38$, gives the conclusion.

\bibliographystyle{abbrv}
\bibliography{bibBHR}

\begin{thebibliography}{10}

\bibitem{BHR16}
V.~Bonnaillie-No\"{e}l, F.~H\'{e}rau, and N.~Raymond.
\newblock Magnetic {WKB} constructions.
\newblock {\em Arch. Ration. Mech. Anal.}, 221(2):817--891, 2016.

\bibitem{BHR21}
V.~Bonnaillie-No\"{e}l, F.~H\'{e}rau, and N.~Raymond.
\newblock Purely magnetic tunneling effect in two dimensions.
\newblock {\em Invent. Math.}, 227(2):745--793, 2022.

\bibitem{Bony}
J.-M. Bony.
\newblock Sur l'in\'{e}galit\'{e} de {F}efferman-{P}hong.
\newblock In {\em Seminaire: \'{E}quations aux {D}\'{e}riv\'{e}es {P}artielles,
  1998--1999}, S\'{e}min. \'{E}qu. D\'{e}riv. Partielles, pages Exp. No. III,
  16. \'{E}cole Polytech., Palaiseau, 1999.

\bibitem{DH93}
M.~Dauge and B.~Helffer.
\newblock Eigenvalues variation. {I}. {N}eumann problem for {S}turm-{L}iouville
  operators.
\newblock {\em J. Differential Equations}, 104(2):243--262, 1993.

\bibitem{Fermanian}
C.~Fermanian~Kammerer.
\newblock Op\'{e}rateurs pseudo-diff\'{e}rentiels semi-classiques.
\newblock In {\em Chaos en m\'{e}canique quantique}, pages 53--100. Ed. \'{E}c.
  Polytech., Palaiseau, 2014.

\bibitem{FH06}
S.~Fournais and B.~Helffer.
\newblock Accurate eigenvalue asymptotics for the magnetic {N}eumann
  {L}aplacian.
\newblock {\em Ann. Inst. Fourier (Grenoble)}, 56(1):1--67, 2006.

\bibitem{FH10}
S.~Fournais and B.~Helffer.
\newblock {\em Spectral methods in surface superconductivity}, volume~77 of
  {\em Progress in Nonlinear Differential Equations and their Applications}.
\newblock Birkh\"{a}user Boston, Inc., Boston, MA, 2010.

\bibitem{FP11}
S.~Fournais and M.~Persson.
\newblock Strong diamagnetism for the ball in three dimensions.
\newblock {\em Asymptot. Anal.}, 72(1-2):77--123, 2011.

\bibitem{FS15}
S.~Fournais and M.~P. Sundqvist.
\newblock A uniqueness theorem for higher order anharmonic oscillators.
\newblock {\em J. Spectr. Theory}, 5(2):235--249, 2015.

\bibitem{Helffer10}
B.~Helffer.
\newblock The {M}ontgomery model revisited.
\newblock {\em Colloq. Math.}, 118(2):391--400, 2010.

\bibitem{HM96}
B.~Helffer and A.~Mohamed.
\newblock Semiclassical analysis for the ground state energy of a
  {S}chr\"{o}dinger operator with magnetic wells.
\newblock {\em J. Funct. Anal.}, 138(1):40--81, 1996.

\bibitem{HM01}
B.~Helffer and A.~Morame.
\newblock Magnetic bottles in connection with superconductivity.
\newblock {\em J. Funct. Anal.}, 185(2):604--680, 2001.

\bibitem{HM02}
B.~Helffer and A.~Morame.
\newblock Magnetic bottles for the {N}eumann problem: the case of dimension 3.
\newblock volume 112, pages 71--84. 2002.
\newblock Spectral and inverse spectral theory (Goa, 2000).

\bibitem{HM04}
B.~Helffer and A.~Morame.
\newblock Magnetic bottles for the {N}eumann problem: curvature effects in the
  case of dimension 3 (general case).
\newblock {\em Ann. Sci. \'{E}cole Norm. Sup. (4)}, 37(1):105--170, 2004.

\bibitem{Keraval}
P.~Keraval.
\newblock {\em Formules de Weyl par réduction de dimension. Applications à
  des Laplaciens électro-magnétiques}.
\newblock PhD thesis, Université de Rennes 1, 2018.

\bibitem{M07}
A.~Martinez.
\newblock A general effective {H}amiltonian method.
\newblock {\em Atti Accad. Naz. Lincei Rend. Lincei Mat. Appl.},
  18(3):269--277, 2007.

\bibitem{Ray}
N.~Raymond.
\newblock {\em Bound states of the magnetic {S}chr\"{o}dinger operator},
  volume~27 of {\em EMS Tracts in Mathematics}.
\newblock European Mathematical Society (EMS), Z\"{u}rich, 2017.

\bibitem{Sj92}
J.~Sj\"{o}strand.
\newblock Semi-excited states in nondegenerate potential wells.
\newblock {\em Asymptotic Anal.}, 6(1):29--43, 1992.

\end{thebibliography}

\end{document}